\documentclass[reqno,11pt]{amsart}

\usepackage{comment, amsmath, amssymb, amsgen, amsthm, amscd, xspace, color, epsfig, float, epic}
\usepackage{genyoungtabtikz}
\usepackage{pifont}
\usepackage{multirow}
\usepackage{threeparttable}
\usepackage[colorlinks,linkcolor=blue,urlcolor=cyan,citecolor=blue,pagebackref]{hyperref}

\usepackage{pstricks,pst-node}

\usepackage{fullpage}
\usepackage[textwidth=0.85in, textsize=tiny]{todonotes}
\makeatletter

\newcommand{\Rmnum}[1]{\expandafter\@slowromancap\romannumeral #1@}

\newcommand{\Comp}{\mathbb{C}}

\newcommand{\lest}{\leqslant}

\renewcommand{\leq}{\leqslant}
\renewcommand{\geq}{\geqslant}

\newcommand{\spn}{\mathfrak{sp}(n,\mathbb{C})}

\newcommand{\sob}{\mathfrak{so}(2n+1,\mathbb{C})}

\newcommand{\sod}{\mathfrak{so}(2n,\mathbb{C})}

\newcommand{\lam}{\lambda}

\newcommand{\mat}{\mathrm{M}}
\newcommand{\mnk}{{\mat}_{n,k}}

\newcommand{\hobox}[3]{\draw (0+#1,0-#2) rectangle (1+#1,-1-#2)++(-0.5,+0.5) node {$ #3$};}

\newcommand{\gkd}{\operatorname{GKdim}}

\newcommand{\domscale}{0.5}
\newcommand{\gl}{{\mathfrak{gl}}}
\newcommand{\go}{{\mathfrak{o}}}

\renewcommand{\gg}{{\mathfrak{g}}}

\newcommand{\Kt}{\widetilde{K}}
\newcommand{\aff}{\mathbf{a}}

\newcommand{\GK}{\mathrm{GKdim}}

\newcommand{\od}{\mathrm{od}}
\newcommand{\ev}{\mathrm{ev}}

\newcommand{\pf}{\begin{proof}}
\newcommand{\epf}{\end{proof}}
\newcommand{\eq}{\begin{equation}}
\newcommand{\eeq}{\end{equation}}
\newcommand{\eqn}{\begin{equation*}}
\newcommand{\eeqn}{\end{equation*}}

\newcommand{\frg}{\mathfrak{g}}
\newcommand{\frh}{\mathfrak{h}}

\newcommand{\frl}{\mathfrak{l}}

\newcommand{\frn}{\mathfrak{n}}

\newcommand{\frp}{\mathfrak{p}}
\newcommand{\frq}{\mathfrak{q}}

\newcommand{\fru}{\mathfrak{u}}

\newcommand{\frsl}{\mathfrak{sl}}

\newcommand{\frso}{\mathfrak{so}}

\newcommand{\Pol}{\mathcal{P}}

\newcommand{\car}{\mathfrak{h}}

\newtheorem{Thm}[equation]{Theorem}
\newtheorem{Cor}[equation]{Corollary}
\newtheorem{prop}[equation]{Proposition}
\newtheorem{lem}[equation]{Lemma}

\newtheorem{Rem}[equation]{Remark}

\theoremstyle{definition}
\newtheorem{definition}[equation]{Definition}

\newtheorem{example}[equation]{Example}

\numberwithin{equation}{section}
\begin{document}
	\bibliographystyle{alpha}
	\title[Reducibility]{On the reducibility of scalar generalized Verma modules associated  to two-step nilpotent parabolic subalgebras}
	
	\author{Zhanqiang Bai, Minyan Fang and Zhaojun Wang}
\address[Bai]{School of Mathematical Sciences, Soochow University, Suzhou 215006, P. R. China}
\email{zqbai@suda.edu.cn}

\address[Fang]{School of Mathematical Sciences,  Soochow University, Suzhou 215006, P. R. China}
	\email{FMY15306129217@163.com}
 
		\address[Wang]{School of Mathematical Sciences,  Soochow University, Suzhou 215006, P. R. China}
	\email{wzj202306@163.com}
	
	

	\begin{abstract}
		Let $\mathfrak{g}$ be a simple complex Lie algebra.
 	A generalized Verma module induced from a one-dimensional representation of a parabolic subalgebra of $\mathfrak{g}$ is
called a scalar generalized Verma module of $\mathfrak{g}$.
In this article,  we use Gelfand-Kirillov dimension to determine the
reducibility of scalar generalized Verma modules of $\mathfrak{g}$ associated to a two-step nilpotent parabolic subalgebra of non-maximal type. Such a module exists only when $\mathfrak{g}=\mathfrak{sl}(n,\mathbb{C})$, $\mathfrak{so}(2n,\mathbb{C})$ or $E_6$. We find that the reducible points of these modules can be drawn in a two-dimensional complex plane.

		
	\end{abstract}
\maketitle

	\tableofcontents
\section{Introduction}
The study of highest weight modules of Lie algebras plays an important role in representation theory. The Gelfand–Kirillov dimension is an important invariant to measure the size of an infinite-dimensional module, which  was first introduced by Gelfand and Kirillov \cite{GK}. In this article,
we will use this invariant to study the reducibility problem of some generalized Verma modules. To state our problem and results, we need some notations.
	
 Let $\mathfrak{g}$ be a finite-dimensional complex simple Lie
algebra and $U(\frg)$ be its universal enveloping algebra.
Let $\mathfrak{h}$ be a Cartan subalgebra  and denote by $\Delta$ the root system associated to $(\frg, \frh)$. We choose a positive root system
$\Delta^+\subset\Delta$ and a simple system $\Pi\subset\Delta^+$. Let $\rho$ be the half sum of positive roots in $\Delta$. Let
$\mathfrak{g}=\frn\oplus\mathfrak{h}\oplus
\bar{\mathfrak{n}}$ be the triangular decomposition of $\frg$ with nilradical $\frn$ and its opposite nilradical $\bar\frn$.  We choose a subset $I\subset\Pi$ and  it generates a subsystem
$\Delta_I\subset\Delta$.
Let $\frq_I$ be the standard parabolic subalgebra corresponding to $I$ with Levi decomposition $\frq_I=\frl_I\oplus\fru_I$. So when $I=\emptyset$, we have $\frq_{\emptyset}=\mathfrak{h}\oplus
\mathfrak{n}=\mathfrak{b}$. 

Let $\frq_I=\frl_I\oplus\fru_I$ and $F(\lambda)$ be a finite-dimensional irreducible $\mathfrak{l}_I$-module with the highest weight $\lambda\in\frh^*$. It can  be viewed as a
$\mathfrak{q}_I$-module by letting $\mathfrak{u}_I$ act on it trivially. The {\it generalized Verma module} (GVM) $M_I(\lambda)$ is defined by
\[
M_I(\lambda):=U(\frg)\otimes_{U(\frq_I)}F(\lambda).
\]
In particular, $M(\lambda)=M_{\emptyset}(\lambda)$ is called a {\it Verma module}.
The irreducible quotient of $M(\lambda)$ is denoted by $L(\lambda)$. It is also the irreducible quotient of $M_I(\lambda)$. When $\dim(F(\lambda))=1$, $M_I(\lambda)$ is called a  {\it generalized Verma
	module of scalar type} or {\it scalar generalized Verma module}. When $F(\lambda)$ is  infinite-dimensional, the reducibility problem  also attracts
considerable attention in the literatures (see for example \cite{FM, KM1, KM2, MCb}). 

The reducibility of generalized Verma modules is very important in representation theory. It is closely related to many other research areas (see for example \cite{BXiao, EHW, Ma, MCb}). In general, the standard method  is to apply Jantzen's criterion  \cite{JC}  to determine the reducibility of generalized Verma
modules.
But in general, this criterion  is not very easy to use in application. For example, Enright-Wolf \cite{EW} studied the reducibility problem  for scalar generalized Verma modules associated
to  many maximal  two-step nilpotent parabolic subalgebras. Kubo \cite{Ku} found some simplified reducibility criteria to solve the reducibility problem for scalar generalized Verma modules associated to exceptional simple Lie algebras and some maximal parabolic subalgebras. By using Kubo's simplification criteria, He \cite{He} gave the reducibility for all scalar generalized Verma modules of Hermitian symmetric pairs. Then He-Kubo-Zierau \cite{HKZ} found the  reducibility for all scalar generalized Verma modules associated to  maximal parabolic subalgebras. Bai-Xiao \cite{BXiao} found the reducibility for all  generalized Verma modules of Hermitian symmetric pairs. Xiao-Zhang \cite{XZ} found some refinements of Jantzen's simplicity criteria and used them to study the reducibility of generalized Verma modules.
All of the above results are based on Jantzen's criterion  \cite{JC}.


Recently Bai-Xiao \cite{BX} found that a scalar generalized Verma module $M_I(\lambda)$ is reducible if and only if the Gelfand-Kirillov dimension of its irreducible quotient $L(\lambda)$ is strictly less than $\dim(\fru_I)$. By using this result, they also obtained the same result as  \cite{He}. 

In general,  it is not easy to compute the
Gelfand-Kirillov dimensions of explicit modules. Bai-Xie \cite{BX} and Bai-Xiao-Xie \cite{BXX} found some simple algorithms using Robinson-Schensted insertion algorithm to compute the Gelfand-Kirillov dimensions of highest weight modules for all classical Lie algebras. By using  these algorithms, Bai-Jiang \cite{BJ} gave a new proof for
the reducibility of the scalar generalized Verma modules associated to  maximal parabolic subalgebras.

In this article, we will use Gelfand-Kirillov dimension to study the reducibility problem  for scalar generalized Verma modules associated
to two-step nilpotent parabolic subalgebras. The definition of $k$-step nilpotent Lie algebras is given in the following:
\begin{definition}
Let $\fru$ be any nonzero Lie algebra. Put $\fru_0=\fru$, $\fru_1=[ \fru,\fru ]$, and $\fru_k=[ \fru,\fru_{k-1} ]$ 
for $k \in \mathbb{Z}_{> 0}$, we call $\fru_k$ the $k$-th step of $\fru$ for $k \in \mathbb{Z}_{\ge 0}$.

The Lie algebra $\fru$ is called nilpotent if $\fru_k=0$ for some $k$, and it is called $k$-step nilpotent if $\fru_{k-1}\neq 0$ and $\fru_k=0$.

If the niradical $\fru$ of a parabolic subalgebra $\frq=\frl\oplus\fru$ is $k$-step nilpotent then we say that $\frq$ is a {\it $k$-step nilpotent parabolic subalgebra}.

\end{definition}
 
From Kubo \cite[Proposition 3.1.4]{Ku}, we find that most results on the reducibility of scalar generalized Verma modules are associated to maximal two-step nilpotent parabolic subalgebras, for example, see \cite{BX,BXiao, He,HKZ,EW}. When the parabolic subalgebra $\frq$ of a simple Lie algebra $\frg$ is  two-step  nilpotent and not a maximal parabolic subalgebra, this case can only happen in type $A$, $D$ and $E_6$.

Suppose $\frq_I \subset \mathfrak{sl}(n, \mathbb{C})$  is two-step  nilpotent and not a maximal parabolic subalgebra. Then $\Pi \backslash I$
contains  two simple roots $\alpha_p$ and $\alpha_q$ by \cite[Proposition 3.1.4]{Ku}. Let $\xi_i$ be the fundamental weight associated to the $i$-th simple root $\alpha_i\in \Pi$.  If $M_I(\lambda)$ is of scalar type, we know that
$\lambda=z_1\xi_p+z_2\xi_q$ for some $z_1,z_2\in\mathbb{C}$. Our criterion is
given as follows.

\begin{Thm}\label{thm-a}
Suppose $ \mathfrak{g}=\mathfrak{sl}(n,\mathbb{C}) $. Let $M_{I}(\lambda)$ be a scalar generalized Verma module of $ \mathfrak{g}$ with highest weight $\lambda=z_1 \xi_p+z_2\xi_q$, where $\xi_p$ is the fundamental weight corresponding to $\alpha_{p}=e_p-e_{p+1}$, $\xi_q$ is the fundamental weight corresponding to $\alpha_{q}=e_q-e_{q+1}$. We write $q-p=k$, $(k\ge 1)$, $m=\min\{p,n-q\}$, $h=\max\{p,n-q\}$. Then $M_{I}(\lambda)$ is reducible if and only if one of the following holds: 
\begin{enumerate}
	\item When $z=z_1=z_2$, we have the following.
 
		\begin{enumerate}
  \item $z\notin \mathbb{Z}$,    
          $m\ge 1$ and $z\in (\frac{1}{2}+\mathbb{Z})\bigcap (-\frac{k+m}{2},+\infty)$.
  \item $z\in \mathbb{Z}$. We have
  \begin{enumerate}
      
			\item If $m\ge k-1$, we have
	\[	z\in\left\{
		\begin{array}{ll}
			-\lceil {\frac {m}{2}}\rceil  -\lfloor {\frac{k-1}{2}}\rfloor + \mathbb{Z}_{\ge 0},&\textnormal{if $k$ is even,}\\	     	  	
			-\lfloor {\frac {m}{2}}\rfloor  -\lfloor {\frac{k-1}{2}}\rfloor + \mathbb{Z}_{\ge 0}, &\textnormal{if $k$ is odd.}	
		\end{array}	
		\right.
		\]
			
			
			\item If $0< m< k-1$, we have 
			\[	z\in\left\{
		\begin{array}{ll}
			-\max \{ \lceil  {\frac{k+m}{2}}\rceil, h \}+1 + \mathbb{Z}_{\ge 0},&\textnormal{if $h<k$, }\\	     	  	
			-k+1 + \mathbb{Z}_{\ge 0}, &\textnormal{if $h\ge k$.}\\	
		\end{array}	
		\right.
		\]
			\item  If $m=0$, we have
			\begin{center}
			    $z \in -\min\{h,k\}+1 + \mathbb{Z}_{\ge 0}$.
			\end{center}
		    
		\end{enumerate}	
		\end{enumerate}
 \item When $z_1\neq z_2$, we have the following.

\begin{enumerate}
		\item If $n-q=0$, we have  $$z_1\in 1-\min\{p, n-p\}+ \mathbb{Z}_{\ge 0}, z_2\in \mathbb{C}.$$
		
		\item If $n-q\ge 1$, we have the following.
		\begin{enumerate}
			\item  $z_2\in 1-\min\{q-p, n-q\}+ \mathbb{Z}_{\ge 0}$, $z_1\in \mathbb{C}$, or
			\item  $z_1\in 1-\min\{p, q-p\}+ \mathbb{Z}_{\ge 0}$, $z_2\in \mathbb{C}$, or
			\item  $(z_1,z_2)\in \mathbb{C}\times \mathbb{C}$, and
				 $z_1+z_2 \in -q+p-\min\{p,n-q\}+\mathbb{Z}_{>0}$.
	\end{enumerate}

	\end{enumerate}

\end{enumerate}
\end{Thm}

Suppose $\frq_I \subset \mathfrak{so}(2n, \mathbb{C})$  is two-step  nilpotent and not a maximal parabolic subalgebra. Then $\Pi \backslash I$
contains  two simple roots $\alpha_p$ and $\alpha_q$ ($p,q\in \{1,n-1,n\}$) by \cite[Proposition 3.1.4]{Ku}. Let $\xi_i$ be the fundamental weight associated to the $i$-th simple root $\alpha_i\in \Pi$.  If $M_I(\lambda)$ is of scalar type, we know that
$\lambda=z_1\xi_p+z_2\xi_q$ for some $z_1,z_2\in\mathbb{C}$. Our criterion is
given as follows.

\begin{Thm}\label{son-thm} 
		Let $ \mathfrak{g}=\mathfrak{so}(2n,\mathbb{C}) $. $M_{I}(\lambda)$ is a scalar generalized Verma module with highest weight $\lambda=z_1 \xi_p+z_2\xi_q$ $(p,q\in \{1,n-1,n\})$, where $\xi_p$ is the fundamental weight corresponding to $\alpha_{p}=e_p-e_{p+1}$, $\xi_q$ is the fundamental weight corresponding to $\alpha_{q}=e_q-e_{q+1}$. Then $M_{I}(\lambda)$ is reducible if and only if one of the following holds:
   \begin{enumerate}
\item If $p=1$ and $q=n-1$ or $n$, we have
\begin{enumerate}
    \item  $z_1\in \mathbb{Z}_{\ge 0}, z_2\in \mathbb{C}$, or
    \item $ z_1+z_2\in -n+2+\mathbb{Z}_{\ge 0} $, or
    \item $ z_1\in \mathbb{C} \textnormal{~and~}
        z_2\in \begin{cases}
         -n+3+\mathbb{Z}_{\ge 0}, &\textnormal{if}~n~\textnormal{is odd},\\
         -n+4+\mathbb{Z}_{\ge 0}, &\textnormal{if}~n~ \textnormal{is even.}
     \end{cases}       
    $
\end{enumerate}
 
\item If $p=n-1$ and $q=n$, we have
 \begin{enumerate}
     \item  $z_1\in \mathbb{Z}_{\ge 0}, z_2\in \mathbb{C}$, or
     \item $ z_2\in \mathbb{Z}_{\ge 0}, z_1\in \mathbb{C}$, or
     \item $z_1+z_2\in \begin{cases}
          -n+1+\mathbb{Z}_{\ge 0}, &\textnormal{if}~n~\textnormal{is odd},\\
           -n+2+\mathbb{Z}_{\ge 0}, &\textnormal{if}~n~\textnormal{is even}.
       \end{cases} $
 \end{enumerate}

 \end{enumerate}
\end{Thm}

Suppose $\frq_I \subset E_6$  is two-step  nilpotent and not a maximal parabolic subalgebra. Then $\Pi \backslash I$
contains  two simple roots $\alpha_1$ and $\alpha_6$ by \cite[Proposition 3.1.4]{Ku}. Let $\xi_i$ be the fundamental weight associated to the $i$-th simple root $\alpha_i\in \Pi$.  If $M_I(\lambda)$ is of scalar type, we know that
$\lambda=z_1\xi_1+z_2\xi_6$ for some $z_1,z_2\in\mathbb{C}$. Our criterion is
given as follows.

\begin{Thm}\label{thm-e}
Suppose $ \mathfrak{g}=E_6$. Let $M_{I}(\lambda)$ be a scalar generalized Verma module of $ \mathfrak{g}$ with highest weight $\lambda=z_1 \xi_1+z_2\xi_6$.  Then $M_{I}(\lambda)$ is reducible if and only if one of the following holds:
\begin{enumerate}
       \item $z_1\in -3+\mathbb{Z}_{\ge 0}$ and $z_2\in \mathbb{C}$.
       \item $z_2\in -3+\mathbb{Z}_{\ge 0}$ and $z_1\in \mathbb{C}$.
       \item $z_1+z_2\in -7+\mathbb{Z}_{\ge 0}$.
   \end{enumerate}
\end{Thm}

The paper is organized as follows. The necessary preliminaries for two-step nilpotent parabolic subalgebras and Gelfand-Kirillov dimension are given in Sections $2$ and $3$. In Section $4$, we give the reducibility of scalar generalized Verma modules for type $A_{n-1}$. In Section $5$, we give the reducibility of scalar generalized Verma modules for type  $D_n$. In Section $6$, we give the reducibility of scalar generalized Verma modules for type  $E_6$.

\section{Two-step nilpotent parabolic subalgebras}\label{nilpotent}
Let $\mathfrak{g}$ be a finite-dimensional complex semisimple Lie
algebra and let $\mathfrak{b}=\mathfrak{h}\oplus\mathop\oplus\limits_{\alpha\in\Delta^+}\frg_\alpha$
be a fixed Borel subalgebra of $\frg$. There is a simple characterization of $k$-step nilpotent  parabolic subalgebras in the following.

\begin{prop}[{\cite[Proposition 3.1.4]{Ku}}]
    Let $\frg$ be a complex simple Lie algebra with highest root $\gamma$, and $\frq_I=\frl_I\oplus\fru_I$ be the parabolic subalgebra of $\frg$ that is parameterized by $I$ with $\Pi-I=\{ \alpha_{i_1},\cdots,\alpha_{i_r} \} \subset \Pi$. Then $\fru$ is k-step nilpotent if and only if $k=m_{i_1}+m_{i_2}+\cdots+m_{i_r}$, where $m_{i_j}$ 
    are the 
   multiplicities   
   of $\alpha_{i_j}$ in  the highest root $\gamma$.
\end{prop}

From \cite{Bur}, we can find the explicit expression of the highest root $\gamma$. For type $B$ and $C$, all two-step nilpotent parabolic subalgebras are maximal ones. 
For  type $A$,  a two-step nilpotent parabolic subalgebra $\frq=\frl\oplus\fru$  corresponds to the subsets  $I=\Pi\setminus\{\alpha_i,\alpha_j\}$
( $\alpha_i,\alpha_j\in\Pi$). For  type $D$,  a two-step nilpotent parabolic subalgebra $\frq$ is maximal or it corresponds to the subsets  $I=\Pi\setminus\{\alpha_i,\alpha_j\}$
( $i,j\in\{1,n-1,n\}$).





In this section, we give the explicit values of $\dim(\fru)$. The results are as follows.
\begin{enumerate}
\item If $\mathfrak{g}$ is of type $A_{n-1}$, choose \begin{center}
		$\Pi=\{ e_1-e_2,e_2-e_3,\cdots,e_{n-1}-e_n\}$.
	\end{center}
 
 For  $I=\Pi\setminus\{\alpha_p,\alpha_{p+1}\}$, we will have
\begin{align*}
	\Delta^+(\frl)=\{e_i-e_j|1\le i< j\le p ~ \text{or}~ p+2\le i< j\le n\}.
\end{align*}
So	$\dim(\fru)=(p+1)(n-p)-1=p(n-p)+(n-p-1)$.	

For  $I=\Pi\setminus\{\alpha_p,\alpha_q\}, 
q-p>1$, we will have
	\begin{align*}
	\Delta^+(\frl)=\{e_i-e_j|1\le i<j\le p~\text{or}~ p+1\le i<j\le q~\text{or}~ q+1\le i<j\le n\}.
	\end{align*}
So	$\dim(\fru)=q(n-q)+p(q-p)$.	

\item If $\mathfrak{g}$ is of type $D_{n}$, choose 
\begin{center}
		$\Pi=\{ e_1-e_2,e_2-e_3,\cdots,e_{n-1}-e_n,e_{n-1}+e_n\}$.
	\end{center}

For  $I=\Pi\setminus\{\alpha_p,\alpha_{q}\}$, we will have 
$$\Delta^+(\frl) =\begin{cases}
     \{e_i-e_j|2\le i< j\le n-1\}\cup\{e_i+e_n|2\le i\le n-1\},  &\text{if}~p=1,q=n-1,\\
     \{e_i-e_j|2\le i< j\le n\}, &\text{if}~p=1,q=n,\\
     \{e_i-e_j|1\le i< j\le n-1\},  &\text{if}~p=n-1,q=n.
 \end{cases}$$
So $\dim(\fru)=\frac{1}{2}(n^2+n-2)$.

\item If $\mathfrak{g}$ is of type $E_{6}$, choose \begin{center}
		$\Pi=\{ \alpha_1=\frac{1}{2}(e_1-e_2-\cdots e_7+e_8),\alpha_2=e_1+e_2,\alpha_3=e_2-e_1,\cdots,\alpha_6=e_5-e_4\}$.
	\end{center}
 
 For  $I=\Pi\setminus\{\alpha_1,\alpha_{6}\}$, we will have
\begin{align*}
	\Delta^+(\frl)=\{e_j\pm e_i|1\le i< j\le 4\}.
\end{align*}
So	$\dim(\fru)=36-12=24$.

\end{enumerate}

\section{Gelfand-Kirillov dimension}
In this section, we recall the algorithms of Gelfand-Kirillov dimension of highest wight modules found by Bai-Xie \cite{BX} and Bai-Xiao-Xie \cite{BXX}.

Let $M$ be a finitely generated $U(\mathfrak{g})$-module. Fix a finite dimensional generating subspace $M_0$ of $M$. Let $U_{n}(\mathfrak{g})$ be the standard filtration of $U(\mathfrak{g})$. Set $M_n=U_n(\mathfrak{g})\cdot M_0$ and
\(
\text{gr} (M)=\bigoplus\limits_{n=0}^{\infty} \text{gr}_n M,
\)
where $\text{gr}_n M=M_n/{M_{n-1}}$. Thus $\text{gr}(M)$ is a graded module of $\text{gr}(U(\mathfrak{g}))\simeq S(\mathfrak{g})$.

\begin{definition} The \textit{Gelfand-Kirillov dimension} of $M$  is defined by
	\begin{equation*}
	\operatorname{GKdim} M = \varlimsup\limits_{n\rightarrow \infty}\frac{\log\dim( U_n(\mathfrak{g})M_{0} )}{\log n}.
	\end{equation*}
\end{definition}

It is easy to see that the above definition is independent of the choice of $M_0$.

The following two Lemmas are very useful in our proof.
\begin{lem}[{\cite[Theorem 1.1]{BX}}]\label{reducible}
	A scalar generalized Verma module $M_I(\lambda)$ is irreducible if and only if $\emph{GKdim}\:L(\lambda)=\dim(\mathfrak{u})$.
\end{lem}

\begin{lem}{{\cite[Lemma 4.4]{BX}}}\label{GKdown}
	Let $\mathfrak{g}$ be a finite-dimensional complex simple Lie
algebra. Let $\{\xi_1,..,\xi_n\} \in \mathfrak{h}^*$ be the fundamental weights.  For any $\lambda\in\mathfrak{h}^*$, we have
	\[
	\GK \:L(\lambda+\xi_i)\leq\GK \: L(\lambda).
	\]
	In particular, if $M_I(\lambda)$ is reducible, then $M_I(\lambda+\xi_i)$ is also reducible.
\end{lem}

 Let $(-, -): \mathfrak{h} \times \mathfrak{h}^* \to \mathbb{C}$ be the canonical pairing. For $\mu\in\mathfrak{h}^*$, define 
\begin{equation*}
\Delta_{[\mu]}:=\{\alpha\in\Delta\mid(\mu, \alpha^\vee)\in\mathbb{Z}\},
\end{equation*}
where $ \alpha^\vee$ is the coroot associated with the root $\alpha \in \Delta$. Set 
\[
W_{[\mu]}:=\{w\in W\mid w\mu-\mu\in \mathbb{Z}\Delta\}.
\]
Then $\Delta_{[\mu]}$ is a root system with Weyl group $W_{[\mu]}$. 
 Let $\Pi_{[\mu]}$ be the simple system of $\Delta_{[\mu]}$. Set $J=\{\alpha\in\Pi_{[\mu]}\mid(\mu, \alpha^\vee)=0\}$. Denote by $W_J$ the Weyl group generated by reflections $s_\alpha$ with $\alpha\in J$. Let $\ell_{[\mu]}$ be the length function on $W_{[\mu]}$. Thus $\ell_{[\mu]}=\ell$ when $\mu$ is integral. Put
\begin{equation*}\label{ceq1}
	W_{[\mu]}^J:=\{w\in W_{[\mu]}\mid \ell_{[\mu]}(ws_\alpha)=\ell_{[\mu]}(w)+1\ \mbox{for all}\ \alpha\in J\}.
\end{equation*}
Thus $W_{[\mu]}^J$ consists of the shortest representatives of the cosets $wW_J$ with $ w\in W_{[\mu]} $. When $\mu$ is integral, we simply write $W^J:=W_{[\mu]}^J$ .

A weight $ \mu\in\mathfrak{h}^* $ is called \textit{anti-dominant} if $ (\mu+\rho, \alpha^\vee) \notin\mathbb{Z}_{>0}$ for all $ \alpha\in\Delta^+ $. For any $\lambda\in\mathfrak{h}^*$, there exists a unique anti-dominant weight $\mu\in\mathfrak{h}^*$ and a unique $w_{\lambda}\in W_{[\mu]}^J$ such that $\lambda=w_{\lambda}\cdot\mu:=w_{\lambda}(\mu+\rho)-\rho$. 

\begin{prop}[{\cite[Prop. 3.5]{hum}}]\label{anti}
    Let $\lambda\in \mathfrak{h}^*$, with corresponding root system $\Delta_{[\lambda]}$ and Weyl group $W_{[\lambda]}$. Let $\Pi_{[\lambda]}$ be the simple system of $\Delta_{[\lambda]} \cap \Delta^+$  in $\Delta_{[\lambda]}$.  Then $\lambda$ is antidominant if and only if one of the
following three equivalent conditions holds:
\begin{enumerate}
    \item $(\lambda+\rho, \alpha^\vee)\leq 0$ for all $\alpha \in \Pi_{[\lambda]}$;
    \item $\lambda\leq s_{\alpha}\cdot\lambda$ for all $\alpha \in \Pi_{[\lambda]}$;
    \item  $\lambda\leq w\cdot\lambda$ for all $w \in W_{[\lambda]}$.
\end{enumerate}
Therefore, there is a unique antidominant weight in the orbit $W_{[\lambda]}\cdot\lambda$.
\end{prop}

\begin{prop}[{\cite[Prop. 3.8]{BX}}]\label{pr:main1}
	Let $ \lam\in\mathfrak{h}^* $. Suppose that $\lambda=w_{\lambda}\cdot\mu$, where $\mu$ is anti-dominant and $w_{\lambda}\in W_{[\mu]}^J$. Then
\begin{equation*}
	\gkd L(\lambda)=|\Delta^+|-{\bf a}_{[\lambda]}(w_{\lambda}),
\end{equation*}
	where ${\bf a}_{[\lambda]}$ is the ${\bf a}$-function on $W_{[\lambda]}=W_{[\mu]}$.
\end{prop}

\subsection{Hecke algebra and \texorpdfstring{$\aff$}{}-functions}\label{sec:cell}

Recall that the Weyl group $ W  $ of $ \mathfrak{g} $ is a Coxeter group generated by $S:=\{s_\alpha \mid\alpha\in\Pi \}$. Let $\ell(-)$ be the length function on $W$ with respect to $S$. Then we have a Hecke algebra $ \mathcal{H} $ over $ \mathcal{A} :=\mathbb{Z}[q,q^{-1}]$, which is generated by $ T_w $, $ w\in W $ with relations \[
T_wT_{w'}=T_{ww'} \text{ if }\ell(ww')=\ell(w)+\ell(w'),
\]\[
\text{and }(T_s+q^{-1})(T_s-q)=0 \text{ for any }s\in S.
\]
The Kazhdan--Lusztig basis $C_w, w\in W$ of $ \mathcal{H} $ are characterized as the unique elements $ C_w \in \mathcal{H}$ such that
\[
\overline{C_w}=C_w,\qquad C_w\equiv T_w \mod{\mathcal{H}_{<0}}
\]where $ \bar{\,} :\mathcal{H}\rightarrow\mathcal{H}$ is the bar involution such that $ \bar{q}=q^{-1} $, $ \overline{T_w} =T_{w^{-1}}^{-1}$, and $ \mathcal{H}_{<0}=\bigoplus_{w\in W}\mathcal{A}_{<0}T_w $, $ \mathcal{A}_{<0}=q^{-1}\mathbb{Z}[q^{-1}] $.

If $ C_y $ occurs in the expansion of $ hC_w $ (resp. $C_wh$) with respect to the KL-basis for some $ h\in\mathcal{H} $, then we write $ y\leftarrow_L w $ (resp. $ y\leftarrow_R w $). Extend $ \leftarrow_L $ (resp. $ \leftarrow_R $) to a preorder $ \lest_L $ (resp. $\leq_R$) on $ W $. For $x, w\in W$, write $x \lest_{LR} w$ if there exist $x=w_1, \dots, w_n=w$ such that for every $1\lest i<n$ we have either $w_i\lest_L w_{i+1}$ or $w_i\lest_R w_{i+1}$. Let $\sim_{L}$, $\sim_{R}$, $\sim_{LR}$ be the equivalence relations associated with $\lest_L$, $\lest_R$, $\lest_{LR}$ (for example, $x\sim_{L}w$ if and only if $x\lest_L w$ and $w\lest_Lx$). The equivalence classes on $W$ for $\sim_L$, $\sim_R$, $\sim_{LR}$ are called \textit{left cells}, \textit{right cells} and \textit{two-sided cells} respectively. 



We have $ C_xC_y=\sum_{z\in W} h_{x,y,z}C_z $ with $ h_{x,y,x}\in\mathcal{A} $. Then Lusztig's \textit{$\aff$-function} $ \mathbf{a}:W\rightarrow\mathbb{N} $ is defined by\[
\aff(z)=\max\{\deg h_{x,y,z}\mid x,y\in W \} \text{ for } z\in W.
\]

The following lemma is an easy consequence of Lusztig's results \cite[14.2]{Lus03}.

\begin{lem}\label{alem1}
Let $x, w\in W$. Then
	\begin{itemize}\item [(1)]$ \aff(w)=\aff(w^{-1}) $.
		\item [(2)] If $x\lest_{LR} w$, then $\aff(x)\geq\aff(w)$. Hence $\aff(x)=\aff(w)$ whenever $x\sim_{LR} w$.
		\item [(3)] If $ w_I $ is the longest element of the parabolic subgroup of $ W $ generated by a subset $ I\subset S $, the $ \aff(w_I)$ is equal to  the length $\ell(w_I) $ of $ w_I $.
		\item [(4)] If $ W $ is a direct product of  Coxeter subgroups $ W_1 $ and $ W_2 $, then\[
		\aff(w)=\aff(w_1)+\aff(w_2)
		\]for  $ w=(w_1,w_2) \in W_1\times W_2=W$.
	\end{itemize}
\end{lem}

\subsection{GKdim of highest weight modules for classical types}

	For  a totally ordered set $ \Gamma $, we  denote by $ \mathrm{Seq}_n (\Gamma)$ the set of sequences $ x=(x_1,x_2,\cdots, x_n) $   of length $ n $ with $ x_i\in\Gamma $. In our paper, we usually take $\Gamma$ to be $\mathbb{Z}$ or a coset of $\mathbb{Z}$ in $\mathbb{C}$.
	Let $p(x)=(p_1,p_2,\cdots)$ be the shape of the Young tableau $P(x)$ obtained by applying Robinson-Schensted algorithm (\cite{BX1}) to $x\in \mathrm{Seq}_n (\Gamma)$, where $p_i$ is the number of boxes in the $i$-th row of the Young tableau $P(x)$.

	For a Young diagram $P$, use $ (k,l) $ to denote the box in the $ k $-th row and the $ l $-th column.
	We say the box $ (k,l) $ is \textit{even} (resp. \textit{odd}) if $ k+l $ is even (resp. odd). Let $ p_i ^{\ev}$ (resp. $ p_i^{\od} $) be the numbers of even (resp. odd) boxes in the $ i $-th row of the Young diagram $ P $.
	One can easily check that
	\begin{equation}\label{eq:ev-od}
	p_i^{\ev}=\begin{cases}
	\left\lceil \frac{p_i}{2} \right\rceil,&\text{ if } i \text{ is odd},\\
	\left\lfloor \frac{p_i}{2} \right\rfloor,&\text{ if } i \text{ is even},
	\end{cases}
	\quad p_i^{\od}=\begin{cases}
	\left\lfloor \frac{p_i}{2} \right\rfloor,&\text{ if } i \text{ is odd},\\
	\left\lceil \frac{p_i}{2} \right\rceil,&\text{ if } i \text{ is even}.
	\end{cases}
	\end{equation}
	Here for $ a\in \mathbb{R} $, $ \lfloor a \rfloor $ is the largest integer $ n $ such that $ n\leq a $, and $ \lceil a \rceil$ is the smallest integer $n$ such that $ n\geq a $. For convenience, we set
	\begin{equation*}
	p^{\ev}=(p_1^{\ev},p_2^{\ev},\cdots)\quad\mbox{and}\quad p^{\od}=(p_1^{\od},p_2^{\od},\cdots).
	\end{equation*}
	
	For $ x=(x_1,x_2,\cdots,x_n)\in \mathrm{Seq}_n (\Gamma) $, set
	\begin{equation*}
	\begin{aligned}
	{x}^-=&(x_1,x_2,\cdots,x_{n-1}, x_n,-x_n,-x_{n-1},\cdots,-x_2,-x_1).
	\end{aligned}
	\end{equation*}
	
	\begin{prop}[{\cite[Theorem 1.5]{BXX}}]\label{integral}
		Let $\lambda+\rho=(\lambda_1, \lambda_2, \cdots, \lambda_n)\in \mathfrak{h}^*$ be  integral. Then
		
		\[	{\rm GKdim}\:L(\lambda)=\left\{
		\begin{array}{ll}
		\dfrac{n(n-1)}{2}-\sum\limits_{i\ge 1}(i-1)p(\lambda+\rho)_i, &\textnormal{if}\:\Delta =A_{n-1},\\
		n^2-n-\sum\limits_{i\ge 1}(i-1)p((\lambda+\rho)^-)_i^{\ev}, &\textnormal{if}\:\Delta =D_{n}.
		\end{array}	
		\right.
		\]

	\end{prop}
	
	When $\lambda$ is non-integral,  we need some more notations before we give the algorithm of GK dimension.
	
	We define three functions $ F_A $, $ F_D $ as
	\begin{align*}
	F_A(x)&=\sum_{k\geq 1} (k-1) p_k,\\
	F_D(x)&=\sum_{k\geq 1} (k-1) p_k^{\ev},
	\end{align*}
	where $ p(x)=(p_1,p_2,\cdots) $ is the shape of the Young tableau $P(x)$.

	\begin{definition}
		Fix $ \lambda+\rho=(\lambda_1,\cdots,\lambda_n) \in \mathfrak{h}^*$.
		
		For $ \mathfrak g= \mathfrak{sl}(n, \mathbb{C})$,  we define $[\lambda]$ to be the set of  maximal subsequence $ x $ of  $ \lambda+\rho $ such that any two entries of $ x $ has an integral difference. For any $x\in [\lambda]$, we can get a Young tableua $P(x)$ by using the Robinson-Schensted
algorithm. We denote $S(\lambda)=\{P(x)\mid x\in [\lambda]\}$, which is a set of Young tableaux associated to $\lambda$.
		
		For $\mathfrak{g}= \sod $, we define $[\lambda] $ to be the set of  maximal subsequence $ x $ of  $ \lambda+\rho $ such that any two entries of $ x $ have an integral  difference or sum. In this case, we set $ [\lambda]_1 $ (resp. $ [\lambda]_2 $)  to be the subset of $ [\lambda] $ consisting of sequences with  all entries belonging to $ \mathbb{Z} $ (resp. $ \frac12+\mathbb{Z} $).
		Since there is at most one element in $[\lambda]_1 $ and $[\lambda]_2 $, we denote them by  $(\lambda+\rho)_{(0)}$ and $(\lambda+\rho)_{(\frac{1}{2})}$.
		We set $[\lambda]_{1,2}=[\lambda]_1\cup [\lambda]_2, \quad [\lambda]_3=[\lambda]\setminus[\lambda]_{1,2}$.
		
	\end{definition}

	\begin{definition}
		Let $\lambda\in\mathfrak{h}^{\ast}$ and $ x=(\lambda_{i_1}, \lambda_{i_2},\cdots \lambda_{i_r})\in[\lambda]_3 $. Let $  y=(\lambda_{j_1}, \lambda_{j_2},\cdots, \lambda_{j_p}) $ be the maximal subsequence of $ x $ such that $ j_1=i_1 $ and the difference of any two entries of $ y$ is an integer. Let $ z= (\lambda_{k_1}, \lambda_{k_2},\cdots, \lambda_{k_q}) $ be the subsequence obtained by deleting $ y$ from $ x $, which is possible empty.
		Define
		$$  \tilde{x}=(\lambda_{j_1}, \lambda_{j_2},\cdots, \lambda_{j_p}, -\lambda_{k_q}, -\lambda_{k_{q-1}},\cdots,-\lambda_{k_1}).  $$
	\end{definition}

	\begin{prop}[{\cite[Theorem 4.6]{BX1}} and {\cite[Theorem 5.7]{BXX}} ]\label{GKdim}
		The GK dimension of  $ L(\lambda) $  can be computed as follows.
		\begin{enumerate}
			\item If $  \mathfrak{g}= \mathfrak{sl}(n,\mathbb{C})$,
			\[
			\gkd L(\lambda)=\frac{n(n-1)}{2}-\sum _{x\in [\lambda]} F_A(x).
			\]

			\item  If $  \mathfrak{g} = \mathfrak{so}(2n,\mathbb{C}) $,
			\[
			\gkd L(\lambda)=n^2-n-F_D((\lambda+\rho)_{(0)}^-)-F_D((\lambda+\rho)_{(\frac{1}{2})}^-)-\sum _{x\in [\lambda]_3} F_A(\tilde{x}).
			\]
		\end{enumerate}
	\end{prop}

\begin{Rem}From Proposition \ref{GKdim}, we know that the GK dimension of a highest weight module $L(\lambda)$ only depends on the shape of some Young tableaux associated to $\lambda$. So in this paper, 	we write \[\tiny{\begin{tikzpicture}[scale=\domscale+0.1,baseline=-19pt]
			\hobox{0}{0}{2}
			\hobox{0}{1}{4}
			\hobox{0}{2}{1}
			\hobox{1}{0}{3}
	\end{tikzpicture}}:=\tiny{\begin{tikzpicture}[scale=\domscale+0.1,baseline=-19pt]
	\hobox{0}{0}{\lambda_2}
	\hobox{0}{1}{\lambda_4}
	\hobox{0}{2}{\lambda_1}
	\hobox{1}{0}{\lambda_3}
	\end{tikzpicture}}\]
	to represent our Young tableau $Y$, where the element $i$ in the Young tableau $Y$ represents the $i$-th component of $\lambda+\rho$ . For example, the Young tableau for $\lambda+\rho=(5,3,3,1)$ is 	\[Y=
		\tiny{\begin{tikzpicture}[scale=\domscale+0.1,baseline=-19pt]
	\hobox{0}{0}{1}
	\hobox{0}{1}{3}
	\hobox{0}{2}{5}
	\hobox{1}{0}{3}
	\end{tikzpicture}}.
	\]

We will use \[\tiny{\begin{tikzpicture}[scale=\domscale+0.1,baseline=-19pt]
	\hobox{0}{0}{4}
	\hobox{0}{1}{2}
	\hobox{0}{2}{1}
	\hobox{1}{0}{3}
	\end{tikzpicture}}\]
to represent our Young tableau $Y$.
\end{Rem}

\subsection{GKdim of highest weight modules for exceptional types}
First we recall the method of finding out the minimal length element $w_{\lambda}$ for an integral weight $\lambda$ such that $w_{\lam}^{-1}\cdot\lam=\mu$  is antidominant.

 Let $\Pi=\{\alpha_i \mid  1\leq i\leq n\}$ be the simple roots of an exceptional type root system $\Delta$ with corresponding fundamental weights $\{\xi_i\mid 1\leq i\leq n\}$. 
The Cartan matrix is defined by
$A=(A_{ij})_{n\times n}$ with $A_{ij}=(\alpha_i, \alpha_j^\vee)$ and $A_{ij}\in \{0,-1,-2,-3\}$ for $i \neq j$.
 Then  $\alpha_j=\sum\limits_{k=1}^n A_{jk}\xi_k$. For any $\lam\in \mathfrak{h}^*$, we can write $\lam+\rho=\sum\limits_{1\leq i\leq n}k_i\xi_i$, where $k_i=(\lam+\rho, \alpha_i^{\vee})$. Then an integral weight $\lam$ is  antidominant  if and only if   $k_i\in \mathbb{Z}_{\leq 0}$ for  all $1\leq i\leq n$.

Now we want to find the $w_{\lam}$ for an integral weight $\lambda$. When $\lam$ is antidominant (resp. dominant), $w_{\lam}=\mathrm{Id}$ (resp. $-\mathrm{Id}$). Suppose $\lam$ is not antidominant, then there exist some $k_i\in \mathbb{Z}_{>0}$. We choose the largest index  $i_{\lam}={i_1}$ such that $k_{i_1}\in \mathbb{Z}_{>0}$. Denote the simple reflection $s_{\alpha_i}$ by $s_i$. Then we have $$s_{i_1}(\lam+\rho)=\lam+\rho-k_{i_1}\alpha_{i_1}=\lam+\rho-k_{i_1}\sum\limits_{j=1}^n A_{i_1j}\xi_j=\lam+\rho-2k_{i_1}\xi_{i_1}-k_{i_1}\sum\limits_{j=1,j\neq i_1}^n A_{i_1j}\xi_j.$$
 Then the coefficient of $\xi_{i_1}$ in $s_{i_1}\lam$ becomes $-k_{i_1}$. 

The above process is called the \emph{positive  reduction algorithm}.

\begin{lem}[{\cite[Lem. 3.1]{do23}}]\label{w-lambda}
    For any integral weight $\lam$ (regular or singular),   we can get an antidominant weight $\mu$ after finite steps of  the positive reduction algorithm. We multiply all the $s_i$ appeared in these steps and get a $w_{\lambda}$. Then this $w_{\lambda}$ has the minimal length in $W$ such that $w_{\lambda}^{-1}\cdot\lam$ is antidominant.
\end{lem}

To compute the ${\bf a}$ value of the Weyl group elements of an exceptional type, we use the computer algebra package PyCox. Some more details about PyCox can be found in \cite{ge}. For example, the function `klcellrepem'  in PyCox will give us the 
${\bf a}$ value of $w \in W$. Note that $s_is_j \cdots s_k$ is denoted by $[i-1,j-1,\cdots,k-1]$ in PyCox for type $E_6$.

\begin{example}
    Let  $\lam+\rho=(-1,1,-1,-1,-1,1,1,-1)$ be  an integral weight of type $E_6$. We can write $\lam+\rho=-\xi_1+2\xi_3-2\xi_4:=[-1,0,2,-2,0,0]$. The largest index of $\lam+\rho$ with a positive element is $i_1=3$. Recall that $\alpha_3=-\xi_{1}+2\xi_{3}-\xi_{4}=:[-1,0,2,-1,0,0]$. Thus $s_{3}(\lam+\rho)=\lam+\rho-2\alpha_3=[1,0,-2,0,0,0]$.  The largest index of $s_3(\lam+\rho)$ with a positive element is $i_2=1$. Recall that  $\alpha_1=2\xi_{1}-\xi_3$. Thus $s_1s_{3}(\lam+\rho)=s_{3}(\lam+\rho)-\alpha_1=[-1,0,-1,0,0,0]$, which is already antidominant.    Thus we have $w_{\lam}=s_3s_{1}:=[2,0]$ in PyCox.
    Then by using PyCox, we have:
   \begin{verbatim}
    >>> W = coxeter("E", 6)
    >>> print(klcellrepelm(W, [2,0]))
   {'size': 6, 'character': [['6_p', 1]], 'a': 1, 'special': '6_p',
   'index': 2, 'elms': False, 'distinv': False}
\end{verbatim}
Then we know that ${\bf a}(w_{\lambda})=1$.
\end{example}

 When $\lambda$ is non-integral, we recall the algorithm in \cite{BGWX}.

For a root system $\Delta$, we use $\mathfrak{h}_{\Delta}$ to denote the corresponding Cartan subalgebra of the Lie algebra $\mathfrak{g}_\Delta$ associated with $\Delta$. For a simple root $\alpha$, we use $\xi_{\alpha}$ to denote the fundamental weight corresponding to it. 
We use $E_{\Delta}$ to denote the $\mathbb{R}$ linear span of $\Delta$ in $\mathfrak{h}_{\Delta}^*$. 


Recall that we use $\Delta_{[\lam]}$ to denote the integral root system for $\lam\in \mathfrak{h}_{\Delta}^*$ and $\Pi_{[\lambda]}$ to denote the simple system of  $\Delta_{[\lambda]}$ lying in $\Delta_{[\lambda]} \cap \Delta^+$.

Then we have the following algorithm to compute ${\bf a}(w_{\lam})$:

\begin{enumerate}
    \item[(1)] suppose $\Delta_{[\lam]}\simeq \Delta_1\times \Delta_2\times \cdots\times \Delta_k$ is a direct product of irreducible root systems. We use $\phi$ to denote this isomorphism.  Then $\phi$ induces an isomorphism from $E_{\Delta_{[\lam]}}$ to $\prod_{1\leq i\leq k} E_{\Delta_i}$. Suppose the simple system of $\Delta_{i}$ is $\Pi_i$ for $1\leq i\leq k$ and   the simple system of $\Delta_{[\lam]}$ is $\Pi_{[\lam]}=\{\alpha_1,\cdots,\alpha_n\}=\prod_{1\leq i\leq k}\Delta_{[\lam]_i}$, where $\Delta_{[\lam]_i}$ corresponds to $\Delta_i$.  Write $\lambda|_{\mathfrak{h}_{\Delta_{[\lambda]}}^*}=\sum_{\alpha_i\in \Pi_{[\lam]}}k_i\xi_i$, where $k_i=\langle \lam, \alpha_i^{\vee}\rangle$;
    \item[(2)] denote $\bar{\lam}'=\phi(\lam|_{\mathfrak{h}_{\Delta_{[\lambda]}}^*})=\sum_{\alpha_i\in \Pi_{[\lam]}}k_i\phi(\xi_i)$;
    \item[(3)]  when $\Delta_i$ is of classical type, denote $$\bar{\lam}'|_{\Delta_i}=\sum_{\alpha_{i_t}\in \Pi_{[\lam]_i}}k_{i_t}\phi(\xi_{i_t})=(\lam'_{1},\lam'_{2},\cdots,\lam'_{d_i}),$$ 
    where $d_i=\dim E_{\Delta_i}$.
    Then we can get the value of ${\bf a}(w_{\bar{\lam}'|_{{\Delta_i}}})$ without finding out $w_{\bar{\lam}'|_{{\Delta_i}}}$ by using the RS algorithm in \cite{BXX}. When some $\Delta_i$ is of exceptional type, we can find $w_{\bar{\lam}'|_{{\Delta_i}}}$ by using the algorithm in Lemma \ref{w-lambda}, and then ${\bf a}(w_{\bar{\lam}'|_{{\Delta_i}}})$ is given by using PyCox;
     \item[(4)] ${\bf a}(w_{\lam})=\sum_{1\leq i\leq k} {\bf a}(w_{\bar{\lam}'|_{{\Delta_i}}})$.
\end{enumerate}

By using Proposition \ref{pr:main1}, we can get \begin{equation*}
	\gkd L(\lambda)=|\Delta^+|-{\bf a}(w_{\lambda}).
\end{equation*}
From now on, we denote ${\bf a}({\lambda}):={\bf a}(w_{\lambda})$.

\section{Reducibility of scalar  generalized Verma modules for type $A_{n-1}$}

First we recall a useful result which will be used in our proof. 

\begin{prop} [{\cite[Theorem 1.3 $\&$ Theorem 5.1]{BZ}} or {\cite[Theorem 5.5]{Br}}]\label{mthm} 
Let $\mathfrak{g}=\mathfrak{sl}(n,\mathbb{C})$.
Suppose the parabolic subalgebra $\mathfrak{p}$ corresponds to two simple roots $\{\alpha_p, \alpha_q\}$.  A simple  module $L(\lambda)$ (with $\lambda=z_1\xi_p+z_2\xi_q$  and $\xi_p$ (resp. $\xi_q$) being the fundamental weight corresponding to the simple root $\alpha_p$ (resp. $\alpha_q$)) has $\gkd L(\lambda)=\dim (\mathfrak{u})$ if and only if $S(\lambda)$ consists of some Young tableaux whose numbers of boxes in their columns are  $\{c_1,c_2,c_3\}=\{p, q-p, n-q\}$.
\end{prop}
In the following, we will give the reducible points for type $A$. Since we're talking about two-step nilpotent  parabolic subalgebras, we need to compute the $GK$ dimension of the irreducible quotient $L(\lambda)$ of the scalar generalized Verma module 
$M_I(\lambda)$ with $\lambda=z_1\xi_p+z_2\xi_q$.

	Let $\mathfrak{g}=\mathfrak{sl}(n,\mathbb{C})$, and $\Delta^+(\frl) = \{\alpha_1,\dots,\alpha_{p-1},\alpha_{p+1},\dots,\alpha_{q-1},\alpha_{q+1},\dots,\alpha_{n-1}\}$, $\alpha_p=e_p-e_{p+1}$, $\alpha_{q}=e_q-e_{q+1}$.

We have
$$\xi_p=e_1+\cdots+e_p-\frac p n(e_1+\cdots+e_n),$$
$$\xi_q=e_1+\cdots+e_q-\frac q n(e_1+\cdots+e_n).$$

Denote $t=\frac{n-p}n$, $s=\frac{n-q}n$, then we have

\begin{center}
    $\xi_p=(\underbrace{t,\cdots, t}_p,\underbrace{t-1,\cdots, t-1}_{n-p})$, $\xi_q=(\underbrace{s,\cdots,s}_q,\underbrace{s-1,\cdots,s-1}_{n-q})$.
\end{center}

From \cite{EHW}, we have
$$\rho=(\frac{n-1}{2},\frac{n-3}{2},\dots,\frac{-n+3}{2},\frac{-n+1}{2}),$$
\begin{equation}\label{z1-z2}
	\begin{split}
		\lambda+\rho=(&\underbrace{z_1t+z_2s+\frac{n-1}{2},\dots,z_1t+z_2s+\frac{n-2p+1}{2}}_{p},
		\\&\underbrace{(t-1)z_1+z_2s+\frac{n-2p-1}{2},\dots,(t-1)z_1+z_2s+\frac{n-2q+1}{2}}_{q-p},
		\\&\underbrace{(t-1)z_1+(s-1)z_2+\frac{n-2q-1}{2},\dots,(t-1)z_1+(s-1)z_2+\frac{-n+1}{2}}_{n-q}).
	\end{split}
\end{equation}
We denote $\lambda+\rho=(\lambda_1,\lambda_2,\cdots,\lambda_n)$.

Since there are two parameters, we will divide our arguments into the following two cases.
\begin{enumerate}
	\item  $z_1=z_2$;
	\item  $z_1\neq z_2$.
	
\end{enumerate}

Suppose $M_{I}(\lambda)$ is a scalar generalized Verma module with highest weight $\lambda=z_1 \xi_p+z_2\xi_q$.
By Lemma \ref{GKdown}, if $(z_1, z_2)\in \mathbb{C}^2$ is a reducible point for $M_{I}(\lambda)$, then $(z_1+k,z_2+d)$ is also a reducible point for any $k,d\in \mathbb{Z}_{\geq 0}$. When $(z_1,z_2)$ is a reducible point of $M_{I}(\lambda)$ but $(z_1-1,z_2)$ or $(z_1,z_2-1)$ is not a reducible point, we will call $(z_1,z_2)$ {\it a boundary reducible point} of $M_{I}(\lambda)$. When $(z_1,z_1)$ is a reducible point but  $(z_1-1,z_1-1)$ is not a  reducible point, this boundary reducible point $(z_1,z_1)$ is also called {\it a first reducible point} of $M_{I}(\lambda)$ since it will be real and minimal in the diagonal line. See \cite{BJ}.

\subsection{Reducibility when $z_1=z_2$}

When $z_1=z_2$, we denote $z=z_1=z_2$.
Then 	
\begin{equation}\label{z1z2}
\begin{split}
\lambda+\rho=(&\underbrace{(s+t)z+\frac{n-1}{2},\dots,(s+t)z+\frac{n-2p+1}{2}}_{p},
\\&\underbrace{(s+t-1)z+\frac{n-2p-1}{2},\dots,(s+t-1)z+\frac{n-2q+1}{2}}_{q-p},
\\&\underbrace{(s+t-2)z+\frac{n-2q-1}{2},\dots,(s+t-2)z+\frac{-n+1}{2}}_{n-q}).
\end{split}
\end{equation}
We have $\lambda_p-\lambda_{p+1}=z+1$, $\lambda_p-\lambda_{q+1}=2z+1+q-p$ and $\lambda_q-\lambda_{q+1}=z+1$.

\begin{lem} 
		Let $ \mathfrak{g}=\mathfrak{sl}(n,\mathbb{C}) $. $M_{I}(\lambda)$ is a scalar generalized Verma module with highest weight $\lambda=z (\xi_p+\xi_q)$. We write $q-p=k$, $(k\ge 1)$, $m=\min\{p,n-q\}$, $h=\max\{p,n-q\}$.
  Suppose $z\notin \mathbb{Z}$. 
        Then $M_{I}(\lambda)$ is reducible if and only if   $m\ge 1$ and $z\in (\frac{1}{2}+\mathbb{Z})\bigcap (-\frac{k+m}{2},+\infty)$.

\end{lem}

\begin{proof}

 When $z\notin \frac{1}{2} +\mathbb{Z}$, 
from equation (\ref{z1z2}) we know that $[\lambda]$ consists of three elements and $\lambda+\rho$ corresponds to three Young tableaux and these three Young tableaux have only one column. The numbers of boxes in these three columns are $\{p, q-p,n-q\}$. 
Thus $M_I(z(\xi_p+\xi_q))$ is irreducible by Lemma \ref{reducible} and Proposition \ref{mthm}.

 When $z\in \frac{1}{2} +\mathbb{Z}$, equivalently $\lambda_p-\lambda_{q+1}\in \mathbb{Z}$, from equation (\ref{z1z2}) we know that $[\lambda]$ consists of two elements and  $\lambda+\rho$ corresponds to two Young tableaux. One of them  has one column and the number of boxes in this column is $q-p$. The other Young tableau may have two columns. We denote it by $P_2$.

Recall that we denote $m=\min\{p,n-q\}$, $h=\max\{p, n-q\}$. Then we have the follows.

\begin{itemize}
    \item[1)] If $m=0$ (equivalently $n-q=0$), the  Young tableau $P_2$ has one column and the number of boxes in this column is $p$, thus $M_I(z(\xi_p+\xi_q))$ is irreducible  by Lemma \ref{reducible} and Proposition \ref{mthm}. 
    \item[2)] If $m\ge 1$,  the Young tableau $P_2$ may have two columns. 

Firstly, 
when $z>-\frac{k+1}{2}=-\frac{q-p+1}{2}$, equivalently $\lambda_p>\lambda_{q+1}$,   the  Young tableau $P_2$ has one column and the number of boxes in this column is $n-q+p$. So $M_I(z(\xi_p+\xi_q))$ is reducible by Lemma \ref{reducible} and Proposition \ref{mthm}.

Secondly, when $-\frac{k+1}{2}\ge z>-\frac{k+m}{2}$, equivalently $\lambda_p\le \lambda_{q+1}$, $\lambda_p>\lambda_{n}$ (if $m=n-q$) or $\lambda_1>\lambda_{q+1}$ (if $m=p$), the  Young tableau $P_2$ has two columns and the numbers of boxes in these columns are $p+r$ and $n-q-r$, where $0<r<n-q$. So $M_I(z(\xi_p+\xi_q))$ is reducible by Lemma \ref{reducible} and Proposition \ref{mthm}.

When $z=-\frac{k+m}{2}$,   the  Young tableau $P_2$ has two columns and the numbers of boxes in these columns are $p$ and $n-q$, so $M_I(z(\xi_p+\xi_q))$ is irreducible by Lemma \ref{reducible} and Proposition \ref{mthm}.
\end{itemize}

Thus when $z\notin \mathbb{Z}$,   $M_I(z(\xi_p+\xi_q))$ is reducible if and only if $m\geq 1$ and  $z\in (\frac{1}{2}+\mathbb{Z})\bigcap (-\frac{k+m}{2},+\infty)$.

\end{proof}

\begin{lem} \label{m=0}
	Keep notations as above.
  Suppose $z\in \mathbb{Z}$ and $m=\min\{p,n-q\}=0$. 
        Then $M_{I}(\lambda)$ is reducible if and only if  \begin{center}
			    $z \in -\min\{h,k\}+1 + \mathbb{Z}_{\ge 0}$.
			\end{center}

\end{lem}

\begin{proof}

 When $z\in \mathbb{Z}$, $\lambda+\rho$ corresponds to one Young tableau $P(\lambda+\rho)$.

By Lemma \ref{GKdown}  we only need to find out the first reducible point for  $M_I(z(\xi_p+\xi_q))$ since any point larger than it will be a reducible point.

When $z>-1$, the Young tableau $P(\lambda+\rho)$ has only one column, so $M_I(z(\xi_p+\xi_q))$ is reducible by Lemma \ref{reducible} and Proposition \ref{mthm}.

When $z\le-1$, the Young tableau $P(\lambda+\rho)$ may consist of  two or three columns. Let's talk about the reducibility in a
case-by-case way.


\begin{enumerate}
    \item[1)] If $k=q-p=1$, we have $\{p, q-p, n-q\}=\{p, 1, 0\}$ since $m=\min\{p,n-q\}=0$ implies that $n-q=0$. Thus $p=n-1$.

	When $z=-1$, by using R-S algorithm we will have\[P(\lambda +\rho)=
	\tiny{\begin{tikzpicture}[scale=\domscale+0.25,baseline=-30pt]
			\hobox{0}{0}{n-1}
   \hobox{1}{0}{n}
			\hobox{0}{1}{\vdots}
			\hobox{0}{2}{1}
	\end{tikzpicture}}.
	\]
	
 Then by Lemma \ref{reducible} and Proposition \ref{mthm}, $z=-1$ is an irreducible point.

When $z=0$,  the Young tableau $P(\lambda+\rho)$ will consist of  one column. Thus $z=0$ is the first reducible point when $k=1$.

\item[2)]  If $k=q-p=2$, we have $\{p, q-p, n-q\}=\{p, 2, 0\}$ since $m=\min\{p,n-q\}=0$ implies that $n-q=0$. Thus $p=n-2$.

\begin{enumerate}
\item If $p=1$ and
 $z=-1$, we will have
\[
 \tiny{\begin{tikzpicture}[scale=\domscale+0.25,baseline=-15pt]
       \hobox{0}{0}{1}
       \end{tikzpicture}}\to
   \tiny{\begin{tikzpicture}[scale=\domscale+0.25,baseline=-15pt]
   		\hobox{0}{0}{1}
   		\hobox{1}{0}{2}
   \end{tikzpicture}}\to
\tiny{\begin{tikzpicture}[scale=\domscale+0.25,baseline=-15pt]
       \hobox{0}{1}{1}
       \hobox{1}{0}{2}
       \hobox{0}{0}{3}
       \end{tikzpicture}}=P(\lambda +\rho).
\]
Then by Lemma \ref{reducible} and Proposition \ref{mthm}, $z=-1$ is an irreducible point.
So $z=0$  is the first reducible point in this case.

\item If $p\ge 2$ and $z=-1$, we will have
	\[
	\tiny{\begin{tikzpicture}[scale=\domscale+0.25,baseline=-30pt]
			\hobox{0}{0}{n-2}
			\hobox{0}{1}{\vdots}
			\hobox{0}{2}{1}
	\end{tikzpicture}}\to
	\tiny{\begin{tikzpicture}[scale=\domscale+0.25,baseline=-30pt]
			\hobox{0}{0}{n-2}
			\hobox{0}{1}{\vdots}
			\hobox{0}{2}{1}
			\hobox{1}{0}{n-1}
	\end{tikzpicture}}\to
	\tiny{\begin{tikzpicture}[scale=\domscale+0.25,baseline=-30pt]
			\hobox{1}{0}{n-1}
			\hobox{0}{0}{n}
			\hobox{0}{1}{n-2}
			\hobox{0}{2}{\vdots}
			\hobox{0}{3}{1}
	\end{tikzpicture}}=P(\lambda +\rho).
	\]
	Then by Lemma \ref{reducible} and Proposition \ref{mthm},  $z=-1$ is a reducible point.
 
	 If $p\ge 2$ and  $z=-2$, we will have
	\[
	\tiny{\begin{tikzpicture}[scale=\domscale+0.25,baseline=-30pt]
			\hobox{0}{0}{n-2}
			\hobox{0}{1}{\vdots}
			\hobox{0}{2}{1}
	\end{tikzpicture}}\to
	\tiny{\begin{tikzpicture}[scale=\domscale+0.25,baseline=-30pt]
			\hobox{0}{0}{n-2}
			\hobox{0}{1}{\vdots}
			\hobox{0}{2}{1}
			\hobox{1}{0}{n-1}
	\end{tikzpicture}}\to
	\tiny{\begin{tikzpicture}[scale=\domscale+0.25,baseline=-30pt]
			\hobox{0}{0}{n-2}
			\hobox{0}{1}{\vdots}
			\hobox{0}{2}{1}
			\hobox{1}{0}{n}
			\hobox{1}{1}{n-1}
	\end{tikzpicture}}=P(\lambda +\rho).
	\]
	Then by Lemma \ref{reducible} and Proposition \ref{mthm},
$z=-2$ is an irreducible point.

So $z=-1$ is the first reducible point when $p\ge 2$. 

\end{enumerate}



 Thus $z=-\min \{p,k\}+1 $ is the first reducible point when $k=2$.

\item[3)]  If $k>2$, we have $\{p, q-p, n-q\}=\{p, q-p, 0\}$ since $m=\min\{p,n-q\}=0$ implies that $n-q=0$. 
\begin{enumerate}
 \item When $z=-p+1$ (if $p<k=n-p$) and $p<t=2p-1<n$, we will have
 \[
 \tiny{\begin{tikzpicture}[scale=\domscale+0.25,baseline=-30pt]
			\hobox{0}{0}{p}
			\hobox{0}{1}{\vdots}
			\hobox{0}{2}{2}
            \hobox{0}{3}{1}
	\end{tikzpicture}}\to
	\tiny{\begin{tikzpicture}[scale=\domscale+0.25,baseline=-30pt]
			\hobox{0}{0}{p}
			\hobox{0}{1}{\vdots}
			\hobox{0}{2}{2}
            \hobox{0}{3}{1}
            \hobox{1}{0}{t}
            \hobox{1}{1}{\vdots}
            \hobox{1}{2}{p+1}
	\end{tikzpicture}}\to
	\tiny{\begin{tikzpicture}[scale=\domscale+0.25,baseline=-30pt]
			\hobox{0}{0}{t+1}
			\hobox{0}{1}{p}
			\hobox{0}{2}{\vdots}
            \hobox{0}{3}{1}
            \hobox{1}{0}{t}
            \hobox{1}{1}{\vdots}
            \hobox{1}{2}{p+1}
	\end{tikzpicture}}\to
	\tiny{\begin{tikzpicture}[scale=\domscale+0.25,baseline=-30pt]
			\hobox{0}{0}{n}
			\hobox{0}{1}{\vdots}
            \hobox{0}{2}{t+1}
			\hobox{0}{3}{\vdots}
            \hobox{0}{4}{1}
            \hobox{1}{0}{t}
            \hobox{1}{1}{\vdots}
            \hobox{1}{2}{p+1}
	\end{tikzpicture}}=P(\lambda +\rho).
	\]
Then by Lemma \ref{reducible} and Proposition \ref{mthm}, $z=-p+1$ is a reducible point.

\item When $z=-k+1$ (if $k\le p$), we will have
 \[
 \tiny{\begin{tikzpicture}[scale=\domscale+0.25,baseline=-30pt]
			\hobox{0}{0}{p}
			\hobox{0}{1}{\vdots}
			\hobox{0}{2}{2}
            \hobox{0}{3}{1}
	\end{tikzpicture}}\to
	\tiny{\begin{tikzpicture}[scale=\domscale+0.25,baseline=-30pt]
			\hobox{0}{0}{p}
			\hobox{0}{1}{\vdots}
			\hobox{0}{2}{2}
            \hobox{0}{3}{1}
            \hobox{1}{0}{n-1}
            \hobox{1}{1}{\vdots}
            \hobox{1}{2}{p+1}
	\end{tikzpicture}}\to
	\tiny{\begin{tikzpicture}[scale=\domscale+0.25,baseline=-30pt]
			\hobox{0}{0}{n}
            \hobox{0}{1}{p}
			\hobox{0}{2}{\vdots}
            \hobox{0}{3}{2}
            \hobox{0}{4}{1}
            \hobox{1}{0}{n-1}
            \hobox{1}{1}{\vdots}
            \hobox{1}{2}{p+1}
	\end{tikzpicture}}=P(\lambda +\rho).
	\]
Then by Lemma \ref{reducible} and Proposition \ref{mthm},  $z=-k+1$ is a reducible point.

 \item When $z=-\min \{p,k\}$, the Young tableau $P(\lambda +\rho)$ has two columns and the numbers of boxes in  these columns are $k$ and $p$. 

 So by Lemma \ref{reducible} and Proposition \ref{mthm},  $z=-\min \{p,k\}$ is an irreducible point. 

\end{enumerate}

Thus  $z=-\min \{p,k\}+1$ is the first reducible point when $m=0$. 
\end{enumerate}

\end{proof}

\begin{lem} 
	Keep notations as above.
  Suppose $z\in \mathbb{Z}$ and $0< m=\min\{p,n-q\}< k-1=q-p-1$. 
        Then $M_{I}(\lambda)$ is reducible if and only if  
        \[	z\in\left\{
		\begin{array}{ll}
			-\max \{ \lceil  {\frac{k+m}{2}}\rceil, h \}+1 + \mathbb{Z}_{\ge 0},&\textnormal{if $h<k$},\\	     	  	
			-k+1 + \mathbb{Z}_{\ge 0}, &\textnormal{if $h\ge k$},\\	
		\end{array}	
		\right.
		\]
  equivalently,
   \[	z\in\left\{
		\begin{array}{ll}
			   	-\lceil  {\frac{k+m}{2}}\rceil+1+\mathbb{Z}_{\ge 0},&\textnormal{if $h\leq \lceil  {\frac{k+m}{2}}\rceil$},\\
-h+1 + \mathbb{Z}_{\ge 0},&\textnormal{if $\lceil  {\frac{k+m}{2}}\rceil<h<k$},\\	  

   	-k+1 + \mathbb{Z}_{\ge 0}, &\textnormal{if $h\ge k$}.\\	
		\end{array}	
		\right.
		\]

\end{lem}

\begin{proof}

Similar to the arguments in Lemma \ref{m=0}, we give the proof by induction on the value of $k$.

If $k=1$ or $2$, $m$ does not exist, so we start with $k=3$.
	
If $k=3$, we will have $m=1$. It is easy to see that the first reducible point is:
	\begin{center}
    $z=\begin{cases}
		-1, \qquad \text{if~} h<3=k, \\ -2, \qquad  \text{if~} h\ge 3=k.
	\end{cases}$
\end{center} 
	
	 If $k=4$, we will have the follows.

  \begin{enumerate}
  \item[1)] 
If $m=1$ and $h<4$, we will have $\{p,q-p,n-q\}=\{1,k,h\}=\{1,4,h\}$.


		When $z=-2$, the Young tableau $P(\lambda+\rho)$ has two columns and the numbers of boxes in these columns are $\{2, n-2\}$ or $\{3,n-3\}$, so by Lemma \ref{reducible} and Proposition \ref{mthm}, 
 $z=-2$ is a reducible point.

		When $z=-3$, the Young tableau $P(\lambda+\rho)$ has three columns and the numbers of boxes in these columns are $\{p, n-q,k=q-p\}$, so by Lemma \ref{reducible} and Proposition \ref{mthm}, 
$z=-3$ is an irreducible point.

		So   $z=-2$ is the first reducible point when $h<4$.
		
		 If $m=1$ and  $h\ge 4$, we have the follows.
		


		When $z=-3$, the Young tableau $P(\lambda+\rho)$ has three columns and the numbers of boxes in these columns are $\{1, 3,n-4\}$, so by Lemma \ref{reducible} and Proposition \ref{mthm},
 $z=-3$ is a reducible point.

		When $z=-4$, the Young tableau $P(\lambda+\rho)$ has three columns and the numbers of boxes in these columns are $\{1, 4,n-3\}$, so by Lemma \ref{reducible} and Proposition \ref{mthm},  $z=-4$ is an irreducible point.

		So  $z=-3$ is the first reducible point when $h\ge 4$.

 Thus   when $m=1$, the first reducible point is:

 \begin{center}
 $z=\begin{cases}
		-2,  \qquad \text{if~} h<4,  \\ -3,  \qquad \text{if~}h\ge 4.
	\end{cases}$
 \end{center}

\item[2)]
	 If $m=2$ and $h<4$, we will have $\{p,q-p,n-q\}=\{2,k,h\}=\{2,4,h\}$.
	
	 When $z=-2$, the Young tableau $P(\lambda+\rho)$ has two columns, so by Lemma \ref{reducible} and Proposition \ref{mthm},  $z=-2$ is a reducible point.
	
	When $z=-3$, the Young tableau $P(\lambda+\rho)$ has three columns and the numbers of boxes in these columns are $\{2, 4,h\}$, so by Lemma \ref{reducible} and Proposition \ref{mthm}, 
     $z=-3$ is an irreducible point.

	If $m=2$ and $h\ge 4$, we have the follows.

	When $z=-3$, the Young tableau $P(\lambda+\rho)$ has three columns and the numbers of boxes in these columns are $\{2, 3,n-5\}$, so by Lemma \ref{reducible} and Proposition \ref{mthm}, 
     $z=-3$ is a reducible point.

	When $z=-4$, the Young tableau $P(\lambda+\rho)$  has three columns and the numbers of boxes in these columns are $\{2, 4,h\}$, so by Lemma \ref{reducible} and Proposition \ref{mthm}, 
    $z=-4$ is an irreducible point.

Thus when $m=2$, the first reducible point is: 
\begin{center}
    $z=\begin{cases}
	-2,  \qquad\text{if}~ h<4,  \\ -3,  \qquad\text{if}~ h\ge 4.
\end{cases}$
\end{center}

\end{enumerate}

So we can conclude that: when $k=4$, the first reducible point is:
\begin{center}
   $z=$ $\begin{cases}
	-2,  \qquad\text{if}~ h<4,  \\ -3,  \qquad\text{if}~ h\ge 4.
\end{cases}$
\end{center}
	
 If $k=5$, we will have the follows.

\begin{enumerate}
	\item[1)] If $m=1$, we will have $\{p,q-p,n-q\}=\{1,k,h\}=\{1,5,h\}$.
	
	When $z=-2$, for any $h$,  the Young tableau $P(\lambda+\rho)$ has two columns, so by Lemma \ref{reducible} and Proposition \ref{mthm}, $z=-2$ is a  reducible point.

		When $z=-3$ and $h\leq \lceil  {\frac{k+m}{2}}\rceil=3$, the Young tableau $P(\lambda+\rho)$ has three columns and the numbers of boxes in these columns are $\{1, h,5\}$, so by Lemma \ref{reducible} and Proposition \ref{mthm}, $z=-3$ is an irreducible point.
		Thus $z=-2$ is the first reducible point when $h\leq \lceil  {\frac{k+m}{2}}\rceil=3$.

		When $z=-3$ and $h=4$, the Young tableau $P(\lambda+\rho)$ has three columns and the numbers of boxes in these columns are $\{1, 3,n-4\}=\{1, 3, 6\}$, so by Lemma \ref{reducible} and Proposition \ref{mthm}, 
      $z=-3$ is a reducible point.

		When $z=-4$   and $h=4$, the Young tableau $P(\lambda+\rho)$  has three columns and the numbers of boxes in these columns are $\{1, 4,n-5\}=\{1, 4,5\}$, so by Lemma \ref{reducible} and Proposition \ref{mthm},  $z=-4$ is an irreducible point.
Thus $z=-3$ is the first reducible point when $h=4$.

		When $z=-4$ and $h\ge 5$, the Young tableau $P(\lambda+\rho)$  has three columns and the numbers of boxes in these columns are $\{1, 4,n-5\}$, so by Lemma \ref{reducible} and Proposition \ref{mthm}, 
       $z=-4$ is a reducible point.

  
	When $z=-5$ and $h\ge 5$, the Young tableau $P(\lambda+\rho)$  has three columns and the numbers of boxes in these columns are $\{1, 5,n-6\}$, so by Lemma \ref{reducible} and Proposition \ref{mthm}, 
       $z=-5$  is an irreducible point.
  Thus $z=-4$ is the first reducible point when $h\ge 5$.

  Therefore when $m=1$, the first reducible point is:
    \begin{center}
    $z=$
      $\begin{cases}-2, \qquad\text{if}~ h\leq \lceil  {\frac{k+m}{2}}\rceil=3,\\ -3, \qquad\text{if}~ h=4, \\
	 	-4, \qquad\text{if}~ h\ge 5.
	 \end{cases}$
  \end{center}

 \item[2)] If $m=2$ or $3$, similarly we can get that  the first reducible point is:  
 \begin{center}
    $z=$ $\begin{cases} -3, \qquad\text{if}~ h\leq \lceil  {\frac{k+m}{2}}\rceil=4<k,\\ -4, \qquad\text{if}~ h\ge 5=k.
	\end{cases}$
 \end{center}

 \end{enumerate}	

So we can conclude that: when $k=5$, the first reducible point is:
\begin{center}
    $z=\begin{cases}-\max\{\lceil \frac{5+m}{2}\rceil,h\}+1,  \qquad& \text{if}~h<5,\\ -4=-5+1, \qquad& \text{if}~h\ge 5.
\end{cases}$
\end{center}

The above arguments can be generalized to the general case. Thus when $0<m<k-1$, the first reducible point is:
\begin{center}
   $z=\begin{cases}-\max\{\lceil \frac{k+m}{2}\rceil,h\}+1,  \qquad &\text{if}~h<k,\\ -k+1, \qquad &\text{if}~h\ge k.
\end{cases}$
\end{center}


\end{proof}

\begin{lem} 
	Keep notations as above.
  Suppose $z\in \mathbb{Z}$ and $m\ge k-1=q-p-1$. 
        Then $M_{I}(\lambda)$ is reducible if and only if  
       
	\[	z\in\left\{
		\begin{array}{ll}
			-\lceil {\frac {m}{2}}\rceil  -\lfloor {\frac{k-1}{2}}\rfloor + \mathbb{Z}_{\ge 0},&\textnormal{if $k$ is even},\\	     	  	
			-\lfloor {\frac {m}{2}}\rfloor  -\lfloor {\frac{k-1}{2}}\rfloor + \mathbb{Z}_{\ge 0}, &\textnormal{if $k$ is odd}.	
		\end{array}	
		\right.
		\]

\end{lem}

\begin{proof}

 If $k=1$, then $m\ge 0$  and $\{p,q-p,n-q\}=\{m,1,h\}$. We will have the follows.
	\begin{enumerate}
	    \item[1)] If $m=0$, from Lemma \ref{m=0} we  know that $z=0=-\lfloor {\frac {0}{2}} \rfloor$ is the first reducible point. 
     \item[2)] If $m=1$ and $z=-1$, the Young tableau $P(\lambda+\rho)$ has three columns and the numbers of boxes in these columns are $\{1,1,n-2\}=\{m,1,h\}$, so by Lemma \ref{reducible} and Proposition \ref{mthm},  $z=0=-\lfloor {\frac {1}{2}} \rfloor$ is the first reducible point.
	\item[3)] If $m>1$ and $z=-\lfloor {\frac {m}{2}}\rfloor$, the Young tableau $P(\lambda+\rho)$ has three columns and the numbers of boxes in these columns are $\{1,m-1,n-m\}$, so by Lemma \ref{reducible} and Proposition \ref{mthm}, $z=-\lfloor {\frac {m}{2}} \rfloor$ is a reducible point.
	
 If $m>1$ and $z=-\lfloor {\frac {m}{2}}\rfloor-1$, the Young tableau $P(\lambda+\rho)$ has three columns and the numbers of boxes in these columns are $\{1,m,n-m-1\}$, so by Lemma \ref{reducible} and Proposition \ref{mthm}, $z=-\lfloor {\frac {m}{2}} \rfloor-1$ is an irreducible point.
 
	\end{enumerate}
	
Thus $z=-\lfloor {\frac {m}{2}} \rfloor=-\lfloor {\frac {m}{2}} \rfloor-\lfloor {\frac {1-1}{2}} \rfloor$ is the first reducible point when $k=1$. 
	
	 If $k=2$, then $m\ge 1$  and $\{p,q-p,n-q\}=\{m,2,h\}$. We will have the follows.
	\begin{enumerate}
		\item[1)] If $m=1$, we will have $\{p,q-p,n-q\}=\{1,2,n-3\}$.
		
		When $z=-1$, the Young tableau $P(\lambda+\rho)$ has two columns and the numbers of boxes in these columns are $\{2, n-2\}$, so by Lemma \ref{reducible} and Proposition \ref{mthm}, 
		    $z=-1$ is a reducible point.

		When $z=-2$, the Young tableau $P(\lambda+\rho)$ has three columns and the numbers of boxes in these columns are $\{1,2,n-3\}$, so by Lemma \ref{reducible} and Proposition \ref{mthm},  $z=-2$ is an irreducible point.

		Thus $z=-1=-\lceil {\frac {1}{2}}\rceil$ is the first reducible point when $m=1$.
		
		\item[2)] If $m=2$, we will have $\{p,q-p,n-q\}=\{2,2,n-4\}$.

		When $z=-1$,  the Young tableau $P(\lambda+\rho)$ has two columns and the numbers of boxes in these columns are $\{2,n-2\}$, so by Lemma \ref{reducible} and Proposition \ref{mthm}, 
		   $z=-1$ is a reducible point.

		When $z=-2$, the Young tableau $P(\lambda+\rho)$ has three columns and the numbers of boxes in these columns are $\{2,2,n-4\}$, so by Lemma \ref{reducible} and Proposition \ref{mthm}, $z=-2$ is an irreducible point.

		Thus $z=-1=-\lceil {\frac {2}{2}}\rceil$ is the first reducible point when $m=2$.
		
		\item[3)] If $m>2$, we will have the follows. 
		
		When $z=-\lceil {\frac {m}{2}}\rceil$, the Young tableau $P(\lambda+\rho)$ has two columns and the numbers of boxes in these columns are $\{2,n-2\}$, so by Lemma \ref{reducible} and Proposition \ref{mthm}, 
		   $z=-1$ is a reducible point.

		When $z=-\lceil {\frac {m}{2}}\rceil-1$, the Young tableau $P(\lambda+\rho)$ has three columns and the numbers of boxes in these columns are $\{m,2,n-m-2\}$, so by Lemma \ref{reducible} and Proposition \ref{mthm},  $z=-\lceil {\frac {m}{2}}\rceil-1$ is an irreducible point.

		Thus $z=-\lceil {\frac {m}{2}}\rceil$ is the first reducible point when $m>2$.

	\end{enumerate}
	
	Therefore $z=-\lceil {\frac {m}{2}}\rceil=-\lceil {\frac {m}{2}}\rceil-\lfloor {\frac {2-1}{2}} \rfloor$ is the first reducible point when $k=2$.

If $k\geq 3$ and $m\ge k-1$, by similar arguments we can get that the first reducible point is:
\begin{center}
   $z=\begin{cases}  -\lceil {\frac {m}{2}} \rceil -\lfloor {\frac {k-1}{2}} \rfloor, &~\text{if}~ k ~\text{is even},\\
   -\lfloor {\frac {m}{2}} \rfloor-\lfloor {\frac {k-1}{2}} \rfloor, &~\text{if}~ k ~\text{is odd}.
\end{cases}$
\end{center}
\end{proof}


	


	

Combined the above lemmas, we complete the proof of  part (1) in Theorem \ref{thm-a}.

\begin{example}
	Let $\frg=\frsl(8, \mathbb{C})$. We consider the parabolic subalgebra $\frq$ corresponding to the subset $\Pi\setminus\{\alpha_2, \alpha_5\}$.
	In this case $p=2$, $q-p=3$ and $n-q=3$. So the scalar generalized Verma module $M_{I}(z(\xi_2+\xi_5))$  is reducible if and only if
	$z\in (-2+\mathbb{Z}_{\ge 0})\bigcup (-\frac{3}{2}+\mathbb{Z}_{\ge 0})=-2+\frac{1}{2}\mathbb{Z}_{\ge 0}$.
\end{example}

\subsection{Reducibility when $z_1\neq z_2$}

First we recall a useful result in \cite{BJ}. 

\begin{prop}[{\cite[Theorem 4.3]{BJ}}]\label{a-reducibility]}
	Let $ \mathfrak{g}=\mathfrak{sl}(n,\mathbb{C}) $. $M_{I}(z\xi_p)$ is a scalar generalized Verma module with highest weight $z\xi_p$, where $\xi_p$ is the fundamental weight corresponding to $\alpha_{p}=e_p-e_{p+1}$. Then $M_{I}(z\xi_p)$ is reducible if and only if $z\in 1-\min\{p,n-p\}+\mathbb{Z}_{\ge 0}$.
\end{prop}

\begin{lem} 
	Let $ \mathfrak{g}=\mathfrak{sl}(n,\mathbb{C}) $. $M_{I}(\lambda)$ is a scalar generalized Verma module with highest weight $\lambda=z_1 \xi_p+z_2\xi_q$ ($z_1\neq z_2$).  We write $m=\min \{p, n-q\}$, $k=q-p$. 
 Suppose $z_1\notin \mathbb{Z}$ and $z_2\notin \mathbb{Z}$. Then   $M_I(z_1\xi_p+z_2\xi_q)$ is reducible if and only if $$m\geq 1 \text{~and~} z_1+z_2\in -k-m+\mathbb{Z}_{>0}.$$

\end{lem}

\begin{proof}
Recall that

\begin{equation*}
	\begin{split}
		\lambda+\rho=(&\underbrace{z_1t+z_2s+\frac{n-1}{2},\dots,z_1t+z_2s+\frac{n-2p+1}{2}}_{p},
		\\&\underbrace{(t-1)z_1+z_2s+\frac{n-2p-1}{2},\dots,(t-1)z_1+z_2s+\frac{n-2q+1}{2}}_{q-p},
		\\&\underbrace{(t-1)z_1+(s-1)z_2+\frac{n-2q-1}{2},\dots,(t-1)z_1+(s-1)z_2+\frac{-n+1}{2}}_{n-q}).
	\end{split}
\end{equation*}
Thus $\lambda_p-\lambda_{p+1}=z_1+1$, $\lambda_p-\lambda_{q+1}=z_1+z_2+1+q-p$ and $\lambda_q-\lambda_{q+1}=z_2+1$.

	When  $z_1\notin \mathbb{Z}$ and $ z_2\notin \mathbb{Z}$, we have $\lambda_p-\lambda_{p+1}\notin \mathbb{Z}$ and $\lambda_q-\lambda_{q+1}\notin \mathbb{Z}$. Then we have the follows.
\begin{enumerate}
    \item[1)]  When $z_1+z_2\notin \mathbb{Z}$, 
	we know that $[\lambda]$ consists of three elements and $\lambda+\rho$ corresponds to  three Young tableaux and these three Young tableaux have only one column. The numbers of boxes in these columns  are $\{p, q-p,n-q\}$. 
 So  $M_I(z_1\xi_p+z_2\xi_q)$ is irreducible by Lemma \ref{reducible} and Proposition \ref{mthm}.
 \item[2)] When $z_1+z_2\in \mathbb{Z}$, equivalently $\lambda_p-\lambda_{q+1}\in \mathbb{Z}$, we know that $[\lambda]$ consists of two elements and $\lambda+\rho$ corresponds to  two Young tableaux. One of  these Young tableaux has only one column and the number of boxes in this column  is $q-p$. The other Young tableau may have two columns. We denote it by $P_2$.

If $m=0$ (equivalently $n-q=0$), the  Young tableau $P_2$ has one column and the number of boxes in this column is $p$, so $M_I(z_1\xi_p+z_2\xi_q)$ is irreducible by Lemma \ref{reducible} and Proposition \ref{mthm}.  

If $m\ge 1$,  the Young tableau $P_2$ may have two columns. 

Firstly, 
when $z_1+z_2>-k-m$,  the  Young tableau $P_2$ has two columns and the numbers of boxes in these columns are $\{p+r,n-q-r\}$, where $0<r\leq n-q$. So $M_I(z_1\xi_p+z_2\xi_q)$ is reducible by Lemma \ref{reducible} and Proposition \ref{mthm}.

Secondly, when $z_1+z_2=-k-m$,   the  Young tableau $P_2$ has two columns and the numbers of boxes in these columns are $\{p,n-q\}$. So $M_I(z_1\xi_p+z_2\xi_q)$ is irreducible by Lemma \ref{reducible} and Proposition \ref{mthm}.
 
\end{enumerate}

 Thus when $z_1,z_2\notin \mathbb{Z}$,   $M_I(z_1\xi_p+z_2\xi_q)$ is reducible if and only if $m\geq 1$ and $z_1+z_2\in -k-m+\mathbb{Z}_{>0}$.
 
\end{proof}

\begin{lem} 
	Keep notations as above. Suppose $z_1\notin \mathbb{Z}$ and $ z_2\in \mathbb{Z}$.  Then $M_I(z_1\xi_p+z_2\xi_q)$ is reducible if and only if 
 $$m\geq 1 \text{~and~} z_2\in 1-\min\{k, n-q\}+ \mathbb{Z}_{\ge 0}.$$

\end{lem}

\begin{proof}

When $z_1\notin \mathbb{Z}$ and $ z_2\in \mathbb{Z}$, equivalently  $\lambda_p-\lambda_{p+1}\notin \mathbb{Z}$ and $\lambda_q-\lambda_{q+1}\in \mathbb{Z}$,  we know that $[\lambda]$ consists of two elements and $\lambda+\rho$ corresponds to  two Young tableaux. Then we have the follows.
	\begin{enumerate}
		\item[1)] If $m=0$ (equivalently $n-q=0$), 
		  each  Young tableau has one column. So $M_I(z_1\xi_p+z_2\xi_q)$ is irreducible by Lemma \ref{reducible} and Proposition \ref{mthm}.  
		\item[2)] If $n-q\geq m>0$,
		one of these Young tableaux has one column, and the number of boxes in this column is $p$.
		The other one has one column or two columns.
  
  By Proposition \ref{a-reducibility]},  $M_I(z_1\xi_p+z_2\xi_q)$  is reducible if and only if 
		\begin{center}
		    $z_2\in 1-\min\{q-p, n-q\}+ \mathbb{Z}_{\ge0}$. 
		\end{center}
		
	\end{enumerate}
This finishes the proof.

\end{proof}

\begin{lem} 
	Keep notations as above. Suppose  $ z_1\in \mathbb{Z}$.  Then $M_I(z_1\xi_p+z_2\xi_q)$ is reducible if and only if 
 $$m\geq 1, z_2\notin \mathbb{Z} \text{~and~} z_1\in 1-\min\{p, q-p\}+ \mathbb{Z}_{\ge 0},$$
 or 
  $$m=0, z_2\in \mathbb{C} \text{~and~} z_1\in 1-\min\{p, n-p\}+ \mathbb{Z}_{\ge 0}.$$

\end{lem}

\begin{proof}

 If $m\geq 1$, 
	 when $z_2\notin \mathbb{Z}$ and $ z_1\in \mathbb{Z}$,  we know that $[\lambda]$ consists of two elements and $\lambda+\rho$ corresponds to  two Young tableaux. One of these Young tableaux has one column, and the number of boxes in this column is $n-q$.
	The other one has one column or two columns.
 
 By Proposition \ref{a-reducibility]},  $M_I(z_1\xi_p+z_2\xi_q)$  is reducible if and only if 
 \begin{center}
	   $z_1\in 1-\min\{p, q-p\}+ \mathbb{Z}_{\ge0}$.
 \end{center}

If $m=0$ (equivalently $n-q=0$), when $z_2\in \mathbb{Z}$ and $ z_1\in \mathbb{Z}$,
$\lambda +\rho$ corresponds to  one Young tableau.
By Proposition \ref{a-reducibility]}, $M_I(z_1\xi_p+z_2\xi_q)$  is reducible if and only if 
 \begin{center}
	   $z_1\in 1-\min\{p, n-p\}+ \mathbb{Z}_{\ge0}$.
 \end{center}

\end{proof}

\begin{lem} 
	Keep notations as above. Suppose 
$m\geq 1$, $z_1\neq z_2$, $z_1\in \mathbb{Z}$ and $ z_2\in \mathbb{Z}$.  Then $M_I(z_1\xi_p+z_2\xi_q)$ is reducible if and only if

				\begin{enumerate}
					\item $z_1\in \mathbb{Z}$, $z_2>-1$, or 
     \item $z_1>-1$, $z_2\in \mathbb{Z}$, or
					
					\item   $z_1\le -1, z_2\le -1$, and
    	\begin{enumerate}
		    \item  $z_1+z_2 >-k-m$, or
     \item $ z_1>-\min \{k, p\}$, $z_1\neq z_2$, or
     \item $z_2>-\min \{k, n-q\}$, $z_1\neq z_2$.
	\end{enumerate}
					
				\end{enumerate}

\end{lem}

\begin{proof}
 
    When $z_1\in \mathbb{Z} $ 
 and $z_2\in \mathbb{Z}$, $\lambda +\rho$ corresponds to  one Young tableau.

	If $n-q\geq m\ge 1$, we will have the following.
	
		 When  $z_1\in \mathbb{Z}$ and $z_2>-1$ or $z_1>-1$ and $z_2\in \mathbb{Z}$, the Young tableau has two columns, so by Lemma \ref{reducible} and Proposition \ref{mthm},
	 $M_I(z_1\xi_p+z_2\xi_q)$ will be reducible .

		
		 When $z_1\le -1$ and $z_2\le -1$,
 we fix the value of $k=q-p$ and change the values of $p$ and $n-q$.  
 
 If $k=1$,  we will have $\{p,q-p,n-q\}=\{p,1,n-p-1\}$. Then we have the following.

\begin{enumerate}
    \item[1)]
 If $m=1$, when $(z_1,z_2)=(-1, -2)$ or $(z_1,z_2)=(-2, -1)$,  the Young tableau $P(\lambda+\rho)$ has three columns and the numbers of boxes in these columns are $\{1,1,n-2\}$.
 By Lemma \ref{reducible} and Proposition \ref{mthm}, $(z_1,z_2)$ is an irreducible point.

\item[2)]  If $m=2$, when $(z_1,z_2)=(-1, -2)$ or $(z_1,z_2)=(-2, -1)$,  the Young tableau $P(\lambda+\rho)$ has three columns and the numbers of boxes in these columns are $\{1,2,n-3\}$.
 By Lemma \ref{reducible} and Proposition \ref{mthm}, $(z_1,z_2)$ is an irreducible point.

   Therefore $M_I(z_1\xi_p+z_2\xi_q)$ is  irreducible when $m=1 ~\text{or}~2$ for all $z_1\le -1$, $z_2\le -1$ and $z_1\neq z_2$.

\item[3)] If $m=3$, when $(z_1,z_2)=(-1, -2)$ or $(z_1,z_2)=(-2, -1)$,  the Young tableau $P(\lambda+\rho)$ has three columns and the numbers of boxes in these columns are $\{1,2,n-3\}$.
 By Lemma \ref{reducible} and Proposition \ref{mthm}, $(z_1,z_2)$ is a reducible point.

 When $(z_1,z_2)=(-1, -3)$ or $(z_1,z_2)=(-3, -1)$,  the Young tableau $P(\lambda+\rho)$ has three columns and the numbers of boxes in these columns are $\{1,3,n-4\}$.
 By Lemma \ref{reducible} and Proposition \ref{mthm}, $(z_1,z_2)$ is an irreducible point.

       Therefore  $(z_1,z_2) $ is a boundary reducible point of   $M_I(z_1\xi_p+z_2\xi_q)$ if and only if $z_1+z_2=-3>-k-m=-1-3=-4$.

\item[4)] If $m\geq 4$, when $z_1+z_2>-(m+1)$, by similar arguments with above, $M_I(z_1\xi_p+z_2\xi_q)$ is reducible.

    	When $z_1+z_2=-1-m$, 
    	the Young tableau $P(\lambda+\rho)$ has three columns, and the numbers of boxes in these columns are $\{1,m,n-m-1\}$.
    	 By Lemma \ref{reducible} and Proposition \ref{mthm}, $(z_1,z_2)$ is an  irreducible point.

    \end{enumerate}
   
    Thus when $m\geq 4$, $M_I(z_1\xi_p+z_2\xi_q)$ is reducible if and only if  $z_1+z_2>-(m+1)$. 

Therefore if $k=1$,  the reducible points are:

\begin{center}
    $ \begin{cases}
         \emptyset, &~\text{if}~ m=1,2,\\
 z_1+z_2>-(m+1), &~ \text{if} ~m\ge 3,
    \end{cases} $
\end{center}
for all $z_1\le -1$, $z_2\le -1$, $z_1\neq z_2$.

 If $k=2$,  we will have $\{p,q-p,n-q\}=\{p,2,n-p-2\}$. Then we have the following.

		\begin{enumerate}
  \item[1)]  If $m=\min\{p,n-q\}=p=n-q=1$,
  it is easy to check that $M_I(z_1\xi_p+z_2\xi_q)$ is irreducible.
		
   If $p=1, n-q>1$. Then we have the follows.
			
When $(z_1,z_2)=(-1,-2)$, the Young tableau $P(\lambda+\rho)$ has three columns and the numbers of boxes in these columns are $\{1,2,n-3\}$.
    By Lemma \ref{reducible} and Proposition \ref{mthm}, $(z_1,z_2)$ is an irreducible point.
				

When $(z_1,z_2)=(-s,-1) $ ($s\ge 2$), the Young tableau $P(\lambda+\rho)$ has three columns and the numbers of boxes in these columns are $\{1,1,n-2\}$.
 By Lemma \ref{reducible} and Proposition \ref{mthm}, $(z_1,z_2)$  is a reducible point.

				When $(z_1,z_2)=(-2,-3) $ or $(-3,-2)$, the Young tableau $P(\lambda+\rho)$ has three columns and the numbers of boxes in these columns are $\{1,2,n-3\}$.
    By Lemma \ref{reducible} and Proposition \ref{mthm}, $(z_1,z_2)$ is an irreducible point.
 
		Thus when $p=1$ and $ n-q>1$,
  $M_I(z_1\xi_p+z_2\xi_q)$ is reducible if and only if $(z_1,z_2)=(-s,-1) $ ($s\ge2$).
  
		 If $p>1, n-q=1$, by similar arguments, we can get that
       $M_I(z_1\xi_p+z_2\xi_q)$ is reducible if and only if  $(z_1,z_2)=(-1,-l) $ ($l\ge2$).
   
  Therefore if $m=1$, the boundary  reducible points of $M_I(z_1\xi_p+z_2\xi_q)$ are: 
\begin{center}
$(z_1,z_2)=\begin{cases} 
\emptyset, &\text{if}~p=n-q=1,\\
(-s, -1),  &\text{if}~ n-q\ge 2,\\
(-1, -l),   &\text{if}~ p\ge 2,
\end{cases}$
\end{center}
where $s$, $l\in \mathbb{Z}_{\ge 2}$.

   \item[2)]   If $m=2~\text{or}~3$, by similar arguments, we can get that $M_I(z_1\xi_p+z_2\xi_q)$ is reducible if and only if  $\max\{z_1, z_2\}=-1$.

\item[3)] If $m\ge 4$, we will have the following.

When $\max\{z_1, z_2\}=-1$, it is easy to check that $M_I(z_1\xi_p+z_2\xi_q)$ is reducible.

If $p=4, n-q\ge 4$,  we have the follows. 

When $(z_1, z_2)=(-2,-3)$ or $(-3,-2)$, the Young tableau $P(\lambda+\rho)$ has three columns and the numbers of boxes in these columns are $\{n-5,3,2\}$.  By Lemma \ref{reducible} and Proposition \ref{mthm}, $(z_1,z_2)$  is a reducible point.

When $(z_1, z_2)=(-2,-4)$ or $(-4,-2)$, the Young tableau $P(\lambda+\rho)$ has three columns and the numbers of boxes in these columns are $\{n-6,4,2\}$.  By Lemma \ref{reducible} and Proposition \ref{mthm}, $(z_1,z_2)$ is an irreducible point.

  	Thus when $p=4, n-q\ge 4$,
       $M_I(z_1\xi_p+z_2\xi_q)$ is reducible if and only if $z_1+z_2>-(2+4)=-6$ or $\max\{z_1, z_2\}=-1$.

		 If $p=m> 4, n-q\ge m $, the arguments are similar.
	

  Thus if $m\geq 4$,
     $M_I(z_1\xi_p+z_2\xi_q)$ is reducible if and only if
  $z_1+z_2> -(2+m)$ or $\max \{ z_1, z_2 \}=-1$.

 \end{enumerate}

Therefore when $k=2$, $(z_1,z_2) $ is a first reducible point of  $M_I(z_1\xi_p+z_2\xi_q)$ if and only if it satisfies the following
\begin{center}
   $\begin{cases}
 (z_1, z_2)= \begin{cases}
 \emptyset, &\text{if}~p=n-q=1,\\
(-s,-1),   &\text{if}~ n-q\ge 2,\\ (-1,-l),   &\text{if}~p\ge 2,
  \end{cases}   &\text{if}~m=1, \\ \max \{z_1,z_2\}=-1,  &\text{if}~m=2,3,  \\ z_1+z_2>-(2+m), \text{~or~} \max \{z_1,z_2\}=-1  &\text{if}~m\ge 4,
\end{cases}$
\end{center}
where  $s,l \in \mathbb{Z}_{\geq 2}$.

  If $k\ge 3$, by similar arguments, 
  we can get that the reducible points of $M_I(z_1\xi_p+z_2\xi_q)$ will satisfy one of the following conditions:
	\begin{enumerate}
		    \item[(a)]  $z_1+z_2 >-k-m$.
     \item[(b)] $ z_1>-\min \{k, p\}$.
     \item[(c)] $z_2>-\min \{k, n-q\}$.
	\end{enumerate}
  Here $z_1\le -1$, $z_2\le -1$, $z_1\neq z_2$.

This finishes the proof.

		
					
					
     


\end{proof}

\begin{Cor}
   Keep notations as above. The scalar generalized Verma module $M_I(z_1\xi_p+z_2\xi_q)$  ($z_1\neq z_2$) is reducible if and only if one of the following holds. 
    \begin{enumerate}
		\item If $n-q=0$, we have  $z_1\in 1-\min\{p, q-p\}+ \mathbb{Z}_{\ge 0},  z_2\in \mathbb{C}$.
		
		\item If $n-q\ge 1$, we have
		\begin{enumerate}
			\item  $z_2\in 1-\min\{q-p, n-q\}+ \mathbb{Z}_{\ge 0}$, $z_1\in \mathbb{C}$, or
			\item  $z_1\in 1-\min\{p, q-p\}+ \mathbb{Z}_{\ge 0}$, $z_2\in \mathbb{C}$, or
			\item  $(z_1,z_2)\in \mathbb{C}\times \mathbb{C}$, and
				 $z_1+z_2 \in -q+p-\min\{p,n-q\}+\mathbb{Z}_{>0}$.
	\end{enumerate}
					
				\end{enumerate}

\end{Cor}
\begin{example}
	Let $\frg=\frsl(10, \mathbb{C})$. We consider the parabolic subalgebra $\frq$ corresponding to the subset $\Pi\setminus\{\alpha_3, \alpha_6\}$.
	In this case $p=3$, $q-p=3$ and $n-q=4$. So the scalar generalized Verma module $M_{I}(z_1\xi_3+z_2\xi_6)$  is reducible if and only if one of the following holds. 
	\begin{enumerate}
		\item $z_1\in \mathbb{C}$, $z_2\in 1-\min\{3, 4\}+ \mathbb{Z}_{\ge 0}=-2+\mathbb{Z}_{\ge0}$.
		\item $z_2\in \mathbb{C}$, $z_1\in 1-\min\{3, 3\}+ \mathbb{Z}_{\ge 0}=-2+\mathbb{Z}_{\ge0}$.
		\item  $(z_1,z_2)\in \mathbb{C}\times \mathbb{C}$, and
				 $z_1+z_2 \in-6+\mathbb{Z}_{>0}$.
		\item $z_1=z_2\in (-2+ \mathbb{Z}_{\ge 0})\bigcup (-\frac{5}{2}+\mathbb{Z}_{\ge 0})=-\frac{5}{2}+\frac{1}{2}\mathbb{Z}_{\ge 0}$.
	\end{enumerate}
 For this example, we can draw these reducible points in the following complex plane $\mathbb{C}^2$:

\hspace{1cm}

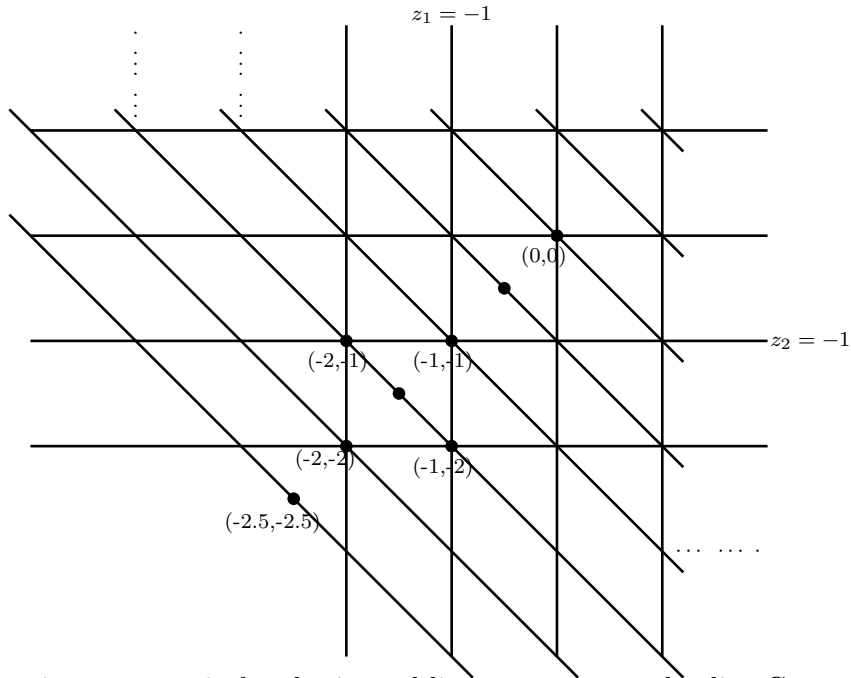
\begin{center}
{
\psset{unit=1.4cm}
   \tiny{ \begin{pspicture}(-4,-4)(2.5,2)
        \psset{linecolor=black}
        \cnode*(0,0){0.06}{a4}
        \cnode*(-2,-1){0.06}{a14}
        \cnode*(-1,-1){0.06}{a15}
        \cnode*(-2,-2){0.06}{b4}
        \cnode*(-1,-2){0.06}{b5}
        \cnode*(-2.5,-2.5){0.06}{b7}
        \cnode*(-1.5,-1.5){0.06}{b8}
        \cnode*(-0.5,-0.5){0.06}{b9}
        \psset{linecolor=black}
        \psline[linewidth=1pt](-5,1)(2,1)
        \psline[linewidth=1pt](-5,0)(2,0)
        \psline[linewidth=1pt](-5,-1)(2,-1)
        \psline[linewidth=1pt](-5,-2)(2,-2)
        \psline[linewidth=1pt](1,-4)(1,2)
        \psline[linewidth=1pt](0,-4)(0,2)
        \psline[linewidth=1pt](-1,-4)(-1,2)
        \psline[linewidth=1pt](-2,-4)(-2,2)
        \psline[linewidth=1pt](-5.2,0.2)(-0.8,-4.2)
        \psline[linewidth=1pt](-5.2,1.2)(0.2,-4.2)
        \psline[linewidth=1pt](-4.2,1.2)(1.2,-4.2)
        \psline[linewidth=1pt](-3.2,1.2)(1.2,-3.2)
        \psline[linewidth=1pt](-2.2,1.2)(1.2,-2.2)
        \psline[linewidth=1pt](-1.2,1.2)(1.2,-1.2)
        \psline[linewidth=1pt](-0.2,1.2)(1.2,-0.2)
        \psline[linewidth=1pt](0.8,1.2)(1.2,0.8)
        \uput[u](-1,1.9){\scriptsize{$z_1=-1$}}
        \uput[r](1.9,-1){\scriptsize{$z_2=-1$}}
 \uput[r](1.9,0)
 {\scriptsize{$z_2=0$}}

        \uput[r](-1.5,-1.2){\scriptsize(-1,-1)}
        \uput[r](-2.5,-1.2){\scriptsize(-2,-1)}
        \uput[r](-1.5,-2.2){\scriptsize(-1,-2)}
        \uput[r](-0.47,-0.2){\scriptsize(0,0)}
        \uput[d](-2.7,-2.5){\scriptsize(-2.5,-2.5)}
        \uput[d](-2.2,-1.9){\scriptsize(-2,-2)}
        \uput[u](-3,1){\scriptsize{$\vdots$}}
        \uput[u](-3,1.4){\scriptsize{$\vdots$}}
        \uput[u](-3,1.75){\scriptsize{$\cdot$}}
        \uput[u](-4,1){\scriptsize{$\vdots$}}
        \uput[u](-4,1.4){\scriptsize{$\vdots$}}
        \uput[u](-4,1.75){\scriptsize{$\cdot$}}
        \uput[r](1,-3){\scriptsize{$\cdots$}}
        \uput[r](1.4,-3){\scriptsize{$\cdots$}}
        \uput[r](1.75,-3){\scriptsize{$\cdot$}}

    \end{pspicture}}}
\end{center}
In the above picture, a vertical or horizontal line means a complex line $\mathbb{C}$.
\end{example}
\begin{example}
	Let $\frg=\frsl(11, \mathbb{C})$. We consider the parabolic subalgebra $\frq$ corresponding to the subset $\Pi\setminus\{\alpha_3, \alpha_9\}$.
	In this case $p=3$, $q-p=6$ and $n-q=2$. So the scalar generalized Verma module $M_{I}(z_1\xi_3+z_2\xi_9)$ is reducible if and only if one of the following
holds.
	\begin{enumerate}
		\item $z_1\in \mathbb{C}$, $z_2\in 1-\min\{6, 2\}+ \mathbb{Z}_{\ge 0}=-1+\mathbb{Z}_{\ge0}$.
		\item $z_2\in \mathbb{C}$, $z_1\in 1-\min\{3, 6\}+ \mathbb{Z}_{\ge 0}=-2+\mathbb{Z}_{\ge0}$.
		\item  $(z_1,z_2)\in \mathbb{C}\times \mathbb{C}$, and
				 $z_1+z_2 \in -8+\mathbb{Z}_{>0}$.
     \item $z_1=z_2\in (-4+ \mathbb{Z}_{\ge 0})\bigcup (-\frac{7}{2}+\mathbb{Z}_{\ge 0})=-4+\frac{1}{2}\mathbb{Z}_{\ge 0}$.
	\end{enumerate}

We can draw these reducible points in the following complex plane $\mathbb{C}^2$:

\hspace{1cm}

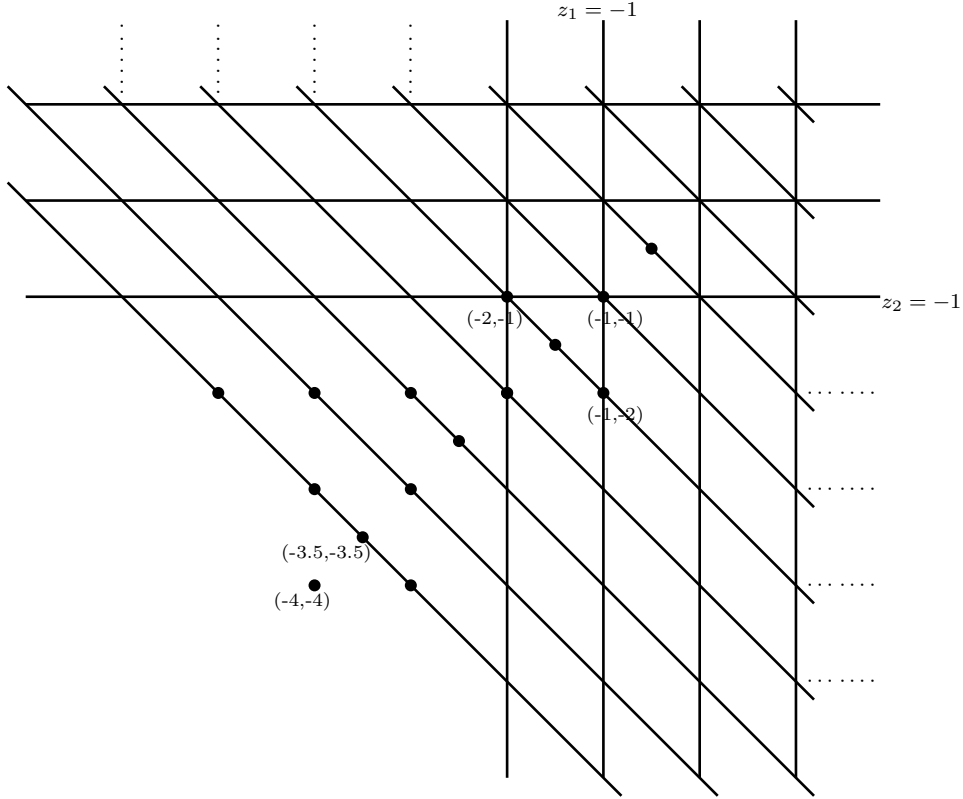
\begin{center}
\psset{unit=1.6cm}
		\begin{pspicture}(-5.6,-5.6)(1.7,1.7)
			\psset{linecolor=black}
			\cnode*(-1.6,-0.8){0.05}{a6}
			\cnode*(-0.8,-0.8){0.05}{a7}
			\cnode*(-4,-1.6){0.05}{b3}
			\cnode*(-3.2,-1.6){0.05}{b4}
			\cnode*(-2.4,-1.6){0.05}{b5}
			\cnode*(-1.6,-1.6){0.05}{b6}
			\cnode*(-0.8,-1.6){0.05}{b7}
            \cnode*(-3.2,-3.2){0.05}{b10}
            \cnode*(-2.8,-2.8){0.05}{b11}
            \cnode*(-2,-2){0.05}{b12}
            \cnode*(-1.6,-1.6){0.05}{b13}
            \cnode*(-1.2,-1.2){0.05}{b14}
            \cnode*(-0.4,-0.4){0.05}{b15}
			\cnode*(-3.2,-2.4){0.05}{c1}
			\cnode*(-2.4,-2.4){0.05}{c2}
			\cnode*(-2.4,-3.2){0.05}{c12}
			\psline[linewidth=1pt](-0.8,1.5)(-0.8,-4.8)
			\psline[linewidth=1pt](0,-4.8)(0,1.5)
			\psline[linewidth=1pt](0.8,-4.8)(0.8,1.5)
            \psline[linewidth=1pt](-1.6,1.5)(-1.6,-4.8)
			\psline[linewidth=1pt](-5.6,-0.8)(1.5,-0.8)
			\psline[linewidth=1pt](-5.6,0)(1.5,0)
			\psline[linewidth=1pt](-5.6,0.8)(1.5,0.8)
            \psline[linewidth=1pt](-5.75,0.15)(-0.65,-4.95)
            \psline[linewidth=1pt](-5.75,0.95)(0.15,-4.95)
            \psline[linewidth=1pt](-4.95,0.95)(0.95,-4.95)
            \psline[linewidth=1pt](-4.15,0.95)(0.95,-4.15)
            \psline[linewidth=1pt](-3.35,0.95)(0.95,-3.35)
            \psline[linewidth=1pt](-2.55,0.95)(0.95,-2.55)
            \psline[linewidth=1pt](-1.75,0.95)(0.95,-1.75)
            \psline[linewidth=1pt](-0.95,0.95)(0.95,-0.95)
            \psline[linewidth=1pt](-0.15,0.95)(0.95,-0.15)
            \psline[linewidth=1pt](0.65,0.95)(0.95,0.65)
        \uput[u](-0.85,1.4){\scriptsize{$z_1=-1$}}
        \uput[r](1.4,-0.85){\scriptsize{$z_2=-1$}}
        \uput[u](-2.4,0.78){\scriptsize{$\vdots$}}
        \uput[u](-2.4,1.05){\scriptsize{$\vdots$}}
        \uput[u](-2.4,1.3){\scriptsize{$\cdot$}}
        \uput[u](-3.2,0.78){\scriptsize{$\vdots$}}
        \uput[u](-3.2,1.05){\scriptsize{$\vdots$}}
        \uput[u](-3.2,1.3){\scriptsize{$\cdot$}}
        \uput[u](-4,0.78){\scriptsize{$\vdots$}}
        \uput[u](-4,1.05){\scriptsize{$\vdots$}}
        \uput[u](-4,1.3){\scriptsize{$\cdot$}}
        \uput[u](-4.8,0.78){\scriptsize{$\vdots$}}
        \uput[u](-4.8,1.05){\scriptsize{$\vdots$}}
        \uput[u](-4.8,1.3){\scriptsize{$\cdot$}}
     
        \uput[r](0.78,-1.6){\scriptsize{$\cdots$}}
        \uput[r](1.05,-1.6){\scriptsize{$\cdots$}}
        \uput[r](1.3,-1.6){\scriptsize{$\cdot$}}
        
        \uput[r](0.78,-2.4){\scriptsize{$\cdots$}}
        \uput[r](1.05,-2.4){\scriptsize{$\cdots$}}
        \uput[r](1.3,-2.4){\scriptsize{$\cdot$}}
        
        \uput[r](0.78,-3.2){\scriptsize{$\cdots$}}
        \uput[r](1.05,-3.2){\scriptsize{$\cdots$}}
        \uput[r](1.3,-3.2){\scriptsize{$\cdot$}}
        
        \uput[r](0.78,-4){\scriptsize{$\cdots$}}
        \uput[r](1.05,-4){\scriptsize{$\cdots$}}
        \uput[r](1.3,-4){\scriptsize{$\cdot$}}
			\uput[d](-0.7,-0.8){\tiny(-1,-1)}
			\uput[d](-1.7,-0.8){\tiny(-2,-1)}
			\uput[d](-0.7,-1.6){\tiny(-1,-2)}
   \uput[d](-3.3,-3.15){\tiny(-4,-4)}
	\uput[d](-3.1,-2.75){\tiny(-3.5,-3.5)}		
		\end{pspicture}
	\end{center}

 \end{example}

\section{Reducibility of scalar  generalized Verma modules for type $D_n$}
In this section, we write $n=2k$ if $n$ is an even integer and  $n=2k+1$ if $n$ is an odd integer.

First we recall a useful result which will be used in our proof. 

\begin{prop} [{\cite[Theorem 1.5 $\&$ Corollary 5.4]{BZ}}]\label{son-socular} 
Let $\mathfrak{g}=\mathfrak{so}(2n,\mathbb{C})$.
Suppose the parabolic subalgebra $\mathfrak{p}$ corresponds to two simple roots $\{\alpha_p, \alpha_{q}\}$ $(p,q\in \{1,n-1,n\})$.  A simple integral  module $L(\lambda)$ (with $\lambda=z_1\xi_p+z_2\xi_{q}$ being integral) has $\gkd L(\lambda)=\dim (\mathfrak{u})$ if and only if $p((\lambda+\rho)^-)^{\ev}$ has the same shape with $(2, \underbrace{1,\cdots,1}_{n-2})$.
\end{prop}

\subsection{Reducibility when $p=1$ and $q=n-1$ or $n$}

\begin{lem}\label{z1z2z} 
		Let $ \mathfrak{g}=\mathfrak{so}(2n,\mathbb{C}) $ with $n\geq 4$. $M_{I}(\lambda)$ is a scalar generalized Verma module with highest weight $\lambda=z_1 \xi_p+z_2\xi_q$ $(p,q\in \{1,n-1,n\})$.
 Suppose $p=1$ and $q=n-1$ or $n$.  
  When $z_1=z_2=z\in \mathbb{Z}$,  $M_{I}(z (\xi_p+\xi_q))$ is reducible if and only if
       \begin{center}
           $z\in \begin{cases}-n+3+\mathbb{Z}_{\ge 0},   &\textnormal{if}~n=2k+1,\\ -1+\mathbb{Z}_{\ge 0},   &\textnormal{if}~n=4,\\
           -n+4+\mathbb{Z}_{\ge 0},  &\textnormal{if}~n=2k\ge 6.
\end{cases}$
       \end{center}

\end{lem}

\begin{proof}
	
 Firstly, we consider the case $p=1,q=n-1$.

		Now we have  $\Delta^+(\frl) = \{\alpha_2,\dots,\alpha_{n-2},\alpha_{n}\}$, where $\alpha_i=e_i-e_{i+1}$ $(1\le i\le n-1)$ and  $\alpha_n=e_{n-1}+e_n$.
	Correspondingly we have
	$\xi_1=(1,0,\cdots,0)$ and
	$\xi_{n-1}=(\frac{1}{2},\frac{1}{2},\cdots,\frac{1}{2},-\frac{1}{2})$.
 
	Since 
	$\rho=(n-1,n-2,\dots,1,0)$, we have
	$$\lambda+\rho=(z_1+\frac{1}{2}z_2+n-1,\frac{1}{2}z_2+n-2,\cdots,\frac{1}{2}z_2+1,-\frac{1}{2}z_2).$$
	
	 When $z_1=z_2=z$,    we  have
	\begin{center}
		$\lambda+\rho=(\frac{3}{2}z+n-1,\frac{1}{2}z+n-2,\cdots,\frac{1}{2}z+1,-\frac{1}{2}z)$.
		
	\end{center}

	When $z\in \mathbb{Z}$, equivalently $\lambda+\rho$ is integral,  we have the following.
	\begin{enumerate}
	\item[1)] When $z\ge 0$,  we will have
\begin{center}
	    $p((\lambda+\rho)^-)^{\ev}=(1,\underbrace{0,1,0,1,0,1,\cdots,0,1}_{2n-2}).$
	\end{center}
Then by Lemma \ref{reducible} and Proposition \ref{son-socular}, $M_{I}(z (\xi_p+\xi_q))$ is reducible.	
  
  \item[2)] Suppose $n$ is odd. Then we have the follows.

		When $z=-n+3$, $(\lambda+\rho)^-=(-\frac{1}{2}n+\frac{7}{2},\cdots,\frac{1}{2}n-\frac{3}{2},-\frac{1}{2}n+\frac{3}{2},\cdots,\frac{1}{2}n-\frac{7}{2})$.
			Then we  have 
			$$p((\lambda+\rho)^-)^{\ev}=(2,\underbrace{1,1,1,1,\cdots,1}_{n-3},0,1).$$
					By Lemma \ref{reducible} and  Proposition \ref{son-socular},
 $M_I(z (\xi_1+\xi_{n-1}))$	is reducible.

     When $z=-n+2$, $(\lambda+\rho)^-=(-\frac{1}{2}n+2,\cdots,\frac{1}{2}n-1,-\frac{1}{2}n+1,\cdots,\frac{1}{2}n-2)$.
         Then we  have $$p((\lambda+\rho)^-)^{\ev}=(2,\underbrace{1,1,1,1,1,\cdots,1}_{n-2}).$$
          By Lemma \ref{reducible} and  Proposition \ref{son-socular},
 $M_I(z (\xi_1+\xi_{n-1}))$		is irreducible.

Therefore  $z=-n+3$ is the first reducible point when $n$ is odd.

		\item[3)] Suppose $n$ is even. Then we have the follows.
			
            
            
            
            
            
	
               When $n=4$, it is easy to find that
 the first reducible point of $M_I(z (\xi_1+\xi_{n-1}))$	is $z=-1$.

   When $z=-n+4$ ($n>4$),
			 $(\lambda+\rho)^-=(-\frac{1}{2}n+5,\cdots,\frac{1}{2}n-2,-\frac{1}{2}n+2,\cdots,\frac{1}{2}n-5)$.	Then we  have 
				\begin{center}
				   $p((\lambda+\rho)^-)^{\ev}=(2,\underbrace{1,1,1,1,\cdots,1}_{n-4},0,1,0,1).$ 
				\end{center}
			   By Lemma \ref{reducible} and  Proposition \ref{son-socular},
 $M_I(z (\xi_1+\xi_{n-1}))$	is reducible.
			
			When $z=-n+3$, $(\lambda+\rho)^-=(-\frac{1}{2}n+\frac{7}{2},\cdots,\frac{1}{2}n-\frac{3}{2},-\frac{1}{2}n+\frac{3}{2},\cdots,\frac{1}{2}n-\frac{7}{2})$.		Then we  have 
		\begin{center}
   $p((\lambda+\rho)^-)^{\ev}=(2,\underbrace{1,1,1,1,\cdots,1}_{n-3},1).$
			\end{center}
	   By Lemma \ref{reducible} and  Proposition \ref{son-socular},
 $M_I(z (\xi_1+\xi_{n-1}))$	is irreducible.

	Thus
	when $n$ is even, the first reducible point is:

 \begin{center}
$z=$    $\begin{cases}-1,  \qquad &\text{if}~n=4,\\ -n+4, \qquad &\text{if}~n\ge 6.
	\end{cases}$
 \end{center}

	\end{enumerate}


For the case $p=1$ and $q=n$,  the arguments are similar to the previous case. So we omit the details here.

\end{proof}

\begin{lem} 
		Let $ \mathfrak{g}=\mathfrak{so}(2n,\mathbb{C}) $ with $n\geq 4$. $M_{I}(\lambda)$ is a scalar generalized Verma module with highest weight $\lambda=z_1 \xi_p+z_2\xi_q$ $(p,q\in \{1,n-1,n\})$.
 Suppose $p=1$ and $q=n-1$ or $n$.  
  When $z_1=z_2=z\notin \mathbb{Z}$,  $M_{I}(z (\xi_p+\xi_q))$ is reducible if and only if
  $$  z\in  -\left\lceil {\frac {n}{2}}\right\rceil +\frac{3}{2}+\mathbb{Z}_{\ge 0}.$$   

\end{lem}

\begin{proof}
Similar to Lemma \ref{z1z2z}, we only need to consider the case $p=1,q=n-1$.

When $z\in \frac{1}{2}+\mathbb{Z}$, we  have $\lambda+\rho\in [\lambda]_3$.

 Denote $x:=\lambda+\rho$. Thus $\tilde{x}=(\frac{3}{2}z+n-1,-\frac{1}{2}z,-\frac{1}{2}z-1,-\frac{1}{2}z-2,\cdots,-\frac{1}{2}z-n+2)$. Then we have the following.
	\begin{enumerate}
		\item[1)] Suppose $n$ is odd. Then we have the follows.
		
   When $z> -\frac{1}{2}(n-1)$, we will have 
			$p(\widetilde{x})=(1,\underbrace{1,1,\cdots,1}_{n-1}).$
				Then by Proposition \ref{GKdim}, we have
 \begin{equation*}
\begin{split}
\GK\:L(z (\xi_1+\xi_{n-1}))&=n^2-
n-(1+2+3+4+\cdots+n-1)\\&=\frac{1}{2}n^2-\frac{1}{2}n<\dim(\mathfrak{u}).
\end{split}
\end{equation*} 
  So by Lemma \ref{reducible},  $ M_I(z (\xi_1+\xi_{n-1}))$	is reducible.
			
			When $z= -\frac{1}{2}n$, we will have 
			$p(\widetilde{x})=(2,\underbrace{1,1,\cdots,1}_{n-2}).$
			Then by Proposition \ref{GKdim}, we have
 $$\GK\:L(z (\xi_1+\xi_{n-1}))=\frac{1}{2}n^2+\frac{1}{2}n-1=\dim(\mathfrak{u}).$$ 
   So by Lemma \ref{reducible},  $M_I(z (\xi_1+\xi_{n-1}))$	is irreducible.

Therefore  $z=-\frac{n}{2}+1$ is the first reducible point of $M_I(z (\xi_1+\xi_{n-1}))$ for this case.

  \item[2)] Suppose $n$ is even. Then we have the follows.
		
   When $z> -\frac{1}{2}(n-1)$, similarly $M_I(z (\xi_1+\xi_{n-1}))$	is reducible.
			
			When $z= -\frac{1}{2}(n-1)$, we will have 
			$p(\widetilde{x})=(2,\underbrace{1,1,\cdots,1}_{n-2}).$
	Then by Proposition \ref{GKdim}, we have
 $$\GK\:L(z (\xi_1+\xi_{n-1}))=\frac{1}{2}n^2+\frac{1}{2}n-1=\dim(\mathfrak{u}).$$ 
 So  by Lemma \ref{reducible}, $M_I(z (\xi_1+\xi_{n-1}))$	is irreducible.

		Therefore  $z=-\frac{n}{2}+\frac{3}{2}$ is the first reducible point of $M_I(z (\xi_1+\xi_{n-1}))$ for this case.
	 
		\end{enumerate}

  When $z\notin \frac{1}{2}\mathbb{Z}$, we  have 
				\begin{center}
		    $x_1=(\frac{1}{2}z+n-2,\frac{1}{2}z+n-3\cdots,\frac{1}{2}z+1,-\frac{1}{2}z)$, $x_2=(\frac{3}{2}z+n-1)\in [\lambda]_3$.
		\end{center} 
		  So $p(\widetilde{x_1})=(1,\underbrace{1,1,\cdots,1}_{n-2})$ and $p(\widetilde{x_2})=(1).$ 
		  Then by Proposition \ref{GKdim}, we have 
 $$\GK\:L(z (\xi_1+\xi_{n-1}))=\frac{1}{2}n^2+\frac{1}{2}n-1=\dim(\mathfrak{u}).$$
   So by Lemma \ref{reducible},  $M_I(z (\xi_1+\xi_{n-1}))$	is irreducible.

    Therefore, when  $z\notin \mathbb{Z} $,
    $M_I(z (\xi_1+\xi_{n-1}))$	is reducible if and only if
    $z\in -\left\lceil {\frac {n}{2}}\right\rceil+\frac{3}{2}+\mathbb{Z}_{\ge 0}$.

\end{proof}
\begin{lem} 
		Let $ \mathfrak{g}=\mathfrak{so}(2n,\mathbb{C}) $ with $n\geq 4$. $M_{I}(\lambda)$ is a scalar generalized Verma module with highest weight $\lambda=z_1 \xi_p+z_2\xi_q$ $(p,q\in \{1,n-1,n\})$.
 Suppose $p=1$ and $q=n-1$ or $n$.  
  When $z_1, z_2 \in \mathbb{Z}$ and $z_1\neq z_2 $,  $M_{I}(z_1\xi_p+z_2\xi_q)$ is reducible if and only if one of the following holds:
 \begin{enumerate}
      \item $z_1>-1, z_2\in\mathbb{Z}$.
      \item $ z_1=-1,z_2\in -n+3+\mathbb{Z}_{\ge 0}$.
      \item $z_1<-1,z_2\in \begin{cases}
         -n+3+\mathbb{Z}_{\ge 0}, &\textnormal{if}~n=2k+1,\\
         -n+4+\mathbb{Z}_{\ge 0}, &\textnormal{if}~n=2k.
     \end{cases}$
     \end{enumerate}
 
\end{lem}

\begin{proof}
Similar to Lemma \ref{z1z2z}, we only need to consider the case $p=1,q=n-1$.

 When $z_1,z_2\in \mathbb{Z}$ and $z_1\neq z_2 $,  $\lambda+\rho$ is integral, then we have 
$$(\lambda+\rho)^-=(z_1+\frac{1}{2}z_2+n-1,\cdots,-\frac{1}{2}z_2,\frac{1}{2}z_2,\cdots,-z_1-\frac{1}{2}z_2-n+1).$$

\begin{enumerate}
	\item[1)] When $z_1>-1,z_2>-1$, we will have 
	
	\begin{center}
	    $p((\lambda+\rho)^-)^{\ev}=(\underbrace{1,0,1,0,\cdots,1,0}_{2n-2},1).$
	\end{center}
Then by Lemma \ref{reducible} and Proposition \ref{son-socular}, $M_{I}(z_1\xi_p+z_2\xi_q)$ is reducible.

	 When $z_1>-1,z_2=-1$, we will have 
	\begin{center}
	    $p((\lambda+\rho)^-)^{\ev}=(1,1,\underbrace{1,0,1,0,\cdots,1,0}_{2n-4}).$
	\end{center}
Then by Lemma \ref{reducible} and Proposition \ref{son-socular}, $M_{I}(z_1\xi_p+z_2\xi_q)$ is reducible.		
	
	 When $z_1>-1,z_2<-1$, we will have 
	
	\begin{center}
	    $\begin{cases}
		p((\lambda+\rho)^-)^{\ev}=(\underbrace{1,1,1,1,\cdots,1}_{n-1},0,1), \quad&\text{if}~ n=2k,
		\\ p((\lambda+\rho)^-)^{\ev}=(\underbrace{1,1,1,1,\cdots,1}_{n-1},1),
		\qquad&\text{if}~ n=2k+1.
	\end{cases}$
	\end{center}
Then by Lemma \ref{reducible} and Proposition \ref{son-socular}, $M_{I}(z_1\xi_p+z_2\xi_q)$ is reducible.		
	
	\item[2)] When $z_1\le -1,z_2>-1$, we will have 
	
	$p((\lambda+\rho)^-)^{\ev}=(1,1,\underbrace{1,0,1,0,\cdots,1,0}_{2n-6},1)$ 
	or $p((\lambda+\rho)^-)^{\ev}=(2,1,\underbrace{1,0,1,0,\cdots,1,0}_{2n-6},1).$
Then by Lemma \ref{reducible} and Proposition \ref{son-socular}, $M_{I}(z_1\xi_p+z_2\xi_q)$ is reducible.


		\item[3)] When $z_1=-1, z_2< -1$, we will have the following.
		
		When $z_2=-2$, we will have  
			\begin{center}
			    $\begin{cases}
				p((\lambda+\rho)^-)^{\ev}=(2,1), &\text{if}~n=3,\\p((\lambda+\rho)^-)^{\ev}=(2,1,1), &\text{if}~n=4, \\p((\lambda+\rho)^-)^{\ev}=(1,1,1,1,\underbrace{1,0,1,0\cdots,1}_{2n-9}),  &\text{if}~n\ge 5.
			\end{cases}$
			\end{center}
	Then by Lemma \ref{reducible} and Proposition \ref{son-socular}, we can get that $$\begin{cases}
			M_I(z_1\xi_1+z_2\xi_{n-1})~	\text{is} ~ \text {irreducible}, &~\text{if}~n=3,4,\\M_I(z_1\xi_1+z_2\xi_{n-1})~	\text{is} ~ \text {reducible}, &~\text{if}~n\ge 5.
		\end{cases}$$
  
		 When $z_2=-n+3$, we will have	
		\begin{center}
	$p((\lambda+\rho)^-)^{\ev}=(\underbrace{1,1,1,1,\cdots,1,1}_{n}).$
		\end{center}
Then by Lemma \ref{reducible} and Proposition \ref{son-socular}, $M_{I}(z_1\xi_p+z_2\xi_q)$ is reducible.		
  
		 When $z_2=-n+2$, we will have
		\begin{center}
  $p((\lambda+\rho)^-)^{\ev}=(2,\underbrace{1,1,1,\cdots,1}_{n-2}).$
		\end{center}
Then by Lemma \ref{reducible} and Proposition \ref{son-socular}, $M_{I}(z_1\xi_p+z_2\xi_q)$ is 	irreducible.

  Thus  $(z_1,z_2)=(-1,z_2)$ is a reducible point of $M_{I}(z_1\xi_p+z_2\xi_q)$ if and only if $z_2\in -n+3+\mathbb{Z}_{\ge 0}$.

		\item[4)] When $z_1<-1, z_2=-1$, we will have 
		  \begin{center}
		      $\begin{cases}
		  	p((\lambda+\rho)^-)^{\ev}=(2,1),\quad &\text{if}~n=3,\\p((\lambda+\rho)^-)^{\ev}=(2,1,1),\qquad &\text{if}~n=4, \\p((\lambda+\rho)^-)^{\ev}=(2,1,\underbrace{1,0,1,0\cdots,1}_{2n-5}), \quad &\text{if}~n\ge 5.
		  \end{cases}$
		  \end{center}
Then by Lemma \ref{reducible} and Proposition \ref{son-socular},
  we can get that $$\begin{cases}
			M_I(z_1\xi_1+z_2\xi_{n-1})~	\text{is} ~ \text {irreducible}, &~\text{if}~n=3,4,\\M_I(z_1\xi_1+z_2\xi_{n-1})~	\text{is} ~ \text {reducible}, &~\text{if}~n\ge 5.
		\end{cases}$$
		
		\item[5)] When $z_1<-1,z_2< -1$ and $n$ is odd, we will have the following.
		
		 When $(z_1,z_2)=(-d,-n+3)$ for some $d\in \mathbb{Z}_{>1}$, we will have 
			\begin{center}
			$p((\lambda+\rho)^-)^{\ev}=(2,\underbrace{1,1,1,\cdots,1}_{n-3},0,1).$
			\end{center}
Then by Lemma \ref{reducible} and Proposition \ref{son-socular}, $M_I(z_1 \xi_1+z_2\xi_{n-1})$	is reducible.
   
			When $(z_1,z_2)=(-d,-n+2)$, we will have 
			\begin{center}
			    $p((\lambda+\rho)^-)^{\ev}=(2,\underbrace{1,1,1,\cdots,1}_{n-2}).$
			\end{center}
Then by Lemma \ref{reducible} and Proposition \ref{son-socular}, $M_I(z_1 \xi_1+z_2\xi_{n-1})$	is irreducible.

  Thus when $z_1<-1,z_2< -1$ and $n$ is odd,  $(z_1,z_2)$ is a reducible point of $M_{I}(z_1\xi_p+z_2\xi_q)$ if and only if $z_2\in -n+3+\mathbb{Z}_{\ge 0}$.

	   \item[6)] When $z_1<-1,z_2< -1$ and $n$ is even, similarly we can find that $(z_1,z_2)$ is a reducible point of $M_{I}(z_1\xi_p+z_2\xi_q)$ if and only if $z_2\in -n+4+\mathbb{Z}_{\ge 0}$.

\end{enumerate}

This finishes the proof.


\end{proof}

\begin{lem} 
		Let $ \mathfrak{g}=\mathfrak{so}(2n,\mathbb{C}) $ with $n\geq 4$. $M_{I}(\lambda)$ is a scalar generalized Verma module with highest weight $\lambda=z_1 \xi_p+z_2\xi_q$ $(p,q\in \{1,n-1,n\})$.
 Suppose $p=1$ and $q=n-1$ or $n$.  
  When $z_1\notin \mathbb{Z}$ or $ z_2 \notin \mathbb{Z}$ and $z_1\neq z_2 $,  $M_{I}(z_1\xi_p+z_2\xi_q)$ is reducible if and only if
      \begin{enumerate}
      \item $ z_1\in \mathbb{Z}_{\ge 0}, z_2\notin \mathbb{Z}$.

      \item  $z_1\notin \mathbb{Z}$, 
      $ z_2\in \begin{cases}
         -n+3+\mathbb{Z}_{\ge 0}, &\textnormal{if}~n=2k+1,\\
         -n+4+\mathbb{Z}_{\ge 0}, &\textnormal{if}~n=2k.
     \end{cases}$
          \item $ z_1+z_2\in -n+2+\mathbb{Z}_{\ge 0}    \textnormal{~if}~z_1,z_2\notin \mathbb{Z}$.

  \end{enumerate}
   
\end{lem}

\begin{proof}
Similar to Lemma \ref{z1z2z}, we only need to consider the case $p=1,q=n-1$.

 When $z_1\in \mathbb{Z}$ and $z_2\notin \mathbb{Z}$, we  have 
    	$\lambda+\rho \in [\lambda]_3$.
     Denote $x:=\lambda+\rho$. Thus $$\tilde{x}=(z_1+\frac{1}{2}z_2+n-1,\frac{1}{2}z_2+n-2,\cdots,\frac{1}{2}z_2+1,\frac{1}{2}z_2).$$ Then we have the following.	
     
     \begin{enumerate}
    		\item[1)] When $z_1>-1$, we will have $p(\widetilde{x})=(1,\underbrace{1,1,\cdots,1}_{n-1}).$
    			Then by Proposition \ref{GKdim},
 $$\GK\:L(z_1 \xi_1+z_2\xi_{n-1})=\frac{1}{2}n^2-\frac{1}{2}n<\dim(\mathfrak{u}).$$ 
   So by Lemma \ref{reducible},  $M_I(z_1 \xi_1+z_2\xi_{n-1})$	is reducible.
    		
    		\item[2)] When $z_1=-1$, we will have $p(\widetilde{x})=(2,\underbrace{1,1,\cdots,1}_{n-2}).$
    			Then by Proposition \ref{GKdim},
 $$\GK\:L(z_1 \xi_1+z_2\xi_{n-1})=\frac{1}{2}n^2+\frac{1}{2}n-1=\dim(\mathfrak{u}).$$
  So by Lemma \ref{reducible},  $M_I(z_1 \xi_1+z_2\xi_{n-1})$	is irreducible.
    	\end{enumerate}
   
     Thus $M_I(z_1 \xi_1+z_2\xi_{n-1})$ is reducible if and only if $z_1\in \mathbb{Z}_{\geq 0}$ when $z_1\in \mathbb{Z},z_2\notin \mathbb{Z}$.
    	
    	 When $z_2\in \mathbb{Z}, z_1\notin \mathbb{Z}$, we will have the following. 
    	
    	\begin{enumerate}
    		\item[1)] When $z_1\in \frac{1}{2}+\mathbb{Z}$, we will have  $[\lambda]_3=\emptyset$.
    	   \begin{enumerate}
    	   	\item When $z_2>-1$, we will  have
    	   	$p((\lambda+\rho)^-_{(0)})^{\ev}+ p((\lambda+\rho)^-_{(\frac{1}{2})})^{\ev}=(2,\underbrace{0,1,\cdots,0,1}_{2n-4})$. 
    	     	Then by Proposition \ref{GKdim},
 \begin{equation*}
\begin{split}
\GK\:L(z_1 \xi_1+z_2\xi_{n-1})&=n^2-
n-(2+4+6+\cdots+2n-4)\\&=2n-2<\dim(\mathfrak{u}).
\end{split}
\end{equation*} 
  So by Lemma \ref{reducible},  $M_I(z_1 \xi_1+z_2\xi_{n-1})$	is reducible.
    	   	
    	   	\item When $z_2=-n+4$, we will have 
    	  \begin{center}  	
    	   	     $\begin{cases}
    	   	 	p((\lambda+\rho)^-_{(0)})^{\ev}+ p((\lambda+\rho)^-_{(\frac{1}{2})})^{\ev}=(2,\underbrace{1,1,\cdots,1,1}_{n-4},0,1,0,1), \quad &\text{if}~n=2k,\\ p((\lambda+\rho)^-_{(0)})^{\ev}+ p((\lambda+\rho)^-_{(\frac{1}{2})})^{\ev}=(2,\underbrace{1,1,\cdots,1,1}_{n-4},1,0,1), \quad &\text{if}~n=2k+1.
    	   	 \end{cases}$
    	   	 \end{center} 
            Then by Proposition \ref{GKdim},        when $n=2k$, we have
 $$\GK\:L(z_1 \xi_1+z_2\xi_{n-1})=\frac{1}{2}n^2+\frac{1}{2}n-4<\dim(\mathfrak{u}),$$  and 
when $n=2k+1$, we have
$$\GK\:L(z_1 \xi_1+z_2\xi_{n-1})=\frac{1}{2}n^2+\frac{1}{2}n-2<\dim(\mathfrak{u}).$$ So by Lemma \ref{reducible},  $M_I(z_1 \xi_1+z_2\xi_{n-1})$	is reducible.
    	   	 
    	   	\item When $z_2=-n+3$, we will have
    	 \begin{center}   	
    	   	$\begin{cases}
    	   		p((\lambda+\rho)^-_{(0)})^{\ev}+ p((\lambda+\rho)^-_{(\frac{1}{2})})^{\ev}=(2,\underbrace{1,1,\cdots,1,1}_{n-3},1), \quad &\text{if}~n=2k,\\ p((\lambda+\rho)^-_{(0)})^{\ev}+ p((\lambda+\rho)^-_{(\frac{1}{2})})^{\ev}=(2,\underbrace{1,1,\cdots,1,1}_{n-3},0,1), \quad &\text{if}~n=2k+1.
    	   	\end{cases}$
    	 \end{center}   	
    	   Then by Proposition \ref{GKdim},
          when $n=2k$, we have
 $$\GK\:L(z_1 \xi_1+z_2\xi_{n-1})=\frac{1}{2}n^2+\frac{1}{2}n-1=\dim(\mathfrak{u}),$$ and 
when $n=2k+1$, we have
$$\GK\:L(z_1 \xi_1+z_2\xi_{n-1})=\frac{1}{2}n^2+\frac{1}{2}n-2<\dim(\mathfrak{u}).$$

So by Lemma \ref{reducible},  we can get that
    	   	    $$\begin{cases}
    	M_I(z_1\xi_1+z_2\xi_{n-1}) ~	\text{is}~ \text{irreducible},  &~\text{if}~n=2k,\\
          M_I(z_1\xi_1+z_2\xi_{n-1}) ~	\text{is}~ \text{reducible},  &~\text{if}~n=2k+1.
    	   	\end{cases}$$
       
    	   	\item When $z_2=-n+2$, we have
    	   	$p((\lambda+\rho)^-_{(0)})^{\ev}+ p((\lambda+\rho)^-_{(\frac{1}{2})})^{\ev}=(2,\underbrace{1,1,\cdots,1,1}_{n-2})$. 
    	   	Then by Proposition \ref{GKdim},
 $$\GK\:L(z_1 \xi_1+z_2\xi_{n-1})=\frac{1}{2}n^2+\frac{1}{2}n-1=\dim(\mathfrak{u}).$$ 
 So by Lemma \ref{reducible},  $M_I(z_1 \xi_1+z_2\xi_{n-1})$	is irreducible.
    	   	
    	   \end{enumerate}
    
      Thus when $z_1\in \frac{1}{2}+\mathbb{Z}, z_2\in \mathbb{Z}$, $M_I(z_1 \xi_1+z_2\xi_{n-1})$ is reducible if and only if  \begin{center}
        $z_2=$  $\begin{cases}
    			-n+4, \quad &\text{if}~n=2k,\\ -n+3, \quad &\text{if}~n=2k+1.
    		\end{cases}$
      \end{center}

    		\item[2)] When $z_1\notin \frac{1}{2}+\mathbb{Z}$, we will have
    		
    		  \begin{center}
    		      $(z_1+\frac{1}{2}z_2+n-1)\in [\lambda]_3$,
    		
    		$(\frac{1}{2}z_2+n-2,\cdots,\frac{1}{2}z_2+1,-\frac{1}{2}z_2,\frac{1}{2}z_2,\cdots,-\frac{1}{2}z_2-n+2)\in [\lambda]_{1,2}$.
    		  \end{center}
    		
    		Then the result is the   same with the case for $z_1\in \frac{1}{2}+\mathbb{Z}$.

    	\end{enumerate}
    	So when $z_1\notin \mathbb{Z}$, and $ z_2\in \mathbb{Z}$, $M_I(z_1 \xi_1+z_2\xi_{n-1})$ is reducible if and only if      
     \begin{center}
       $z_2\in$  $\begin{cases}
    		 -n+4+\mathbb{Z}_{\ge0}, \quad &\text{if}~n=2k,\\  -n+3+\mathbb{Z}_{\ge 0}, \quad &\text{if}~n=2k+1.
    	\end{cases}$
     \end{center}

  	 When $z_1\notin \mathbb{Z}$ and $z_2\notin \mathbb{Z}$, we will have the following.

    	\begin{enumerate}
            \item[1)] When $z_1+z_2\in \mathbb{Z}$, we will have
    		$\lambda +\rho \in [\lambda]_3$.
            Denote $x:=\lambda+\rho$. Thus $$\tilde{x}=(z_1+\frac{1}{2}z_2+n-1,-\frac{1}{2}z_2, -\frac{1}{2}z_2-1,\cdots,-\frac{1}{2}z_2-n+2).$$ Then we have the following:			
    		\begin{center}
    		     $\begin{cases}
    			p(\widetilde{x})=(\underbrace{1,1,\cdots,1}_{n}),  &~\text{if}~z_1+z_2>-n+1,\\
    			p(\widetilde{x})=(2,\underbrace{1,1,\cdots,1}_{n-2}),  &~\text{if}~z_1+z_2\le -n+1.
    		\end{cases}$
    	
    		\end{center}
       Then by Proposition \ref{GKdim},   when $z_1+z_2>-n+1$, we have
 $$\GK\:L(z_1 \xi_1+z_2\xi_{n-1})=\frac{1}{2}n^2+\frac{1}{2}n-2<\dim(\mathfrak{u}),$$ and 
when $z_1+z_2\le -n+1$, we have
 $$\GK\:L(z_1 \xi_1+z_2\xi_{n-1})=\frac{1}{2}n^2+\frac{1}{2}n-1=\dim(\mathfrak{u}).$$
    So by Lemma \ref{reducible},  $M_I(z_1\xi_1+z_2\xi_{n-1})$	is reducible if and only if $z_1+z_2>-n+1$ when $z_1+z_2\in \mathbb{Z}$.

    		\item[2)] When
      $z_1+z_2\notin \mathbb{Z}$ and $z_1+\frac{1}{2}z_2\in \frac{1}{2}\mathbb{Z}$, we will have 
    			$$(z_1+\frac{1}{2}z_2+n-1)\in [\lambda]_{1,2} \text{~and~} 
    		x:=(\frac{1}{2}z_2+n-2,\cdots,\frac{1}{2}z_2+1,-\frac{1}{2}z_2)\in [\lambda]_3.$$
    		Thus $p(\widetilde{x})=(\underbrace{1,1,\cdots,1}_{n-1})$.
    		Then by Proposition \ref{GKdim},      $$\GK\:L(z_1 \xi_1+z_2\xi_{n-1})=\frac{1}{2}n^2+\frac{1}{2}n-1=\dim(\mathfrak{u}).$$ 
 So by Lemma \ref{reducible},  $M_I(z_1\xi_1+z_2\xi_{n-1})$	is irreducible.

    		\item[3)] When $z_1+z_2\notin \mathbb{Z}$ and $z_1+\frac{1}{2}z_2\notin \frac{1}{2}\mathbb{Z}$, we will have 
    		$$(z_1+\frac{1}{2}z_2+n-1)\text{~and}~ x:=(\frac{1}{2}z_2+n-2,\cdots,\frac{1}{2}z_2+1,-\frac{1}{2}z_2)\in [\lambda]_3.$$
    			 Thus $p(\widetilde{x})=(2,\underbrace{1,1,\cdots,1}_{n-2})$.
    			Then by Proposition \ref{GKdim},
 $$\GK\:L(z_1 \xi_1+z_2\xi_{n-1})=\frac{1}{2}n^2+\frac{1}{2}n-1=\dim(\mathfrak{u}).$$ 
 So  by Lemma \ref{reducible},  $M_I(z_1\xi_1+z_2\xi_{n-1})$	is irreducible.
    	\end{enumerate}
  
    Thus we can conclude that
    when $z_1\notin \mathbb{Z}$ and $z_2\notin \mathbb{Z}$, $M_I(z_1 \xi_1+z_2\xi_{n-1})$ is reducible if and only if         
    \begin{center}
       $z_1+z_2\in -n+2+\mathbb{Z}_{\ge 0}$. 
    \end{center}


This finishes the proof.
\end{proof}

Combined the above lemmas, we complete the proof of  part (1) in Theorem \ref{son-thm}.

\subsection{Reducibility when $p=n-1$ and $q=n$}

\begin{lem}\label{z1z2-so}
		Let $ \mathfrak{g}=\mathfrak{so}(2n,\mathbb{C}) $ with $n\geq 4$. $M_{I}(\lambda)$ is a scalar generalized Verma module with highest weight $\lambda=z_1\xi_{n-1}+z_2\xi_{n}$.
 When $z_1=z_2=z$,  $M_{I}(z (\xi_{n-1}+\xi_n))$ is reducible if and only if 
 $$z\in \begin{cases}
        -\frac{n}{2}+\frac{1}{2}+\frac{1}{2}\mathbb{Z}_{\ge 0}, &\textnormal{if}~n=2k+1,\\-\frac{n}{2}+1+\frac{1}{2}\mathbb{Z}_{\ge 0},&\textnormal{if}~n=2k.
    \end{cases} $$

\end{lem}

\begin{proof}
	
	We have  $\Delta^+(\frl) = \{\alpha_1,\alpha_2,\cdots,\alpha_{n-2}\}$, where $\alpha_i=e_i-e_{i+1}$ $(1\le i\le n-1)$ and  $\alpha_n=e_{n-1}+e_n$.
	Correspondingly we have
	$\xi_{n-1}=(\frac{1}{2},\frac{1}{2},\cdots,\frac{1}{2},-\frac{1}{2})$,
	$\xi_{n}=(\frac{1}{2},\frac{1}{2},\cdots,\frac{1}{2},\frac{1}{2})$ and 
	$$\lambda+\rho=(\frac{1}{2}(z_1+z_2)+n-1,\frac{1}{2}(z_1+z_2)+n-2,\cdots,\frac{1}{2}(z_1+z_2)+1,-\frac{1}{2}z_1+\frac{1}{2}z_2).$$
 
	 When  $z_1=z_2=z\in \mathbb{Z}$, equivalently $\lambda+\rho$ is integral, we have
	\begin{center}
		$\lambda+\rho=(z+n-1,z+n-2,\cdots,z+1,0)$.
	\end{center}
Then  we have the following.
	
	\begin{enumerate}
		\item[1)] Suppose $n$ is odd. Then we have the follows.
		
   When $z>-1$, we will have
			$$p((\lambda+\rho)^-)^{\ev}=(1,\underbrace{0,1,0,1\cdots,1}_{2n-2}).$$
 By Lemma \ref{reducible} and  Proposition \ref{son-socular},
 $M_I(z (\xi_{n-1}+\xi_{n}))$	is reducible.
			
			 When $z=-\lfloor \frac{n}{2} \rfloor=-\frac{n-1}{2}=-\frac{n}{2}+\frac{1}{2}$, we will have 
			$$p((\lambda+\rho)^-)^{\ev}=(2,\underbrace{1,1,\cdots,1}_{n-3},0,1).$$
				By Lemma \ref{reducible} and  Proposition \ref{son-socular},
 $M_I(z (\xi_{n-1}+\xi_{n}))$	is reducible.
			
			When $z=-\lfloor \frac{n}{2} \rfloor-1=-\frac{n-1}{2}-1$, we will have 
			$$p((\lambda+\rho)^-)^{\ev}=(2,\underbrace{1,1,\cdots,1}_{n-2}).$$
		By Lemma \ref{reducible} and  Proposition \ref{son-socular},
 $M_I(z (\xi_{n-1}+\xi_{n}))$	is irreducible.
		
		\item[2)] Suppose $n$ is even. Then we have the follows.
		
   When $z>-1$, it is the same with the case for $n$ being odd.
        So $M_I(z (\xi_1+\xi_{n-1}))$	is reducible.
			
			When  $z=-\frac{n}{2}+1$, we will have 
			$$p((\lambda+\rho)^-)^{\ev}=(2,\underbrace{1,1,\cdots,1}_{n-4},0,1,0,1).$$
By Lemma \ref{reducible} and  Proposition \ref{son-socular},
 $M_I(z (\xi_{n-1}+\xi_{n}))$	is reducible.
			
			When $z=-\frac{n}{2}$, we will have 
			$$p((\lambda+\rho)^-)^{\ev}=(2,\underbrace{1,1,\cdots,1}_{n-2}).$$
	By Lemma \ref{reducible} and  Proposition \ref{son-socular},
 $M_I(z (\xi_{n-1}+\xi_{n}))$	is irreducible.
	\end{enumerate}
 
	Thus when $z\in \mathbb{Z}$, the first reducible point of $M_I(z (\xi_{n-1}+\xi_{n}))$ is:
 \begin{center}
	  $z=\begin{cases}
		-\frac{n}{2}+\frac{1}{2},   &~\text{if}~n=2k+1,	\\
  -\frac{n}{2}+1,  &~\text{if}~n=2k.
	\end{cases}$
	\end{center}

 When $z\notin \mathbb{Z}$,  we have the following.
 
	\begin{enumerate}
	\item[1)] When $z\notin \frac{1}{2}+\mathbb{Z}$, we will have 
		
		\begin{center}
		    $(\lambda+\rho)_{(0)}=(0)$, $(\lambda+\rho)_{(\frac{1}{2})}=\emptyset$ and $x:=(z+n-1,z+n-2,\cdots,z+1)\in [\lambda]_3$.

		\end{center}
So $p(\widetilde{x})=(\underbrace{1,1,\cdots,1}_{n-1})$.
		Then by Proposition \ref{GKdim},
  $$\GK\:L(z(\xi_{n-1}+\xi_{n}))=\frac{1}{2}n^2+\frac{1}{2}n-1=\dim(\mathfrak{u}).$$
So by Lemma \ref{reducible},  $M_I(z(\xi_{n-1}+\xi_{n}))$	is irreducible.

  \item[2)] When $z\in \frac{1}{2}+\mathbb{Z}$, we will have 
	\begin{center}
		 $[\lambda]_3=\emptyset$,   $(\lambda+\rho)_{(0)}=(0)$ and $x:=(z+n-1,z+n-2,\cdots,z+1)=(\lambda+\rho)_{(\frac{1}{2})}$.
      \end{center}

			\begin{enumerate}
				\item Suppose $n$ is odd. Then we have the follows.
				
     When $z>-1$, we will have
					\begin{center}
				   $p((\lambda+\rho)^-_{(\frac{1}{2})})^{\ev}=(\underbrace{1,0,1,0,\cdots,0,1}_{2n-3})$. 
					\end{center}
			Then by Proposition \ref{GKdim},
 \begin{equation*}
\begin{split}
\GK\:L(z(\xi_{n-1}+\xi_{n}))&=n^2-
n-(2+4+\cdots+2n-4)\\&=2n-2<\dim(\mathfrak{u}).
\end{split}
\end{equation*}
So by Lemma \ref{reducible},  $ M_I(z(\xi_{n-1}+\xi_{n}))$	is reducible.
					
					When $z=-\frac{n}{2}+1$, we will have
					\begin{center}
					 $p((\lambda+\rho)^-_{(\frac{1}{2})})^{\ev}=(\underbrace{1,1,1,\cdots,1}_{n-2},0,1)$.
					\end{center}
	Then by Proposition \ref{GKdim},
 $$\GK\:L(z(\xi_{n-1}+\xi_{n}))=\frac{1}{2}n^2+\frac{1}{2}n-2<\dim(\mathfrak{u}).$$ 
 So by Lemma \ref{reducible},  $M_I(z(\xi_{n-1}+\xi_{n}))$	is reducible.
					
					When $z=-\frac{n}{2}$, we will have				
					\begin{center}
					    $p((\lambda+\rho)^-_{(\frac{1}{2})})^{\ev}=(\underbrace{1,1,\cdots,1}_{n-1})$.
					\end{center}
			Then by Proposition \ref{GKdim},
 $\GK\:L(z(\xi_{n-1}+\xi_{n}))=\frac{1}{2}n^2+\frac{1}{2}n-1=\dim(\mathfrak{u}).$ 
 
So by Lemma \ref{reducible},  $M_I(z(\xi_{n-1}+\xi_{n}))$	is irreducible.

				\item Suppose $n$ is even. Then we have the follows.
				
     When $z>-1$, it is the   same with the case for $n$ being odd, so  $M_I(z(\xi_{n-1}+\xi_{n}))$ is reducible.
					
					When $z=-\frac{n}{2}+\frac{3}{2}$, we will have
					
					\begin{center}
					    $p((\lambda+\rho)^-_{(\frac{1}{2})})^{\ev}=(\underbrace{1,1,1,\cdots,1}_{n-3},0,1,0,1)$.
					\end{center}
			Then by Proposition \ref{GKdim},
 \begin{equation*}
\begin{split}
\GK\:L(z(\xi_{n-1}+\xi_{n}))&=n^2-
n-(1+2+3+\cdots+n-4+n-2+n)\\&=\frac{1}{2}n^2+\frac{1}{2}n-4<\dim(\mathfrak{u}).
\end{split}
\end{equation*}
So by Lemma \ref{reducible},  $ M_I(z(\xi_{n-1}+\xi_{n}))$	is  reducible.
					
					When $z=-\frac{n}{2}+\frac{1}{2}$, we will have
					
					\begin{center}
					    $p((\lambda+\rho)^-_{(\frac{1}{2})})^{\ev}=(\underbrace{1,1,\cdots,1}_{n-2},1)$.
					\end{center}
		Then by Proposition \ref{GKdim},
 $$\GK\:L(z(\xi_{n-1}+\xi_{n}))=\frac{1}{2}n^2+\frac{1}{2}n-1=\dim(\mathfrak{u}).$$
So by Lemma \ref{reducible},   $M_I(z(\xi_{n-1}+\xi_{n}))$	is irreducible.
					
				\end{enumerate}

 Thus when $z\in \frac{1}{2}+\mathbb{Z}$, the first reducible point of $M_I(z(\xi_{n-1}+\xi_{n}))$ is:
  \begin{center}
      $z=$ 
			$\begin{cases}
				-\frac{n}{2}+1,  &~\text{if}~n=2k+1,\\-\frac{n}{2}+\frac{3}{2}, &~\text{if}~n=2k.
			\end{cases}$ 
   \end{center}

	\end{enumerate}

This finishes the proof.

\end{proof}

\begin{lem} 
	Keep notations as above.
 When $z_1, z_2 \in \mathbb{Z}$ and $z_1\neq z_2 $,  $M_{I}(z_1 \xi_{n-1}+z_2\xi_{n})$ is reducible if and only if

\begin{enumerate}
    \item $ z_1\in \mathbb{Z}_{\ge 0}, z_2\in \mathbb{Z}$.
    \item $z_2\in \mathbb{Z}_{\ge 0}, z_1\in \mathbb{Z}$.

    \item $z_1+z_2\in \begin{cases}
          -n+1+\mathbb{Z}_{\ge 0}, &\textnormal{if}~n=2k+1,\\
           -n+2+\mathbb{Z}_{\ge 0}, &\textnormal{if}~n=2k.
       \end{cases} $
\end{enumerate}

\end{lem}

\begin{proof}
 When $z_1\neq z_2$, we will have	
	$$\lambda+\rho=(\frac{1}{2}(z_1+z_2)+n-1,\frac{1}{2}(z_1+z_2)+n-2,\cdots,\frac{1}{2}(z_1+z_2)+1,-\frac{1}{2}z_1+\frac{1}{2}z_2).$$

	When $z_1$  and $z_2\in \mathbb{Z}$, equivalently $\lambda+\rho$ is integral, we have the following.
	
	\begin{enumerate}
		\item[1)] When $\max \{z_1,z_2\}<-1$, we will have the following.
		\begin{enumerate}
			\item Suppose $n$ is odd. Then we have the follows.
			
    When $z_1+z_2=-n+1$, we will have 
		$$p((\lambda+\rho)^-)^{\ev}=(2,\underbrace{1,1,\cdots,1}_{n-3},0,1).$$
	By Lemma \ref{reducible} and Proposition \ref{son-socular},
 $M_I(z_1\xi_{n-1}+z_2\xi_{n})$	is reducible.

				When $z_1+z_2=-n$, we will have $$p((\lambda+\rho)^-)^{\ev}=(2,\underbrace{1,1,\cdots,1}_{n-2}).$$
By Lemma \ref{reducible} and Proposition \ref{son-socular},
 $M_I(z_1\xi_{n-1}+z_2\xi_{n})$		is irreducible.

		\item Suppose $n$ is even. Then we have the follows. 
		
   When $z_1+z_2=-n+2$, we will have
		$$p((\lambda+\rho)^-)^{\ev}=(2,\underbrace{1,1,\cdots,1}_{n-4},0,1,0,1).$$
	By Lemma \ref{reducible} and Proposition \ref{son-socular},
 $M_I(z_1\xi_{n-1}+z_2\xi_{n})$		is reducible.

		When  $z_1+z_2=-n+1$, we will have 
			$$p((\lambda+\rho)^-)^{\ev}=(2,\underbrace{1,1,\cdots,1}_{n-3},1).$$
	By Lemma \ref{reducible} and Proposition \ref{son-socular},
 $M_I(z_1\xi_{n-1}+z_2\xi_{n})$		is irreducible.
		\end{enumerate}

	Thus when $\max \{z_1,z_2\}<-1$, $(z_1,z_2)$ is a  reducible point of $M_I(z_1\xi_{n-1}+z_2\xi_{n})$ if and only if    
  \begin{center}
      $z_1+z_2\in $
		$\begin{cases}
			-n+1+\mathbb{Z}_{\geq 0},  &~\text{if}~n=2k+1,\\-n+2+\mathbb{Z}_{\geq 0},  &~\text{if}~n=2k.
		\end{cases}$
  \end{center}

		\item[2)] When $z_1=-1$, we will have the following.
		\begin{enumerate}
			\item Suppose $n$ is odd. Then we have the follows.
			
    When $z_2>-1$, we will have 
		$$p((\lambda+\rho)^-)^{\ev}=(1,1,\underbrace{1,0,1,0,\cdots,1,0}_{2n-4}).$$
	By Lemma \ref{reducible} and Proposition \ref{son-socular},
 $M_I(z_1\xi_{n-1}+z_2\xi_{n})$		is reducible.

    When $z_2=-n+2$, we will have 
		$$p((\lambda+\rho)^-)^{\ev}=(2,\underbrace{1,1,\cdots,1}_{n-3},0,1).$$
	By Lemma \ref{reducible} and Proposition \ref{son-socular},
 $M_I(z_1\xi_{n-1}+z_2\xi_{n})$		is reducible.

    When If $z_2=-n+1$, we will have 
		$$p((\lambda+\rho)^-)^{\ev}=(2,\underbrace{1,1,\cdots,1}_{n-2}).$$
	By Lemma \ref{reducible} and Proposition \ref{son-socular},
 $M_I(z_1\xi_{n-1}+z_2\xi_{n})$		is irreducible.

			\item Suppose $n$ is even. Then we have the follows.
			
    When $z_2>-1$, the result is the   same with the case for $n$ being odd, so $M_I(z_1\xi_{n-1}+z_2\xi_{n})$	is reducible.
			
    When $z_2=-n+3$,  we will have
    $$p((\lambda+\rho)^-)^{\ev}=(2,\underbrace{1,1,\cdots,1}_{n-4},0,1,0,1).$$
	By Lemma \ref{reducible} and Proposition \ref{son-socular},
 $M_I(z_1\xi_{n-1}+z_2\xi_{n})$		is reducible.
	
   When $z_2=-n+2$, we will have				$$p((\lambda+\rho)^-)^{\ev}=(2,\underbrace{1,1,\cdots,1}_{n-3},1).$$
		By Lemma \ref{reducible} and Proposition \ref{son-socular},
 $M_I(z_1\xi_{n-1}+z_2\xi_{n})$		is irreducible.

		\end{enumerate}
	
  Thus  when $z_1=-1$, $(z_1,z_2)$ is a reducible point of $M_I(z_1\xi_{n-1}+z_2\xi_{n})$ if and only if  
		\begin{center}
		  $z_2\in \begin{cases}
			-n+2+\mathbb{Z}_{\geq 0},   &~\text{if}~n=2k+1.\\-n+3+\mathbb{Z}_{\geq 0},   &~\text{if}~n=2k.
		\end{cases}$
		\end{center}
	
	Similarly, when $z_2=-1$, $(z_1,z_2)$ is a  reducible point of $M_I(z_1\xi_{n-1}+z_2\xi_{n})$ if and only if  
		\begin{center}
		  $z_1\in \begin{cases}
			-n+2+\mathbb{Z}_{\geq 0},  &~\text{if}~n=2k+1,\\-n+3+\mathbb{Z}_{\geq 0},  \qquad &~\text{if}~n=2k.
		\end{cases}$
		\end{center}

	\item[3)] When $z_1>-1,z_1\neq z_2$, we will have the following.
	 
   When $z_1>z_2$,  we will have the following:	  	
$$p((\lambda+\rho)^-)^{\ev}=\begin{cases}
	  		(\underbrace{1,0,1,0,\cdots,1,0}_{2n-2},1),&~\text{if}~z_2>-1,
     \\
	  		(\underbrace{1,1,\cdots,1}_{n-s},0,1,\cdots), &~\text{if}~ -1\ge z_2\ge -n+2,
     \\
	  		(\underbrace{1,1,\cdots,1}_{n}),&~\text{if}~ z_2<-n+2.
	  	\end{cases}$$
By Lemma \ref{reducible} and Proposition \ref{son-socular},
 $M_I(z_1\xi_{n-1}+z_2\xi_{n})$		is reducible.

	   When $z_1<z_2$, we will have 
	  	$$p((\lambda+\rho)^-)^{\ev}=(\underbrace{1,0,1,0,\cdots,1,0}_{2n}).$$
By Lemma \ref{reducible} and Proposition \ref{son-socular},
 $M_I(z_1\xi_{n-1}+z_2\xi_{n})$		is reducible.

 Thus  $M_I(z_1\xi_{n-1}+z_2\xi_n)$	is reducible for all $z_1>-1$ and $z_2\neq z_1$.

	\item[4)] When $z_1<-1,z_2>-1$,  we  will have the following:	$$p((\lambda+\rho)^-)^{\ev}=\begin{cases}
		(\underbrace{1,1,\cdots,1}_{n}),&~\text{if}~z_1\le -n+1,\\
		(\underbrace{1,1,\cdots,1}_{n-s},0,1,\cdots),&~\text{if}~-1>z_1>-n+1, 
	\end{cases}$$
	where $0<s\le n-2$.
 By Lemma \ref{reducible} and Proposition \ref{son-socular},
 $M_I(z_1\xi_{n-1}+z_2\xi_{n})$		is reducible.

	\end{enumerate}

 Thus when $z_1,z_2\in \mathbb{Z}$, $(z_1,z_2)$ is a  reducible point of $M_I(z_1\xi_{n-1}+z_2\xi_{n})$ if and only if  
\begin{center}
    $z_1+z_2\in \begin{cases}
    -n+1+\mathbb{Z}_{\ge 0}, &~\text{if}~n=2k+1,\\-n+2+\mathbb{Z}_{\ge 0}, &~\text{if}~n=2k,
\end{cases}$
or $\begin{cases} z_2>-1,&~\text{if}~z_1<-1,\\ z_2\in \mathbb{Z}, &~\text{if}~z_1>-1.
\end{cases}$
\end{center}

\end{proof}

\begin{lem}\label{z1z2-so-2}
	Keep notations as above.
 When $z_1$ or $z_2 \notin \mathbb{Z}$ and $z_1\neq z_2 $,  $M_{I}(z_1 \xi_{n-1}+z_2\xi_{n})$ is reducible if and only if

\begin{enumerate}
    \item $ z_1\in \mathbb{Z}_{\ge 0}, z_2\notin \mathbb{Z}$.
    \item $z_2\in \mathbb{Z}_{\ge 0}, z_1\notin \mathbb{Z}$.

    \item $z_1+z_2\in \begin{cases}
          -n+1+\mathbb{Z}_{\ge 0}, &\textnormal{if}~n=2k+1,\\
           -n+2+\mathbb{Z}_{\ge 0}, &\textnormal{if}~n=2k.
       \end{cases} $
\end{enumerate}

\end{lem}

\begin{proof}
By the isomorphism of root systems of $D_n$, we only need to consider the case 
 $z_2\notin \mathbb{Z}$.

When $z_1\in \mathbb{Z}$ and  $z_2\notin \mathbb{Z}$, we will have
 $x:=\lambda+\rho \in [\lambda]_3$. Then we have the following.

		 \begin{enumerate}
		 	\item[1)] When $z_1>-1$, we will have 
		 	$p(\widetilde{x})=(\underbrace{1,1,\cdots,1}_{n})$.
		 	Then by Proposition \ref{GKdim},
 $$\GK\:L(z_1 \xi_{n-1}+z_2\xi_{n})=\frac{1}{2}n^2-\frac{1}{2}n<\dim(\mathfrak{u}).$$
  So by Lemma \ref{reducible},   $M_I(z_1\xi_{n-1}+z_2\xi_{n})$	is reducible.

    \item[2)] When $z_1\le-1$, we will have 
		 	$p(\widetilde{x})=(2,\underbrace{1,1,\cdots,1}_{n-2})$.
	 	Then by Proposition \ref{GKdim},
 $$\GK\:L(z_1 \xi_{n-1}+z_2\xi_{n})=\frac{1}{2}n^2+\frac{1}{2}n-1=\dim(\mathfrak{u}).$$
 So  by Lemma \ref{reducible}, $M_I(z_1\xi_{n-1}+z_2\xi_{n})$ is irreducible.
		 \end{enumerate}
		
Thus, by symmetry, we complete the proof of (1) and (2).	
		
		When $z_1\notin \mathbb{Z}$ and $z_2\notin \mathbb{Z}$, we will have the following.
		\begin{enumerate}

			\item When $z_1+z_2\notin \mathbb{Z}$ and $z_1-z_2\in \mathbb{Z}$, we have  $(-\frac{1}{2}z_1+\frac{1}{2}z_2)\in [\lambda]_{1,2}$ and 
    \begin{center}
       $x:=( \frac{1}{2}(z_1+z_2)+n-1,\frac{1}{2}(z_1+z_2)+n-2,\cdots,\frac{1}{2}(z_1+z_2)+1 ) \in [\lambda]_3$.
   \end{center}
	So $p(\widetilde{x})=(1,\underbrace{1,1,\cdots,1}_{n-2})$.			
			Then by Proposition \ref{GKdim},
 $$\GK\: L(z_1\xi_{n-1}+z_2\xi_n)=\frac{1}{2}n^2+\frac{1}{2}n-1=\dim(\mathfrak{u}).$$
 So by Lemma \ref{reducible},  $M_I(z_1\xi_{n-1}+z_2\xi_n)$	is  irreducible.
			
			\item When $z_1+z_2\notin \mathbb{Z}$ and $z_1-z_2\notin \mathbb{Z}$, we have 
    $$x_1:=( \frac{1}{2}(z_1+z_2)+n-1,\cdots,\frac{1}{2}(z_1+z_2)+1 ), x_2:=(-\frac{1}{2}z_1+\frac{1}{2}z_2)\in [\lambda]_3.$$
			So $p(\widetilde{x_1})=(1,\underbrace{1,1,\cdots,1}_{n-2})$ and $p(\widetilde{x_2})=(1)$.
	Then by Proposition \ref{GKdim},
 $$\GK\: L(z_1\xi_{n-1}+z_2\xi_n)=\frac{1}{2}n^2+\frac{1}{2}n-1=\dim(\mathfrak{u}).$$
  So by Lemma \ref{reducible}$, M_I(z_1\xi_{n-1}+z_2\xi_n)$	is  irreducible.

   \item When $z_1\pm z_2\in \mathbb{Z}$, we will have $[\lambda]_3=\emptyset$. Then we have the follows.
		
  \begin{enumerate}
			\item  When $z_1+z_2>-2$, we will have 
	\begin{center}
				    $p((\lambda+\rho)^-_{(0)})^{\ev}+p((\lambda+\rho)^-_{(\frac{1}{2})})^{\ev} =
         \begin{cases}
           (2,\underbrace{0,1,\cdots,0,1}_{2n-4}),&~\text{if}~z_2-z_1\le 0,\\ (1,1,\underbrace{1,0,1,0,\cdots,1}_{2n-5}),&~\text{if}~z_2-z_1> 0.
        \end{cases}$ 
				\end{center}
				
				Then by Proposition \ref{GKdim}, 
    \begin{center}
        $\GK\: L(\lambda)$=$\begin{cases}n^2-n-(2+4+\cdots+2n-4)=2n-2,&~\text{if}~z_2-z_1\le 0,\\
                n^2-n-(1+2+4+\cdots+2n-4)=2n-3,&~\text{if}~z_2-z_1> 0.
				\end{cases}$
    \end{center}
       So $\GK\: L(z_1\xi_{n-1}+z_2\xi_n)<\dim(\mathfrak{u})$. Thus by Lemma \ref{reducible},  $M_I(z_1\xi_{n-1}+z_2\xi_n)$	is  reducible.
				
			\item	When $z_1+z_2=-n+2$, we will have 
			$$p((\lambda+\rho)^-_{(0)})^{\ev}+p((\lambda+\rho)^-_{(\frac{1}{2})})^{\ev}= \begin{cases}
				(2,\underbrace{1,1,\cdots,1}_{n-4},1,0,1),  &~\text{if}~n=2k+1,\\
    (2,\underbrace{1,1,\cdots,1}_{n-4},0,1,0,1),  &~\text{if}~n=2k.
				\end{cases}$$				Then by Proposition \ref{GKdim},
    when $n=2k+1$, we will have
 $$\GK\: L(z_1\xi_{n-1}+z_2\xi_n)=\frac{1}{2}n^2+\frac{1}{2}n-2<\dim(\mathfrak{u}),$$
 and  when $n=2k$, we will have
 $$\GK\: L(z_1\xi_{n-1}+z_2\xi_n)=\frac{1}{2}n^2+\frac{1}{2}n-4<\dim(\mathfrak{u}).$$
 So by Lemma \ref{reducible}$, M_I(z_1\xi_{n-1}+z_2\xi_n)$	is  reducible.
				
		\item When $z_1+z_2=-n+1$, we will have 
				$$p((\lambda+\rho)^-_{(0)})^{\ev}+p((\lambda+\rho)^-_{(\frac{1}{2})})^{\ev}=\begin{cases}
				(2,\underbrace{1,1,\cdots,1}_{n-3},0,1),  &~\text{if}~n=2k+1,\\
    	(2,\underbrace{1,1,\cdots,1}_{n-2}),  &~\text{if}~n=2k.
				\end{cases}$$
			
			Then by Proposition \ref{GKdim},  when $n=2k+1$, we will have
 $$\GK\: L(z_1\xi_{n-1}+z_2\xi_n)=\frac{1}{2}n^2+\frac{1}{2}n-2<\dim(\mathfrak{u}),$$
     and  when $n=2k$, we will have
 $$\GK\: L(z_1\xi_{n-1}+z_2\xi_n)=\frac{1}{2}n^2+\frac{1}{2}n-1=\dim(\mathfrak{u}).$$
So by Lemma \ref{reducible}, we can get that 
  $$\begin{cases}
          M_I(z_1\xi_1+z_2\xi_{n-1}) ~	\text{is}~ \text{reducible},  &~\text{if}~n=2k+1,\\
          M_I(z_1\xi_1+z_2\xi_{n-1}) ~	\text{is}~ \text{irreducible},  &~\text{if}~n=2k.
    	   	\end{cases}$$

				\item When $z_1+z_2=-n$, we will have 
				
	\begin{center}
				    $p((\lambda+\rho)^-_{(0)})^{\ev}+p((\lambda+\rho)^-_{(\frac{1}{2})})^{\ev}=(2,\underbrace{1,1,\cdots,1}_{n-2})$ .
				\end{center}
				
				Then by Proposition \ref{GKdim},
 $$\GK\: L(z_1\xi_{n-1}+z_2\xi_n)=\frac{1}{2}n^2+\frac{1}{2}n-1=\dim(\mathfrak{u}).$$
  So by Lemma \ref{reducible}$, M_I(z_1\xi_{n-1}+z_2\xi_n)$	is  irreducible.
			\end{enumerate}

Therefore,  when $z_1\pm z_2\in \mathbb{Z}$, $(z_1,z_2)$ is a  reducible point of $M_I(z_1\xi_{n-1}+z_2\xi_{n})$ if and only if  
\begin{center}
    $z_1+z_2\in \begin{cases}
      -n+1+\mathbb{Z}_{\geq 0}, &~\text{if}~n=2k+1,\\
       -n+2+\mathbb{Z}_{\geq 0}, &~\text{if}~n=2k.
    \end{cases}$
\end{center}

\item When $z_1+z_2\in \mathbb{Z}$ and $z_1-z_2\notin \mathbb{Z}$, we have $(-\frac{1}{2}z_1+\frac{1}{2}z_2)\in [\lambda]_3$, and  the result is the   same with the case for $z_1\pm z_2\in \mathbb{Z}$.
				
		\end{enumerate}

\end{proof}

Combined the above lemmas, we complete the proof of  part (2) in Theorem \ref{son-thm}.

From the arguments in the proof of Theorem \ref{son-thm}, we can get the following result.

\begin{Cor}
For $ \mathfrak{g}=\mathfrak{so}(2n,\mathbb{C}) $, $M_{I}(\lambda)$ is a scalar generalized Verma module with highest weight $\lambda=z_1 \xi_p+z_2\xi_q$ $(p,q\in \{1,n-1,n\})$, where $\xi_p$ is the fundamental weight corresponding to $\alpha_{p}=e_p-e_{p+1}$, $\xi_q$ is the fundamental weight corresponding to $\alpha_{q}=e_q-e_{q+1}$. Then $M_{I}(\lambda)$ is irreducible if and only if 
there are two elements $x\in [\lambda]_{1,2}$ and $y\in [\lambda]_3$, such that one of the following holds:
\begin{enumerate}
    \item  $p(x^-)^{\ev}=(2, \underbrace{1,\cdots,1}_{n-2})$.
    \item $p(x^-)^{\ev}=(1, \underbrace{1,\cdots,1}_{n-2})$.
    \item $p(\widetilde{y})=(2, \underbrace{1,\cdots,1}_{n-2})$.
    \item $p(\widetilde{y})=(1, \underbrace{1,\cdots,1}_{n-2})$.
\end{enumerate}

\end{Cor}

\begin{example}
  Let $\mathfrak{g}=\mathfrak{so}(12,\mathbb{C})$, we consider the parabolic subalgebra $\frq$ corresponding to the subset $\Pi\setminus\{\alpha_1, \alpha_5\}$.
	In this case, the scalar generalized Verma module $M_{I}(z_1\xi_1+z_2\xi_5)$ is reducible if and only if one of the following holds.
 \begin{enumerate}
    \item  
     $z_1\in \mathbb{Z}_{\ge 0}, z_2\in \mathbb{C}$.
     \item $z_1+z_2\in -4+\mathbb{Z}_{\ge 0}$ if
     $z_1=-1~\text{or}~ z_1,z_2\notin \mathbb{Z}$.
     \item 
     $z_1=z_2\in -\frac{3}{2}+\mathbb{Z}_{\ge 0} $.
     \item $z_2\in -2+\mathbb{Z}_{\ge 0}$ if
     $z_1\in \mathbb{C}~\text{or}~ z_1=z_2\in \mathbb{Z}$.
 \end{enumerate}

 We can draw these reducible points in the following complex plane $\mathbb{C}^2$:

 \hspace{1cm}
 
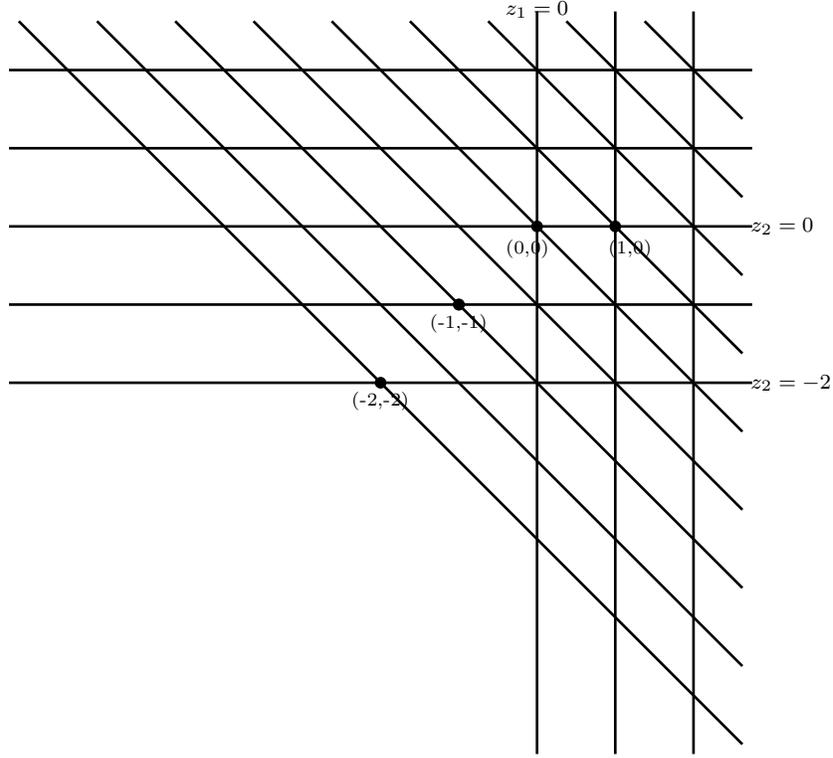
\begin{center}
\psset{unit=1.3cm}
    \begin{pspicture}(-5,-5)(2.4,2.4)
        \psset{linecolor=black}

        \cnode*(0,0){0.06}{a1}
        \cnode*(-0.8,-0.8){0.06}{d1}
        \cnode*(-1.6,-1.6){0.06}{d3}
        \cnode*(0.8,0){0.06}{d10}
        \cnode*(-0.8,-0.8){0.06}{e3}

        \psline[linewidth=1pt](-5.3,2.1)(2.1,-5.3)
        \psline[linewidth=1pt](-4.5,2.1)(2.1,-4.5)
        \psline[linewidth=1pt](-3.7,2.1)(2.1,-3.7)
        \psline[linewidth=1pt](-2.9,2.1)(2.1,-2.9)
        \psline[linewidth=1pt](-2.1,2.1)(2.1,-2.1)
        \psline[linewidth=1pt](-1.3,2.1)(2.1,-1.3)
        \psline[linewidth=1pt](-0.5,2.1)(2.1,-0.5)
        \psline[linewidth=1pt](0.3,2.1)(2.1,0.3)
        \psline[linewidth=1pt](1.1,2.1)(2.1,1.1)
        \psline[linewidth=1pt](-5.4,-0.8)(2.2,-0.8)
        \psline[linewidth=1pt](-5.4,-1.6)(2.2,-1.6)
        \psline[linewidth=1pt](-5.4,0)(2.2,0)
        \psline[linewidth=1pt](-5.4,0.8)(2.2,0.8)
        \psline[linewidth=1pt](-5.4,1.6)(2.2,1.6)
        \psline[linewidth=1pt](0.,2.2)(0,-5.4)
        \psline[linewidth=1pt](1.6,2.2)(1.6,-5.4)
        \psline[linewidth=1pt](0.8,2.2)(0.8,-5.4)

        \uput[u](0,2.){\scriptsize{$z_1=0$}}
        \uput[r](2.05,0){\scriptsize{$z_2=0$}}
           \uput[r](2.05,-1.6){\scriptsize{$z_2=-2$}}
        \uput[d](-0.1,0.0){\tiny(0,0)}
        \uput[d](-0.8,-0.76){\tiny(-1,-1)}
        \uput[d](-1.6,-1.56){\tiny(-2,-2)}
        \uput[d](0.95,0.0){\tiny(1,0)}

        \end{pspicture}
\end{center}

So the reducible points contain the following set:
\begin{enumerate}
     \item  
     $z_1\in \mathbb{Z}_{\ge 0}, z_2\in \mathbb{C}$.
     \item $z_2\in -2+\mathbb{Z}_{\ge 0}$, $z_1\in \mathbb{C}$.
     \item 
     $z_1+z_2\in -4+\mathbb{Z}_{\ge 0}$.
     
\end{enumerate}

\end{example}

\begin{example}
  Let $\mathfrak{g}=\mathfrak{so}(14,\mathbb{C})$, We consider the parabolic subalgebra $\frq$ corresponding to the subset $\Pi\setminus\{\alpha_6, \alpha_7\}$.
	In this case, the scalar generalized Verma module $M_{I}(z_1\xi_6+z_2\xi_7)$ is reducible if and only if one of the following holds.
 \begin{enumerate}
     \item $z_1\in \mathbb{Z}_{\ge 0}, z_2\in \mathbb{C}$.
     \item $z_2\in \mathbb{Z}_{\ge 0}, z_1\in \mathbb{C}$.
     \item $(z_1,z_2)\in \mathbb{C}\times\mathbb{C}, ~\text{and}~ z_1+z_2\in 
          -6+\mathbb{Z}_{\ge 0}$.
     \item $z_1=z_2\in -3+\frac{1}{2}\mathbb{Z}_{\ge 0}$.
 \end{enumerate}
    
    We can draw these reducible points in the following complex plane $\mathbb{C}^2$: 

\hspace{1cm}

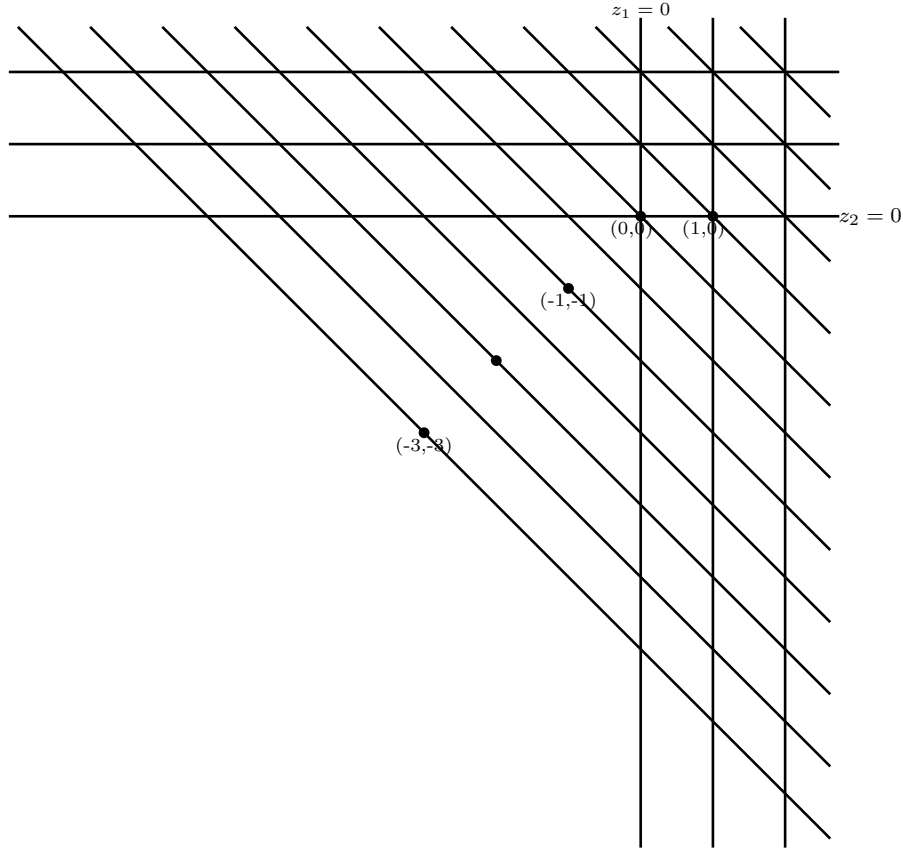
\begin{center}
\psset{unit=1.2cm}
    \begin{pspicture}(-6.8,-6.8)(2.4,2.4)
        \psset{linecolor=black}
        \cnode*(0,0){0.06}{e1}
        \cnode*(-0.8,-0.8){0.06}{g1}
        \cnode*(-1.6,-1.6){0.06}{i1}
        \cnode*(-2.4,-2.4){0.06}{k1}
        \cnode*(0.8,0){0.06}{c}

        \psline[linewidth=1pt](-7,0)(2.2,0)
        \psline[linewidth=1pt](-6.9,2.1)(2.1,-6.9)
        \psline[linewidth=1pt](-6.1,2.1)(2.1,-6.1)
        \psline[linewidth=1pt](-5.3,2.1)(2.1,-5.3)
        \psline[linewidth=1pt](-4.5,2.1)(2.1,-4.5)
        \psline[linewidth=1pt](-3.7,2.1)(2.1,-3.7)
        \psline[linewidth=1pt](-2.9,2.1)(2.1,-2.9)
        \psline[linewidth=1pt](-2.1,2.1)(2.1,-2.1)
        \psline[linewidth=1pt](-1.3,2.1)(2.1,-1.3)
        \psline[linewidth=1pt](-0.5,2.1)(2.1,-0.5)
        \psline[linewidth=1pt](0.3,2.1)(2.1,0.3)
        \psline[linewidth=1pt](1.1,2.1)(2.1,1.1)
        \psline[linewidth=1pt](0,2.2)(0,-7)
        \psline[linewidth=1pt](-7,0.8)(2.2,0.8)
        \psline[linewidth=1pt](-7,1.6)(2.2,1.6)
        \psline[linewidth=1pt](0.8,2.2)(0.8,-7)
        \psline[linewidth=1pt](1.6,2.2)(1.6,-7)

        \uput[u](0,2.05){\tiny{$z_1=0$}}
        \uput[r](2.05,0){\scriptsize{$z_2=0$}}
        \uput[d](-0.1,0.1){\tiny(0,0)}
        \uput[d](-0.8,-0.70){\tiny(-1,-1)}
        \uput[d](0.7,0.1){\tiny(1,0)}
        \uput[d](-2.4,-2.3){\tiny(-3,-3)}
        
        \end{pspicture}
\end{center}

So the reducible points contain the following set:
\begin{enumerate}
     \item  
     $z_1\in \mathbb{Z}_{\ge 0}, z_2\in \mathbb{C}$.
     \item $z_2\in \mathbb{Z}_{\ge 0}$, $z_1\in \mathbb{C}$.
     \item 
     $z_1+z_2\in -6+\mathbb{Z}_{\ge 0}$.
     
\end{enumerate}

\end{example}

\section{Reducibility of scalar  generalized Verma modules for type $E_{6}$}
In this section, we will give the reducible points for type $E_6$. 

Let $M_{I}(\lambda)$ be a scalar generalized Verma module with highest weight $\lambda=z_1 \xi_1+z_2\xi_6$. 
From \S \ref{nilpotent}, we know that $\dim(\fru)=24$ and $M_{I}(z (\xi_1+\xi_6))$ is reducible if and only if $\GK\: L(z_1\xi_{1}+z_2\xi_6)<24$ by Lemma \ref{reducible}, equivalently ${\bf a}({\lambda})>12$.

\begin{lem}\label{e6z1z2z} 
		Let $ \mathfrak{g}=E_6$. $M_{I}(\lambda)$ is a scalar generalized Verma module with highest weight $\lambda=z_1 \xi_1+z_2\xi_6$.
   When $z_1=z_2=z\in \mathbb{Z}$,  $M_{I}(z (\xi_1+\xi_6))$ is reducible if and only if
       \begin{center}
           $z\in -3+\mathbb{Z}_{\ge 0}.$
       \end{center}

\end{lem}
\begin{proof}

Since 
	$\rho=\xi_1+\cdots+\xi_6$, we have
	\begin{align*}
	    \lambda+\rho&=(1+z_1)\xi_1+\xi_2+\cdots+\xi_5+(1+z_2)\xi_6:=[1+z_1,1,1,1,1,1+z_2]\\
     &=(0,1,2,3,z_2+4,-\frac{2z_1+z_2}{3}-4,-\frac{2z_1+z_2}{3}-4,\frac{2z_1+z_2}{3}+4).
	\end{align*}
	
	 When $z_1=z_2=z$,    we  have
	\begin{center}
		$\lambda+\rho=[1+z,1,1,1,1,1+z]=(0,1,2,3,z+4,-z-4,-z-4,z+4)$.
		
	\end{center}

	When $z\in \mathbb{Z}$, equivalently $\lambda+\rho$ is integral,  we have the following.
	\begin{enumerate}
	\item[1)] When $z=-3$,  we will have
$\lambda+\rho=[-2,1,1,1,1,-2]=(0,1,2,3,1,-1,-1,1)$. From Lemma \ref{w-lambda}, we have $w_{\lambda}=[4,3,4,2,3,4,5,1,3,4,5,2,3,4,1,3,2,0,2,3,4,5,1,3,4,2,3,0]$. By using PyCox, we have ${\bf a}(w_{\lambda})=15>12$. Thus $M_{I}(z (\xi_1+\xi_6))$ is reducible.
 
\item[2)] When $z=-4$,  we will have
$\lambda+\rho=[-3,1,1,1,1,-3]=(0,1,2,3,0,0,0,0)$. From Lemma \ref{w-lambda}, we have $w_{\lambda}=[4,3,4,2,3,4,1,3,4,5,2,3,4,1,3,0,2,3,4,1,3]$. By using PyCox, we have ${\bf a}(w_{\lambda})=12$. Thus $M_{I}(z (\xi_1+\xi_6))$ is irreducible.

\end{enumerate}

Therefore, by Lemma \ref{GKdown}, $M_{I}(z (\xi_1+\xi_6))$ is reducible if and only if
       \begin{center}
           $z\in -3+\mathbb{Z}_{\ge 0}.$
       \end{center}

\end{proof}

\begin{lem} 
		Let $ \mathfrak{g}=E_6$. $M_{I}(\lambda)$ is a scalar generalized Verma module with highest weight $\lambda=z_1 \xi_1+z_2\xi_6$.
  When $z_1=z_2=z\notin \mathbb{Z}$,  $M_{I}(z (\xi_1+\xi_6))$ is reducible if and only if
$$z\in -\frac{7}{2}+\mathbb{Z}_{\ge 0}.$$

\end{lem}
\begin{proof}
    When $z_1=z_2=z$,    we  have
	\begin{center}
		$\lambda+\rho=[1+z,1,1,1,1,1+z]=(0,1,2,3,z+4,-z-4,-z-4,z+4)$.
		
	\end{center}
When $z\notin \mathbb{Z}$, equivalently $\lambda+\rho$ is not integral,  we have the following.
	\begin{enumerate}
	\item[1)] When $z\notin \frac{1}{2}\mathbb{Z}$, we have $\Delta_{[\lambda]}\simeq D_4$ (denote this isomorphism by $\phi$) and the simple system is $\Pi_{[\lambda]}=\{\alpha_2,\alpha_3,\alpha_4,\alpha_5\}$.
Denote $\bar{\lam}'=\phi(\lam|_{\mathfrak{h}_{\Delta_{[\lambda]}}^*})$. Then $\bar{\lam}'|_{D_4}=\xi_1+\xi_2+\xi_3+\xi_4=(3,2,1,0)$. Thus ${\bf a}(\lambda)={\bf a}(\bar{\lam}'|_{D_4})=12$ by Proposition \ref{integral}, which implies that $M_{I}(z (\xi_1+\xi_6))$ is irreducible.

\item[2)] When $z\in \frac{1}{2}+\mathbb{Z}$, we have the following. 
\begin{enumerate}
	\item When $z=-\frac{9}{2}$, we have $\Delta_{[\lambda]}\simeq D_5$ (denote this isomorphism by $\phi$) and the simple system is $\Pi_{[\lambda]}=\{\frac{1}{2}(-e_1-e_2-e_3-e_4+e_5-e_6-e_7+e_8),\alpha_2,\alpha_4,\alpha_5,\alpha_3\}$.
Denote $\bar{\lam}'=\phi(\lam|_{\mathfrak{h}_{\Delta_{[\lambda]}}^*})$. Then $\bar{\lam}'|_{D_5}=-4\xi_1+\xi_2+\xi_3+\xi_4+\xi_5=(-1,3,2,1,0)$. Thus ${\bf a}(\lambda)={\bf a}(\bar{\lam}'|_{D_5})=12$ by Proposition \ref{integral}, which implies that $M_{I}(z (\xi_1+\xi_6))$ is irreducible.

\item When $z=-\frac{7}{2}$, we have $\Delta_{[\lambda]}\simeq D_5$ (denote this isomorphism by $\phi$) and the simple system is $\Pi_{[\lambda]}=\{\gamma,\alpha_2,\alpha_4,\alpha_5,\alpha_3\}$ where $\gamma=\frac{1}{2}(-e_1-e_2-e_3-e_4+e_5-e_6-e_7+e_8)$.
Denote $\bar{\lam}'=\phi(\lam|_{\mathfrak{h}_{\Delta_{[\lambda]}}^*})$. Then $\bar{\lam}'|_{D_5}=-2\xi_1+\xi_2+\xi_3+\xi_4+\xi_5=(1,3,2,1,0)$. Thus ${\bf a}(\lambda)={\bf a}(\bar{\lam}'|_{D_5})=13>12$ by Proposition \ref{integral}, which implies that $M_{I}(z (\xi_1+\xi_6))$ is reducible.
\end{enumerate}
\end{enumerate}

Therefore, by Lemma \ref{GKdown}, $M_{I}(z (\xi_1+\xi_6))$ is reducible if and only if
       $$z\in -\frac{7}{2}+\mathbb{Z}_{\ge 0}.$$
 
\end{proof}



\begin{lem} 
		Let $ \mathfrak{g}=E_6$. $M_{I}(\lambda)$ is a scalar generalized Verma module with highest weight $\lambda=z_1 \xi_1+z_2\xi_6$.
  When $z_1, z_2 \in \mathbb{Z}$ and $z_1\neq z_2 $,  $M_{I}(z_1\xi_1+z_2\xi_6)$ is reducible if and only if one of the following holds:
 \begin{enumerate}
      \item $z_1\in -3+\mathbb{Z}_{\geq 0}$ and $z_2\in\mathbb{Z}$.
      \item $z_2\in -3+\mathbb{Z}_{\geq 0}$ and $z_1\in \mathbb{Z}$.
     \end{enumerate}
 
\end{lem}
\begin{proof}
    Recall that we  have
	\begin{align*}
	    \lambda+\rho&=(1+z_1)\xi_1+\xi_2+\cdots+\xi_5+(1+z_2)\xi_6:=[1+z_1,1,1,1,1,1+z_2]\\
     &=(0,1,2,3,z_2+4,-\frac{2z_1+z_2}{3}-4,-\frac{2z_1+z_2}{3}-4,\frac{2z_1+z_2}{3}+4).
	\end{align*}
When $z_1$, $z_2 \in \mathbb{Z}$ and $z_1\neq z_2 $, we have the following.
\begin{enumerate}
\item[1)] When $z_1=-4$ and $z_2=-3$,  we will have
$\lambda+\rho=[-3,1,1,1,1,-2]=(0,1,2,3,1,-\frac{1}{3},-\frac{1}{3},\frac{1}{3})$. From Lemma \ref{w-lambda}, we have $w_{\lambda}=[4,3,4,2,3,4,5,1,3,4,5,2,3,4,1,3,2,0,2,3,4,5,1,3,2]$. By using PyCox, we have ${\bf a}(w_{\lambda})=13>12$. Thus $M_{I}(z_1\xi_1+z_2\xi_6)$ is reducible.

\item[2)] When $z_1=-5$ and $z_2=-3$,  we will have
$\lambda+\rho=[-4,1,1,1,1,-2]=(0,1,2,3,1,\frac{1}{3},\frac{1}{3},-\frac{1}{3})$. From Lemma \ref{w-lambda}, we have $w_{\lambda}=[4,3,4,2,3,4,5,1,3,4,5,2,3,4,1,3,2,0,2,3,4,1]$. By using PyCox, we have ${\bf a}(w_{\lambda})=13>12$. Thus $M_{I}(z_1\xi_1+z_2\xi_6)$ is reducible.

\item[3)] When $z_1=-6$ and $z_2=-3$,  we will have
$\lambda+\rho=[-5,1,1,1,1,-2]=(0,1,2,3,1,1,1,-1)$. From Lemma \ref{w-lambda}, we have $w_{\lambda}=[4,3,4,2,3,4,5,1,3,4,5,2,3,4,1,3,2,0,2,3]$. By using PyCox, we have ${\bf a}(w_{\lambda})=13>12$. Thus $M_{I}(z_1\xi_1+z_2\xi_6)$ is reducible.

\item[4)] When $z_1=-7$ and $z_2=-3$,  we will have
$\lambda+\rho=[-6,1,1,1,1,-2]=(0,1,2,3,1,\frac{5}{3},\frac{5}{3},-\frac{5}{3})$. From Lemma \ref{w-lambda}, we have $w_{\lambda}=[4,3,4,2,3,4,5,1,3,4,5,2,3,4,1,3,2,0]$. By using PyCox, we have ${\bf a}(w_{\lambda})=13>12$. Thus $M_{I}(z_1\xi_1+z_2\xi_6)$ is reducible.
 
\item[5)] When $z_1\leq -8$ and $z_2=-3$,  we will have
$\lambda+\rho=[z_1+1,1,1,1,1,-2]=(0,1,2,3,1,-3-\frac{2z_1}{3},-3-\frac{2z_1}{3},3+\frac{2z_1}{3})$. From Lemma \ref{w-lambda}, we have $w_{\lambda}=[4,3,4,2,3,4,5,1,3,4,5,2,3,4,1,3,2]$. By using PyCox, we have ${\bf a}(w_{\lambda})=13>12$. Thus $M_{I}(z_1\xi_1+z_2\xi_6)$ is reducible.

\item[6)] When $z_1\geq -2$ and $z_2=-3$ or $z_2\geq -2$ and $z_1=-3$,   $M_{I}(z_1\xi_1+z_2\xi_6)$ is reducible by Lemma \ref{e6z1z2z} and Lemma \ref{GKdown}.

\item[7)] When $z_1\leq -5$ and $z_2=-4$,   $M_{I}(z_1\xi_1+z_2\xi_6)$ is irreducible by Lemma \ref{e6z1z2z} and Lemma \ref{GKdown} since $M_{I}(-4\xi_1-4\xi_6)$ is irreducible.

\end{enumerate}
By the symmetry of $z_1$ and $z_2$, we can get that 
$M_{I}(z_1\xi_1+z_2\xi_6)$ is irreducible if and only if
$z_1\leq -4$ and $z_2\leq -4$ when $z_1$, $z_2 \in \mathbb{Z}$ and $z_1\neq z_2 $.
Equivalently, $M_{I}(z_1\xi_1+z_2\xi_6)$ is reducible if and only if
$z_1\geq -3$ or $z_2\geq -3$ when $z_1$, $z_2 \in \mathbb{Z}$ and $z_1\neq z_2 $.

\end{proof}

\begin{lem} 
	Keep notations as above.
 When $z_1$ or $z_2 \notin \mathbb{Z}$ and $z_1\neq z_2 $,  $M_{I}(z_1 \xi_1+z_2\xi_6)$ is reducible if and only if one of the following holds:

\begin{enumerate}
    \item $ z_2\notin \mathbb{Z}$ and $z_1\in -3+\mathbb{Z}_{\ge 0}$.
    \item $ z_1\notin \mathbb{Z}$ and $z_2\in -3+\mathbb{Z}_{\ge 0}$.

    \item $z_1+z_2\in -7+\mathbb{Z}_{\ge 0}.$
\end{enumerate}

\end{lem}

\begin{proof}
    Recall that we  have
	\begin{align*}
	    \lambda+\rho&=(1+z_1)\xi_1+\xi_2+\cdots+\xi_5+(1+z_2)\xi_6:=[1+z_1,1,1,1,1,1+z_2]\\
     &=(0,1,2,3,z_2+4,-\frac{2z_1+z_2}{3}-4,-\frac{2z_1+z_2}{3}-4,\frac{2z_1+z_2}{3}+4).
	\end{align*}

When $z_1$ or $z_2 \notin \mathbb{Z}$ and $z_1\neq z_2 $, we have the following.
\begin{enumerate}
	\item[1)] When $ z_2\notin \mathbb{Z}$ and $z_1\in \mathbb{Z}$, we have $\Delta_{[\lambda]}\simeq D_5$ (denote this isomorphism by $\phi$) and the simple system is $\Pi_{[\lambda]}=\{\alpha_1,\alpha_3,\alpha_4,\alpha_2,\alpha_5\}$.
Denote $\bar{\lam}'=\phi(\lam|_{\mathfrak{h}_{\Delta_{[\lambda]}}^*})$. Then $\bar{\lam}'|_{D_5}=(1+z_1)\xi_1+\xi_2+\xi_3+\xi_4+\xi_5=(4+z_1,3,2,1,0)$. 
\begin{enumerate}
	\item
When $z_1=-4$, we have $\bar{\lam}'|_{D_5}=(0,3,2,1,0)$.
Thus ${\bf a}(\lambda)={\bf a}(\bar{\lam}'|_{D_5})=12$ by Proposition \ref{integral}, which implies that $M_{I}(z (\xi_1+\xi_6))$ is irreducible.
\item
When $z_1=-3$, we have $\bar{\lam}'|_{D_5}=(1,3,2,1,0)$.
Thus ${\bf a}(\lambda)={\bf a}(\bar{\lam}'|_{D_5})=13>12$ by Proposition \ref{integral}, which implies that $M_{I}(z (\xi_1+\xi_6))$ is reducible.
\end{enumerate}
Therefore when $ z_2\notin \mathbb{Z}$ and $z_1\in \mathbb{Z}$, $M_{I}(z_1 \xi_1+z_2\xi_6)$ is reducible if and only if $z_1\in -3+\mathbb{Z}_{\ge 0}$.
\item[2)] When $ z_1\notin \mathbb{Z}$ and $z_2\in \mathbb{Z}$, we have $\Delta_{[\lambda]}\simeq D_5$ (denote this isomorphism by $\phi$) and the simple system is $\Pi_{[\lambda]}=\{\alpha_6,\alpha_5,\alpha_4,\alpha_3,\alpha_2\}$.
Denote $\bar{\lam}'=\phi(\lam|_{\mathfrak{h}_{\Delta_{[\lambda]}}^*})$. Then $\bar{\lam}'|_{D_5}=(1+z_2)\xi_1+\xi_2+\xi_3+\xi_4+\xi_5=(4+z_2,3,2,1,0)$.
Similarly when $ z_1\notin \mathbb{Z}$ and $z_2\in \mathbb{Z}$, we can prove that $M_{I}(z_1 \xi_1+z_2\xi_6)$ is reducible if and only if $z_2\in -3+\mathbb{Z}_{\ge 0}$.
\item[3)]When $ z_1\notin \mathbb{Z}$, $z_2\notin \mathbb{Z}$ and $ z_1+z_2\notin \mathbb{Z}$, we have $\Delta_{[\lambda]}\simeq D_4$ (denote this isomorphism by $\phi$) and the simple system is $\Pi_{[\lambda]}=\{\alpha_2,\alpha_4,\alpha_5,\alpha_3\}$.
Denote $\bar{\lam}'=\phi(\lam|_{\mathfrak{h}_{\Delta_{[\lambda]}}^*})$. Then $\bar{\lam}'|_{D_4}=\xi_1+\xi_2+\xi_3+\xi_4=(3,2,1,0)$ which is dominant integral. Thus ${\bf a}(\lambda)={\bf a}(\bar{\lam}'|_{D_4})=12$ by Proposition \ref{integral}, which implies that $M_{I}(z (\xi_1+\xi_6))$ is irreducible.

\item[4)]When $ z_1\notin \mathbb{Z}$, $z_2\notin \mathbb{Z}$ and $ z_1+z_2\in \mathbb{Z}$, we have $\Delta_{[\lambda]}\simeq D_5$ (denote this isomorphism by $\phi$) and the simple system is $\Pi_{[\lambda]}=\{\gamma,\alpha_2,\alpha_4,\alpha_5,\alpha_3\}$ where $\gamma=\frac{1}{2}(-e_1-e_2-e_3-e_4+e_5-e_6-e_7+e_8)$.
Denote $\bar{\lam}'=\phi(\lam|_{\mathfrak{h}_{\Delta_{[\lambda]}}^*})$. Then $\bar{\lam}'|_{D_5}=(5+z_1+z_2)\xi_1+\xi_2+\xi_3+\xi_4+\xi_5=(8+z_1+z_2,3,2,1,0)$. 
\begin{enumerate}
	\item
When $z_1+z_2=-8$, we have $\bar{\lam}'|_{D_5}=(0,3,2,1,0)$.
Thus ${\bf a}(\lambda)={\bf a}(\bar{\lam}'|_{D_5})=12$ by Proposition \ref{integral}, which implies that $M_{I}(z (\xi_1+\xi_6))$ is irreducible.
\item
When $z_1+z_2=-7$, we have $\bar{\lam}'|_{D_5}=(1,3,2,1,0)$.
Thus ${\bf a}(\lambda)={\bf a}(\bar{\lam}'|_{D_5})=13>12$ by Proposition \ref{integral}, which implies that $M_{I}(z (\xi_1+\xi_6))$ is reducible.
\end{enumerate}
Therefore when $ z_1\notin \mathbb{Z}$, $z_2\notin \mathbb{Z}$ and $ z_1+z_2\in \mathbb{Z}$, $M_{I}(z_1 \xi_1+z_2\xi_6)$ is reducible if and only if $z_1+z_2\in -7+\mathbb{Z}_{\ge 0}$.

\end{enumerate}
  
\end{proof}

\begin{Cor}\label{e6-zong}
   Let $ \mathfrak{g}=E_6$. The scalar generalized Verma module $M_I(z_1\xi_1+z_2\xi_6)$ is reducible if and only if one of the following holds:
   \begin{enumerate}
       \item $z_1\in -3+\mathbb{Z}_{\ge 0}$ and $z_2\in \mathbb{C}$.
       \item $z_2\in -3+\mathbb{Z}_{\ge 0}$ and $z_1\in \mathbb{C}$.
       \item $z_1+z_2\in -7+\mathbb{Z}_{\ge 0}$.
   \end{enumerate}
\end{Cor}

Now we complete the proof of Theorem \ref{thm-e}.

Similarly we can draw these reducible points in the following complex plane $\mathbb{C}^2$:

\hspace{1cm}

\begin{center}
{
\psset{unit=1.4cm}
   \tiny{ \begin{pspicture}(-4,-4)(2.5,2)
        \psset{linecolor=black}
        \cnode*(0,0){0.06}{a4}
        \cnode*(-2,-1){0.06}{a14}
        \cnode*(-1,-1){0.06}{a15}
        \cnode*(-2,-2){0.06}{b4}
        \cnode*(-1,-2){0.06}{b5}
        \cnode*(-2.5,-2.5){0.06}{b7}
        \cnode*(-1.5,-1.5){0.06}{b8}
        \cnode*(-0.5,-0.5){0.06}{b9}
        \psset{linecolor=black}
        \psline[linewidth=1pt](-5,1)(2,1)
        \psline[linewidth=1pt](-5,0)(2,0)
        \psline[linewidth=1pt](-5,-1)(2,-1)
        \psline[linewidth=1pt](-5,-2)(2,-2)
        \psline[linewidth=1pt](1,-4)(1,2)
        \psline[linewidth=1pt](0,-4)(0,2)
        \psline[linewidth=1pt](-1,-4)(-1,2)
        \psline[linewidth=1pt](-2,-4)(-2,2)
        \psline[linewidth=1pt](-5.2,0.2)(-0.8,-4.2)
        \psline[linewidth=1pt](-5.2,1.2)(0.2,-4.2)
        \psline[linewidth=1pt](-4.2,1.2)(1.2,-4.2)
        \psline[linewidth=1pt](-3.2,1.2)(1.2,-3.2)
        \psline[linewidth=1pt](-2.2,1.2)(1.2,-2.2)
        \psline[linewidth=1pt](-1.2,1.2)(1.2,-1.2)
        \psline[linewidth=1pt](-0.2,1.2)(1.2,-0.2)
        \psline[linewidth=1pt](0.8,1.2)(1.2,0.8)
        \uput[u](-1,1.9){\scriptsize{$z_1=-2$}}
        \uput[u](-2,1.9){\scriptsize{$z_1=-3$}}
        \uput[u](0,1.9){\scriptsize{$z_1=-1$}}
        \uput[u](1,1.9){\scriptsize{$z_1=0$}}
        \uput[r](1.9,-1){\scriptsize{$z_2=-2$}}
 \uput[r](1.9,0) {\scriptsize{$z_2=-1$}}
 \uput[r](1.9,1){\scriptsize{$z_2=0$}}
 \uput[r](1.9,-2){\scriptsize{$z_2=-3$}}

        \uput[r](-1.5,-1.2){\scriptsize(-2,-2)}
        \uput[r](-2.5,-1.2){\scriptsize(-3,-2)}
        \uput[r](-1.5,-2.2){\scriptsize(-2,-3)}
        \uput[r](-0.47,-0.2){\scriptsize(-1,-1)}
        \uput[d](-2.7,-2.5){\scriptsize(-3.5,-3.5)}
        \uput[d](-2.2,-1.9){\scriptsize(-3,-3)}
        \uput[u](-3,1){\scriptsize{$\vdots$}}
        \uput[u](-3,1.4){\scriptsize{$\vdots$}}
        \uput[u](-3,1.75){\scriptsize{$\cdot$}}
        \uput[u](-4,1){\scriptsize{$\vdots$}}
        \uput[u](-4,1.4){\scriptsize{$\vdots$}}
        \uput[u](-4,1.75){\scriptsize{$\cdot$}}
        \uput[r](1,-3){\scriptsize{$\cdots$}}
        \uput[r](1.4,-3){\scriptsize{$\cdots$}}
        \uput[r](1.75,-3){\scriptsize{$\cdot$}}

    \end{pspicture}}}
\end{center}

\subsection*{Acknowledgments}
	 Z. Bai was supported  by the National Natural Science Foundation of
	China (No. 12171344).

\end{document}